\documentclass[leqno]{amsart}
\usepackage{amsmath,amssymb,amsthm,amsfonts,graphicx,color}
\usepackage{amssymb}
\makeatletter
\def\theequation{\thesection.\@arabic\c@equation}
\makeatother

\long\def\red#1{{\color{red}#1}}
\long\def\blue#1{{\color{black}#1}}

\long\def\comment#1{\marginpar{\raggedright\small$\bullet$\ #1}}

\newcommand{\ttt}{\tilde }

\newcommand{\nn}{ {\nabla}  }
\newcommand{\inn}{  \quad\hbox{in } }
\newcommand{\onn}{  \quad\hbox{on } }

\newcommand{\pp}{ {\partial} }

\newcommand{\ww}{{\tt w}}
\newcommand{\www}{{\bf w}  }

\newcommand{\A}{\mathcal{A}}
\newcommand{\vp}{\varphi}

\newcommand{\cJ}{\mathcal{J}}
\newcommand{\tJ}{{\mathcal{J}_0}}
\newcommand{\J}{ \mathbb J}
\newcommand{\cB}{\mathcal{B}}
\newcommand{\cT}{\mathcal{T}}

\newcommand{\cN}{{\mathcal{N}}}

\newcommand{\NNN}{ \mathbf N}

\newcommand{\R} {\mathbb R}
\newcommand{\cuad}{{\sqcap\kern-.68em\sqcup}}

\newcommand{\dist}{{\rm dist}\, }
\newcommand{\foral}{\quad\mbox{for all}\quad}
\newcommand{\ve}{\varepsilon}

\newcommand{\be}{\begin{equation}}
\newcommand{\ee}{\end{equation}}

\newcommand{\la}{\lambda}

\newcommand{\uu }{{\bf u}}
\newcommand{\equ}[1]{{(\ref{#1})}}

\newcommand{\tth}{{\tt h}}

\newcommand{\tL}{{\tt L}}
\newcommand{\tB}{{\tt B}}
\newcommand{\cL}{{\mathcal  L}}
\newcommand{\cF}{{\mathcal  F}}

\newcommand{\tg}{{\tt g}}

\newcommand{\tf}{{\tt f}}
\newcommand{\NN}{{\tt N}}
\newcommand{\tS}{{\tt S}}

\renewcommand{\theequation}{\thesection.\arabic{equation}}

\newtheorem{lemma}{Lemma}[section]

\newtheorem{teo}{Theorem}[section]
\newtheorem{proposition}{Proposition}[section]
\newtheorem{corollary}{Corollary}[section]
\newtheorem{remark}{Remark}[section]
\newcommand{\bremark}{\begin{remark} \em}
\newcommand{\eremark}{\end{remark} }

\begin{document}

\title[De Giorgi conjecture in dimension $N\ge 9$]
{On De Giorgi conjecture in dimension $N\ge 9$}
\author{Manuel del Pino}
\address{\noindent M. del Pino - Departamento de
Ingenier\'{\i}a  Matem\'atica and CMM, Universidad de Chile,
Casilla 170 Correo 3, Santiago,
Chile.}
\email{delpino@dim.uchile.cl}

\author{Michal Kowalczyk}
\address{\noindent  M. Kowalczyk - Departamento de
Ingenier\'{\i}a  Matem\'atica and CMM, Universidad de Chile,
Casilla 170 Correo 3, Santiago,
Chile.}
\email{kowalczy@dim.uchile.cl}

\author{Juncheng Wei}
\address{\noindent J. Wei - Department of Mathematics, Chinese University of Hong Kong, Shatin, Hong Kong.
} \email{wei@math.cuhk.edu.hk}

\begin{abstract}
A celebrated conjecture due to De Giorgi states that any bounded solution of the  equation
$\Delta u + (1-u^2) u = 0 \  \hbox{in} \ \R^N $
with  $\pp_{y_N}u >0$ must be such that its level sets
$\{u=\la\}$ are all  hyperplanes, {\em \bf at least} for dimension $N\le 8$.  A counterexample for $N\ge 9$ has long been believed to exist.
Based on a minimal graph $\Gamma$ which is not a hyperplane, found by Bombieri, De Giorgi and Giusti in $\R^N$, $N\ge 9$, we prove  that for any
small $\alpha >0$ there is a bounded solution $u_\alpha(y)$ with $\pp_{y_N}u_\alpha >0$, which resembles
 $ \tanh \left ( \frac t{\sqrt{2}}\right ) $,
where $t=t(y)$ denotes a choice of signed distance to the blown-up minimal graph $\Gamma_\alpha := \alpha^{-1}\Gamma$.
This solution constitutes a counterexample to De Giorgi conjecture for $N\ge 9$.
\end{abstract}
\date{\today}\maketitle

\tableofcontents

\setcounter{equation}{0}
\section{Introduction}
This paper deals with entire solutions of the Allen-Cahn equation
\be \Delta u + (1-u^2) u = 0 \quad \hbox{in} \ \R^N .
\label{ac}\ee Equation \equ{ac} arises in the gradient theory of
phase transitions by Cahn-Hilliard and Allen-Cahn, in connection
with the energy functional in bounded domains $\Omega$ \be J_\ve
(u) = \frac \ve 2\int_\Omega|\nabla u|^2   + \frac 1{4\ve}
\int_\Omega (1-u^2)^2 ,\quad \ve>0 \label{energy}\ee whose
Euler-Lagrange equation corresponds precisely to a $\ve$-scaling
of equation \equ{ac} in the expanding domain $\ve^{-1}\Omega$. The
theory of $\Gamma$-convergence developed in the 70s and 80s,
showed a  deep connection between this problem and the theory of
minimal surfaces, see  Modica, Mortola, Kohn, Sternberg,
\cite{ks,modica1,modica2,modicamortola,sternberg}. In fact, it is
known that  for a family $u_\ve$  of local minimizers of $u_\ve$
with uniformly bounded energy must converge, up to subsequences,
in $L^1$-sense  to a function of the form $\chi_E - \chi_{E^c}$
where $\chi$ denotes characteristic function, and $\partial E$ has
minimal perimeter. Thus the interface between the stable {\em
phases} $u=1$ and $u=-1$, represented by the sets $[u_\ve =\la]$
with $|\la|<1$ approach a minimal hypersurface, see Caffarelli
and C\'ordoba \cite{caffarellicordoba,caffarellicordoba2}  (also R\"{o}ger and Tonegawa \cite{rt}) for stronger convergence and uniform regularity results on these level
surfaces.

\medskip
The above described connection led  E. De Giorgi \cite{dg} to
formulate in 1978 the following celebrated conjecture concerning
entire solutions of equation \equ{ac}.

\medskip
{\em De Giorgi's Conjecture: Let $u$ be a bounded solution of equation $\equ{ac}$ such that $\pp_{x_N}u >0$. Then the level sets
$\{u=\la\}$ are all  hyperplanes, {\em \bf at least} for dimension $N\le 8$. }

\medskip
Equivalently, $u$ depends on just one Euclidean variable so that it must have the form
\be
u(x) = \tanh \left (\,\frac{ x\cdot a -b }{\sqrt{2}}\,\right ),
\label{formu}\ee
for some $b\in \R$ and some $a$ with $|a|=1$ and $a_N>0$.  We observe that the function
$
w(t) =  \tanh \left (  t /{\sqrt{2}} \right )
$ is the unique solution of
 the one-dimensional problem,
$$ w'' + (1-w^2) w = 0, \quad w(0)= 0 \quad w(\pm \infty) = \pm 1\ . $$

\medskip
The monotonicity assumption in $u$ makes the scalings $u(x/\ve)$
local minimizers in suitable sense of $J_\ve$, moreover the level
sets of $u$ are all graphs. In this setting, De Giorgi's
conjecture is a natural, parallel statement to {\em Bernstein
theorem} for minimal graphs, which in its most general form, due
to Simons \cite{simons}, states that any minimal hypersurface  in
$\R^N$, which is also a graph of a function of $N-1$ variables,
must be a hyperplane if $N\le 8$. Strikingly, Bombieri, De Giorgi
and Giusti \cite{bdg} proved that this fact is false in dimension
$N\ge 9$. This  was most certainly the reason for the particle
{\em at least} in De Giorgi's statement.

Great advance in De Giorgi conjecture has been achieved in recent years, having been fully established  in dimensions $N=2$ by Ghoussoub and Gui \cite{gg} and for $N=3$ by Ambrosio and Cabr\'e  \cite{cabre}. More recently Savin  \cite{savin} established its validity for $4\le N\le 8$  under the following additional assumption (see \cite{aac} for a discussion of this condition)  
\begin{equation}
\label{gibbs}
\lim_{x_N \to \pm \infty} u(x^{'}, x_N)= \pm 1.
\end{equation}

Condition (\ref{gibbs}) is related to the  so-called Gibbons'
Conjecture:

\medskip
{\em Gibbons' Conjecture:  Let $u$ be a bounded solution of equation $\equ{ac}$ satisfying
\begin{equation}
\label{gibbs1} \lim_{x_N \to \pm \infty} u(x^{'}, x_N)= \pm 1, \
\mbox{uniformly in} \ x'.
\end{equation}
 Then the level sets $\{u=\la\}$ are all  hyperplanes.}

\medskip
Gibbons' Conjecture has been proved in all dimensions  with
different methods by Caffarelli and C\'ordoba \cite{caffarellicordoba2}, Farina \cite{Fa}, Barlow, Bass and Gui
\cite{BBG}, and Berestycki, Hamel, and Monneau \cite{BHM}.  In references \cite{caffarellicordoba2, BBG} it is proven that the conjecture is true for any solution that has one level set which is a globally Lipschitz graph. If the
uniformity in (\ref{gibbs1}) is dropped, a counterexample can be
built using the method by  Pacard and the authors in \cite{dkpw},
so that Savin's result is nearly optimal.

\medskip
A counterexample to De Giorgi's Conjecture in  dimension $N\ge 9$ has long been believed to exist, but the issue has remained elusive.  Partial progress in this direction has been achieved by  Jerison and Monneau \cite{jerison}
and by Cabr\'e and Terra \cite{cabreterra}. See the survey article by Farina and Valdinoci \cite{FV}.

\bigskip
In this paper we disprove De Giorgi's conjecture in dimension $N\ge 9$ by constructing a bounded solution of equation \equ{ac} which is monotone in one direction whose level sets are not hyperplanes.
The basis of our construction is a minimal graph different from a hyperplane  built by Bombieri, De Giorgi and Giusti \cite{bdg}.
In this work a solution of the zero mean curvature equation

\begin{equation}
\nabla \cdot \left (\, \frac{\nabla F}{\sqrt{1+ |\nabla F|^2 }}\right ) = 0\quad \hbox{in } \R^{N-1}
\label{0mc}\end{equation}
different from a linear affine function was built, provided that $N\ge 9$, in other words a non-trivial minimal graph in $\R^{N}$.
 Let us observe that if $F(x')$ solves equation \equ{0mc} then so does
$$
F_\alpha (x') := \alpha^{-1}F(\alpha x'), \quad \alpha>0 , $$ and
hence \be \Gamma_\alpha = \{ (x',x_N )\ / x'\in \R^{N-1},\ x_N=
F_\alpha ( x') \} \label{ga}\ee is a minimal graph in $\R^N$.

Our main result states as follows:
\begin{teo}\label{main}
Let $N\ge 9$.
There is a solution  $F$ to equation $\equ{0mc}$ which is not a linear affine function, such that for all $\alpha>0$ sufficiently small there exists a bounded solution $u_\alpha(y)$ of equation $\equ{ac}$ such that $u_\alpha (0) = 0$,
$$\partial_{y_N} u_\alpha (y)  >0 \foral y \in \R^N, $$
and
\be
|u_\alpha (y)|\  \rightarrow \ 1  \quad\hbox{ as }\    {\rm dist}\, (y, \Gamma_\alpha )\  \rightarrow \ + \infty .
\label{prop}\ee
uniformly in small $\alpha>0$, where $\Gamma_\alpha$ is given by $\equ{ga}$.
\end{teo}
Property \equ{prop} implies that the $0$ level set of $u_\alpha$ lies inside  the region ${\rm dist}\, (y, \Gamma_\alpha ) < R$ for some
fixed $R>0$ and all small $\alpha$, and hence it cannot  be a hyperplane.  Much more accurate information on the solution will
be drawn from the proof. The idea is simple. If $t(y)$ denotes a choice of signed distance to the graph $\Gamma_\alpha$ then, for
a small fixed number $\delta >0$, our solution looks like
$$ u_\alpha (y) \sim \tanh \left ( \frac t{\sqrt{2}}\right ) \quad\hbox{if } |t| < \frac \delta\alpha .
$$

\bigskip
As we have mentioned, a key ingredient of our proof is the construction of a non-trivial solution of equation \equ{0mc} carried out in \cite{bdg}.
 We shall derive  accurate information on its asymptotic behavior, which in particular will help us to find global estimates for its derivatives. This is a crucial step since
 the mean curvature operator yields in general poor gradient estimates. In addition we shall derive a similar theory of
existence and uniform estimates for the Jacobi operator around the minimal graph thus found. This work is carried out in sections \S 2 and \S 3.
In \S 4 a suitable first approximation for a solution is built, around which
 we linearize and carry out an infinite-dimensional Lyapunov Schmidt reduction, which eventually reduces the full problem to one of solving a nonlinear, nonlocal equation which involves as a main term the Jacobi operator of the minimal graph.
Schemes of this type have been successful in capturing solutions to singular perturbation elliptic problems in various settings, while in finding concentrating solutions on higher dimensional objects many difficulties arise. For the Allen-Cahn equation
in compact situations this has been done in the works del Pino, Kowalczyk and Wei \cite{dkw1}, Kowalczyk \cite{k}, Pacard and Ritore \cite{pacard}. In particular in \cite{pacard} solutions concentrating on a minimal submanifold of a compact Riemannian manifold are found through an argument that shares some similarities with the one used here.    In the non-compact settings for nonlinear Schr\"odinger equation solutions   have been constructed in del Pino, Kowalczyk and Wei \cite{dkw}, del Pino, Kowalczyk, Pacard and Wei \cite{dkpw},  and Malchiodi \cite{malchiodi3}.  See also  Malchiodi and Montenegro \cite{malchiodi1,malchiodi2}. We should emphasize here the importance of our earlier work \cite{dkpw} in the context of the present paper, and especially the idea of constructing solutions concentrating on a family of unbounded  sets, all coming from a suitably  rescaled basic set. While in \cite{dkpw} the concentration set was determined by solving a Toda system and the rescaling was the one appropriate to this system, here the concentration set is the minimal graph and the rescaling is the one that  leaves invariant the mean curvature operator.

\medskip
Let us observe that a counterexample to De Giorgi conjecture in $N=9$ induces one in $\R^N = \R^9\times \R^{N-9}$ for any $N> 9$,  by just extending the solution in $\R^9$ to the remaining variables in a constant manner.
For this reason, in what follows of this paper we shall assume $N=9$.

\section{Preliminaries and  an outline of the argument}\label{sec preliminaries}

In this section after introducing  the necessary notations we will outline our argument.

Let us recall that for any $\alpha >0$ the  minimal  surface $\Gamma_\alpha$ found in \cite{bdg} is given as the graph of  the function  $F_\alpha$ \blue{which is in a natural way obtained as scaling of the function $F$, a solution of (\ref{0mc}), namely $F_\alpha(x')=\alpha^{-1} F(\alpha x')$. In addition the function has $F$ the following properties} 
\begin{align}
\begin{aligned}
x_9=F(u,v), \quad u&=(x_1^2+\dots+x_4^2)^{1/2}, \quad v=(x_{5}^2+\dots+x_8^2)^{1/2},
\end{aligned}
\label{per 9}
\end{align}
and additionally satisfies
\begin{align}
F(u,v)=-F(v,u).
\label{per 10}
\end{align}
We observe that $\Gamma_\alpha$ is an embedded manifold and so it has a natural differential structure inherited from $\R^9$. The first and second covariant derivatives on $\Gamma_\alpha$ will be denoted, respectively by:
\begin{align*}
\nabla_{\Gamma_\alpha}\quad\mbox{and}\ \nabla_{\Gamma_\alpha}^2.
\end{align*}

In order to introduce the ansatz we fix an orientation  of  $\Gamma_\alpha$ and introduce the  Fermi coordinates in a neighborhood of $\Gamma_\alpha$:
\begin{align}
\begin{aligned}
x\mapsto (y,z),\ &\mbox{where}\ x=y+zn(y), \quad y=\big(y_1, \dots, y_8, F_\alpha(u',v')\big)\in \Gamma_\alpha,\\
& u'=\big((y_1)^2+\dots+(y_4)^2\big)^{1/2}, \quad v'=\big((y_{5})^2+\dots+(y_8)^2\big)^{1/2},
\end{aligned}
\label{pre 1}
\end{align}
and  $n(y)$ is the unit normal to $\Gamma_\alpha$ at $y$.  Let us denote:
\begin{align}
\Pi_{\R^8}(y)=(y_1, \dots, y_8), \quad r_\alpha(y)= \sqrt{1+\alpha^2|\Pi_{\R^8}(y)|^2}, \quad y\in \Gamma_\alpha.
\label{pre def r}
\end{align}
We notice that the Fermi coordinates are well  defined in
\blue{the following  neighborhood of $\Gamma_\alpha$:
\begin{align}
{\mathcal U}_{\theta_0}=\{x\in \R^9\mid x=y+zn(y), y\in \Gamma_\alpha,  |z|<\frac{\theta_0r_\alpha}{\alpha}\},
\label{def ndelta1}
\end{align}
where $\theta_0>0$ is a fixed small number independent of $\alpha$. This fact will be proven later 
 (see section \ref{sec grad est}, formula  (\ref{coord 2}) and the argument that follows).}
One advantage of working with the Fermi coordinates  is the fact that the Laplacian in $\R^9$ has a particularly simple expression in these coordinates. \blue{To explain this let us denote  $H_{\Gamma_{\alpha, z}}$ the mean curvature of the surface $\Gamma_{\alpha, z}$ obtained from $\Gamma_{\alpha}$  after translating it by $z$ in the direction of the normal. Then we have
\begin{align}
\Delta&=\Delta_{\Gamma_{\alpha,z}}+\partial_z^2-H_{\Gamma_{\alpha, z}}\partial_z,
\label{prelim 60}
\end{align}
where $\Delta_{\Gamma_{\alpha,z}}$ is the Laplace-Beltrami operator on $\Gamma_{\alpha, z}$.}

Intuitively, near $\Gamma_\alpha$ the solution of (\ref{ac}) we are after should resemble a function of the form
\begin{align}
u(x)\sim  w(z-h_\alpha), \ \mbox{where}\ h_\alpha=h_\alpha(y), \quad y\in \Gamma_\alpha,
\label{pre 1a}
\end{align}
and  $w(t)=\tanh\Big(\frac{t}{\sqrt{2}}\Big)$ is the heteroclinic solution of one dimensional version of (\ref{ac}). \blue{The introduction of the  new, uknown function $h_\alpha$ reflects the fact that while we expect the profile of the solution in the direction transversal to $\Gamma_\alpha$  to be similar to the one dimensional heteroclinic, however its zero level set is not expected to coincide  with $\Gamma_\alpha$ but to be its small perturbation.} 
Later on we will see that this perturbation, represented by $h_\alpha$ is actually a small quantity,  of order  $o(1)$,  as $\alpha\to 0$. 
\blue{Finally we observe that we can identify  $h_\alpha$ as a function on $\Gamma_\alpha$ with a function defined on the original surface $\Gamma$ via the formula:
\begin{align*}
h(\alpha y)=h_\alpha({y}), \quad y\in \Gamma_\alpha,
\end{align*} 
where $h\colon \Gamma\to \R$ is a given function. This identification will be very useful in the sequel and used without further reference whenever it does not cause confusion. }

Now, let $\chi$ 
be a smooth cutoff function such that $\chi(s)=1$, $-1\leq s\leq 1$, and
$\chi(s)=0$, $|s|>2$.
Given the  hypersurface $\Gamma_\alpha$ and the local Fermi coordinates $x\mapsto (y,z)\in \blue{{\mathcal U}}_{\theta_0}$,  as above we set:
\begin{align}
{\tt w}(x)=\begin{cases} \chi\big(\frac{4\alpha z}{\theta_0r_\alpha})\big(w(z-h_\alpha)+1\big)-1, &\quad z<0,\\
\chi\big(\frac{4\alpha z}{\theta_0r_\alpha})\big(w(z-h_\alpha)-1\big)+1, &\quad 0\leq z,
\end{cases}.
\label{pre 7}
\end{align}
Notice that  function $\ww$ is well defined  in an expanding (as a function of $r_\alpha$) neighborhood ${\mathcal U}_{\theta_0}$ of $\Gamma_\alpha$. \blue{To complete this definition we need to extend  
$\ww(x)$ as a smooth function to the whole $\R^9$. Noting that $\R^9\setminus{\mathcal U}_{\theta_0}$ consists of two disjoint components (since $\Gamma_\alpha$ is a graph) we define  $\ww$ to be the smooth function which satisfies (\ref{pre 7}) and takes only values $\pm 1$ in  $\R^9\setminus{\mathcal U}_{\theta_0}$. We will use the same symbol $\ww$ for the extended function.}

We will look for the solution of (\ref{ac}) in the form:
\begin{align}
u_\alpha=\ww(x)+\phi(x).
\label{pre 7a}
\end{align}
Substituting in (\ref{ac}) we get for  the function $\phi$
\begin{align}
\Delta\phi+f'(\ww)\phi=S[\ww]+N(\phi),
\label{pre 8}
\end{align}
where
\begin{align}
S[\ww]=-\Delta\ww-f(\ww), \quad N(\phi)=-[f(\ww+\phi)-f(\ww)-f'(\ww)\phi], \quad f(\ww)=\ww(1-\ww^2).
\label{pre 5}
\end{align}
For future references let us denote as well:
\begin{align}
\cL(\phi)=\Delta\phi+f'(\ww)\phi.
\label{def l1}
\end{align}
The remaining part of this paper is devoted to solving (\ref{pre 8}) and in particular to showing that
$\cL$ has a uniformly (in small $\alpha$) bounded inverse in a suitable function space.  To explain the theory we will need let us observe that locally, that is near $\Gamma_\alpha$, for $\alpha$ small $\cL$ resembles the following operator
\begin{align*}
L(\phi)=\Delta_{\Gamma_\alpha}\phi+\partial_{zz}\phi+f'(w)\phi,
\end{align*}
where $\Delta_{\Gamma_\alpha}$ is the Laplace-Beltrami operator on $\Gamma_\alpha$.  \blue{This follows from the fact that $\ww(x)\sim w(z-h_\alpha)$, and also (\ref{prelim 60}) since $\Delta_{\Gamma_{\alpha, z}}\sim \Delta_{\Gamma_\alpha}$ and $H_{\Gamma_{\alpha, z}}\sim H_{\Gamma_\alpha}$.} Immediately we notice that function $w_z(z-h_\alpha)$ \blue{satisfies,
\begin{align*}
L\big(w(z-h_\alpha)\big)&=-\Delta_{\Gamma_\alpha}h_\alpha w_{z}(z-h_\alpha)+|\nabla_{\Gamma_\alpha}h_\alpha|^2 w_{zz}(z-h_\alpha)\\
&=o(1),
\end{align*}}
and consequently we do not expect  to find a uniformly bounded  inverse of  $\cL$ without introducing  some restriction on its range. In this paper we deal with this difficulty using a version of infinite dimensional Lyapunov-Schmidt  reduction (c.f \cite{dkw,dkw1,dkpw}). The essence of this method is to introduce a function $c(y)$,
$y\in \Gamma_\alpha$ and consider the following problem:
\begin{align}
\begin{aligned}
L(\phi)&=S[\ww]+N(\phi)+c(y) w_z(z-h_\alpha),\quad (y,z)\in \Gamma_\alpha\times\R,\\
\int_\R\phi(y,z)w_z(z-h_\alpha)\,dz&=0, \quad \forall y\in \Gamma_\alpha.
\end{aligned}
\label{pre 9}
\end{align}
Recall that the ansatz $\ww$ depends on, still undetermined, function $h_\alpha$. Solving (\ref{pre 9})  for a given $h_\alpha$ and then adjusting $h_\alpha$ in such a way that 
\begin{align}
c(y;h_\alpha)=0, \quad \forall y\in \Gamma_\alpha,
\label{pre 10}
\end{align}
we get the solution of (\ref{pre 8}). 
Actually, the following extra steps are needed to solve (\ref{pre 8}): (1) {\it gluing} the local (inner) solution of (\ref{pre 9}) and a suitable  outer solution; (2) a fixed point argument;  but  the main point of the method is to solve  (\ref{pre 9}).  Equation (\ref{pre 10}), called here the {\it reduced problem}, is a nonlocal PDE for $h_\alpha$ and its solvability is in itself a nontrivial step extensively treated in this paper (sections \ref{sec superl} and \ref{sec sred}).

\subsection{Improvement of the initial approximation}\label{sec imp}

Let us consider more carefully $S[\ww]$.  Near the surface $\Gamma_\alpha$ (say in the set 
${\mathcal U}_{\theta_0/4}$), where the Fermi coordinates are well defined and where $\ww=w$ we have, using (\ref{prelim 60}):
\begin{align}
\begin{aligned}
S[\ww]&=-\Delta_{\Gamma_{\alpha,z}}\ww-\partial_z^2\ww+H_{\Gamma_{\alpha, z}}\partial_z\ww-f(\ww)\\
&=-\Delta_{\Gamma_{\alpha,z}}w+H_{\Gamma_{\alpha, z}}\partial_z w.
\end{aligned}
\label{imp 1}
\end{align}
As we will see later the first term above is of relatively small size. The second term  is of the leading order and it has to be treated separately.
To see this we \blue{use Taylor expansion of $H_{\Gamma_{\alpha,z}}$ around $z=0$:
\begin{align}
H_{\Gamma_{\alpha,z}}(y)=z|A_{\Gamma_\alpha}(y)|^2+z^2 {\mathcal R}_{\alpha}(y,z), \quad y\in \Gamma_\alpha,
\label{blue 1}
\end{align}
where $|A_{\Gamma_\alpha}(y)|$ is the norm of the second fundamental form on $\Gamma_\alpha$ and ${\mathcal R}_\alpha$ is the remainder in the Taylor's expansion.  Here, for future references, we observe that 
\begin{align}
H_{\Gamma_{\alpha,z}}=\sum_{i=1}^8\frac{\kappa_i}{1-z\kappa_i}, 
\label{hgammaz}
\end{align}
where $\kappa_i$ denote the principal curvatures of $\Gamma_\alpha$. From this (\ref{blue 1}) follows immediately.
Next, using that $\ww(x)\sim w(z-h_\alpha)$ in ${\mathcal U}_{\theta_0}$, we get setting $\bar z=z-h_\alpha$:}
\begin{align}
\begin{aligned}
H_{\Gamma_{\alpha, z}}\partial_z w(z-h_\alpha)
&\sim (\bar z+h_\alpha)|A_{\Gamma_\alpha}|^2\partial_{\bar z}w
\\
&\sim \bar z|A_{\Gamma_\alpha}|^2\partial_{\bar z}w.
\end{aligned}
\label{imp 2}
\end{align}
\blue{This formal calculation is justified be the fact that 
\begin{align}
|A_{\Gamma_\alpha}(y)|^2\sim r_\alpha^{-2}(y)
\label{blue 2}
\end{align}
and we expect that all terms carrying $h_\alpha$ are, besides being small in $\alpha$ are in addition decaying in $r_\alpha$}. We observe that:
\begin{align}
\int_\R\bar z|A_{\Gamma_\alpha}|^2w_{\bar z}^2\,d\bar z=0,
\label{imp 3}
\end{align}
hence there exists a unique solution $w_1$ of the problem
\begin{align}
\partial^2_{\bar z}w_1+f'(w)w_1=\bar z|A_{\Gamma_\alpha}|^2\partial_{\bar z}w, \quad \bar z\in \R,
\label{imp 4}
\end{align}
which is explicitly given by:
\begin{align}
w_1(\bar z)=-|A_{\Gamma_\alpha}|^2 w'(\bar z)\int_0^{\bar z}\frac{d\sigma}{(w'(\sigma))^2}\int_{-\infty}^\sigma \eta(w'(\eta))^2\,d\eta.
\label{imp 5}
\end{align}
We define now 
\begin{align}
{\ww}_1(x)=\chi\big(\frac{4\alpha z}{\theta_0r_\alpha})w_1(\bar z).
\label{imp 6}
\end{align}
Function $\ww_1$ gives an improvement of the initial approximation, and in some sense it represents the  first term in the asymptotic expansion of $\phi$ in (\ref{pre 8}). \blue{Let us explain this. Going back to (\ref{imp 1})--(\ref{blue 1}) we see, again formally, that the error of the new approximation is:
\begin{align}
\begin{aligned}
S[\ww+\ww_1]&=S[\ww]+S[\ww_1]-[f(\ww+\ww_1)-f(\ww)-f(\ww_1)]\\
&\sim S[\ww]-\partial^2_{\bar z}w_1-f'(w)w_1\\
&\sim (\Delta_{\Gamma_\alpha}h_\alpha+|A_{\Gamma_\alpha}|^2 h_\alpha)w_{\bar z}+\bar z^2{\mathcal R}_{\alpha}(y,0)w_{\bar z}.
\end{aligned}
\label{blue 3}
\end{align}
In  (\ref{blue 3}) we have neglected  those terms that are  decaying fast in $r_\alpha$, which means here faster than $r_\alpha^{-3}$. This agrees with   our expectation that  we should have ${\mathcal R}_{\alpha}(y,0)\sim r_\alpha^{-3}$. In turn, this will imply that, if we write $u_\alpha=\ww+\ww_1+\phi$, then this "new" $\phi\sim r_\alpha^{-3}$, while $\ww_1\sim r_\alpha^{-2}$, which apparently is a technical, but rather crucial point in our analysis.

Now, let us recall that $w_{\bar z}$ is an element of the "approximate" kernel of the operator $L$ appearing in  (\ref{pre 9}) and thus the problem of bounded solvability of (\ref{pre 9}) depends, roughly speaking,  on the orthogonality  of the right hand side of (\ref{pre 9}) and the function $w_{\bar z}$. This in turn be achieved if 
\begin{align}
\int_\R[(\Delta_{\Gamma_\alpha}+|A_{\Gamma_\alpha}|^2) h_\alpha+\bar z^2{\mathcal R}_{\alpha}(y,0)]w^2_{\bar z}\,d\bar z\approx 0,
\label{blue 4}
\end{align}
which is equivalent to  (\ref{pre 10}). 
Let us now summarize what is needed  in order to reduce the full nonlinear problem  (\ref{pre 9}) to the  {\it reduced problem} (\ref{blue 4}) and solve (\ref{ac}) at the end. 
\begin{itemize}
\item[(i)] The linearized operator $L$ has a bounded inverse in the space of functions satisfying the orthogonality condition in (\ref{pre 9}).
\item[(ii)]
Problem of the form:
\begin{align}
(\Delta_{\Gamma_\alpha}+|A_{\Gamma_\alpha}|^2) h_\alpha=f_\alpha\sim O(r_\alpha^{-3}),
\label{blue 5}
\end{align}
given on the manifold $\Gamma_\alpha$ can be solved in a suitable function space whose norm takes into account the decay of its elements in $r_\alpha$. 
\end{itemize}
The rest of this paper is devoted to addressing (i) and (ii) above. In fact, since these two issues are not quite independent it is convenient to first treat (ii) and then later deal with (i). }
One of the main steps required to carry out the plan outlined above is  a refinement of the existence result of Bombieri-De Giorgi and Giusti \cite{bdg} (section 3). In addition we need to find precise  decay estimates for the minimal graph $F_\alpha$ and its derivatives (up to order $3$)  which amounts to a refinement of a  result of Simon \cite{simon 1} (section 4).

\setcounter{equation}{0}
\section{The minimal surface equation}
In this section we will consider only one fixed minimal graph, denoted here by $F$, since as we have pointed out $\Gamma_\alpha$ is obtained as a graph of $F_\alpha(x')=\alpha^{-1}F(\alpha x')$, $x'\in \R^8$.
Thus, we consider the mean
curvature equation in $\mathbb{R}^8$
\begin{equation}\label{mc1a}
  \sum_{i=1}^8  \partial_{x_i} \left (\, \frac{ F_{x_i} }{\sqrt{1+ |\nabla F|^2}}\right ) = 0\quad \hbox{in } \R^8 .
\end{equation}
Bombieri, De Giorgi and Giusti \cite{bdg} found that this equation possesses a non-constant  entire solution
if $N\ge 8$, therefore a minimal graph different from a hyperplane exists in  dimensions 9 or higher.
The solution found in \cite{bdg} enjoys some simple symmetries and also is a function of the variables $(u,v)$ defined  above.
It is straightforward to check that
the mean curvature operator  written in terms of these variables  becomes
\be
H[F]:= \frac{1}{ (uv)^3} \nabla \cdot \left ( \frac{ (uv)^3 \nabla F }{\sqrt{1+|\nabla F|^2}} \right),\quad \nabla F=(F_u, F_v),
\label{56.1}\ee
while the equation (\ref{mc1a}) reads
\be
\frac 1{u^3}\, \partial_u \left ( \frac { u^3 F_u }{ \sqrt{ 1+ |\nabla F|^2 } }\right )
\ +\
\frac 1{v^3}\, \partial_v\left ( \frac { v^3 F_v }{ \sqrt{ 1+ |\nabla F|^2}} \right )\ = \ 0.
\label{eqmc}\ee

\medskip
Since the solution in \cite{bdg} satisfies
$$F(u, v) = -F(v , u)\quad\hbox{if }u <  v,$$
and  in particular $F= 0$ on the cone  $u = u$, therefore it is sufficient to consider (\ref{eqmc}) in the region (see figure \ref{fig:two}).
\begin{align}
T=\{(u,v)\mid u>0, v>0, u<v\}.
\label{def sector}
\end{align}
\begin{figure}
\vspace{0.75cm} 
\resizebox{8.0cm}{8.0cm}
{\includegraphics[angle=0]{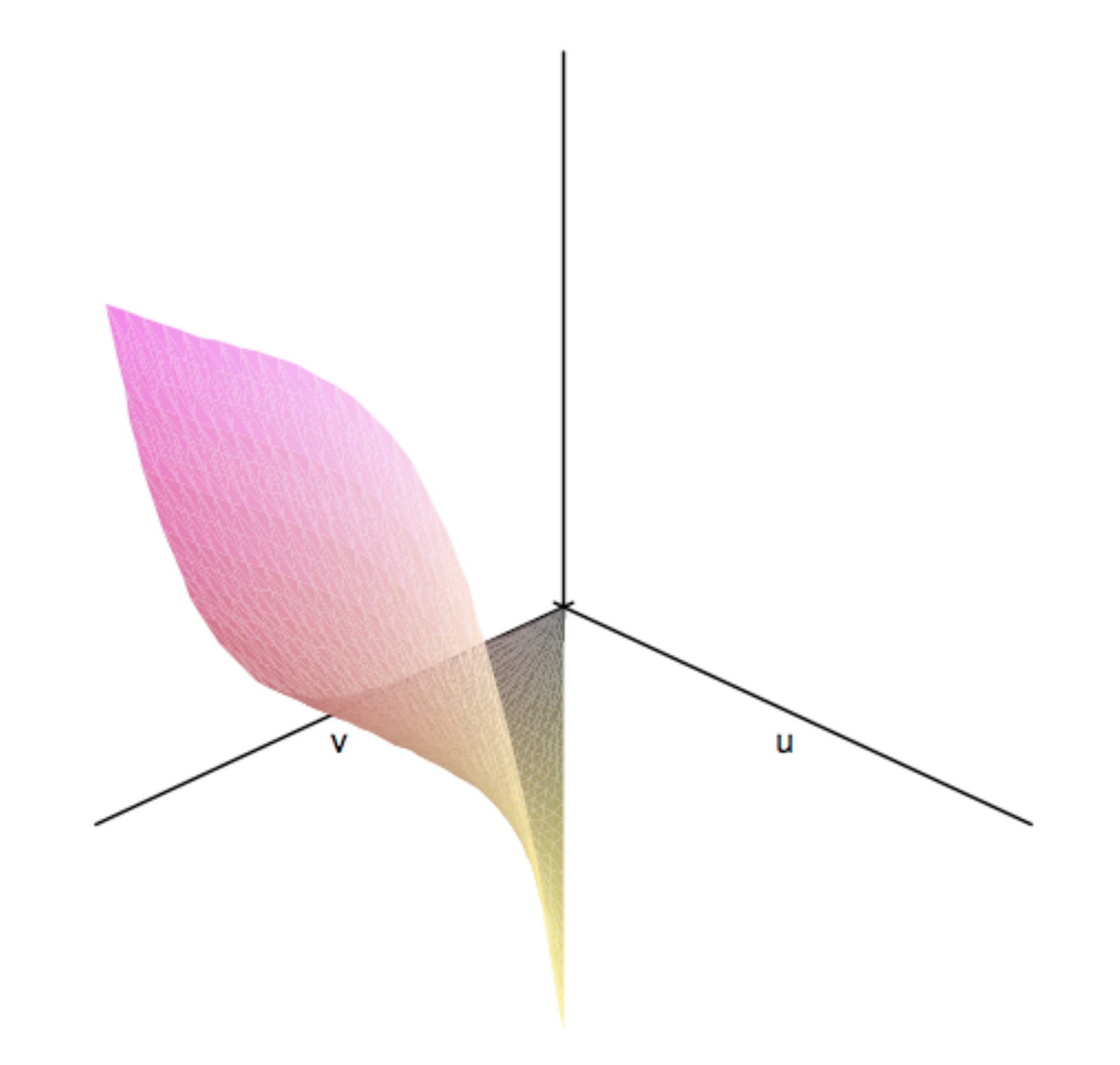}}
\put(-232,58){$\theta=\frac{\pi}4$}
\put(-109,33){$\theta=\frac{\pi}2$}
\put(-158,38){$T$}
\caption{\small{Schematic view of the function $F(u,v)$ representing  $\Gamma$ in the sector $T=\{u>0,v>0,u<v\}$.}} \label{fig:two}
\end{figure}

Let us introduce polar coordinates in $T$, setting
$$
u = r\cos\theta, \quad v = r\sin \theta,\quad \theta\in \big(\frac{\pi}{4},\frac \pi 2\big),
$$
so that $r= |\uu|$, $\uu=(u,v)$. First, we will  show that the solution in \cite{bdg} can be described at main order as
$$
F(r,\theta) \sim r^3 g(\theta ) \quad\hbox{as } r\to \infty ,
$$
where $g$ satisfies $g(\frac \pi 4 )= 0$, $g_\theta(\frac \pi 4) > 0$, $g_\theta(\frac \pi 2) = 0$. Intuitively  $g(\theta)$ resembles
$-\cos 2\theta$, $\pi/4\leq \theta\leq \pi/2$. In the sequel we will denote $F_0=r^3g(\theta)$.

Second, we introduce  coordinates $(s,t)$ in $T$ which are adapted to $F_0$ and play a fundamental role in this paper.

\subsection{Equation for $g$}
Since we  expect
$$
F(u,v)\sim F_0(u,v)=r^3 g(\theta), \quad r \gg 1,
$$
therefore it is reasonable to require that $F_0$ should be a solution of
\begin{align*}
\nabla\cdot \Big(\frac{\nabla F_0}{|\nabla F_0|}\Big)=0, \quad \nabla =(\partial_u, \partial_v).
\end{align*}
Assuming that $F_0=r^3g(\theta)\geq 0$ in the sector $T$ we get the following equation for
the positive function $g(\theta)$
\begin{equation}
\label{9}
\frac{ 21 \sin^3 (2\theta) g}{\sqrt{ 9 g^2 + g_\theta^2}} + \Bigl( \frac{\sin^3 (2\theta) g_\theta}{ \sqrt{  9 g^2 + g_\theta^2}}\Bigr)_\theta=0,
\quad \theta\in (\frac \pi 4, \frac\pi 2),
\end{equation}
with the boundary conditions
\begin{equation}
\label{10}
g (\frac{\pi}{4})=0,\quad g_\theta(\frac \pi 2 )=0.
\end{equation}
The boundary conditions (\ref{10}) follow from the symmetries of $F$.

Let us observe that if $g(\theta)$ is a solution of (\ref{9}) then so is $C g(\theta)$, for any constant $C$. The following lemma proves the existence of solutions to (\ref{9}).

\begin{lemma}\label{lemma min 1}
Problem (\ref{9}) has a solution such that:
\begin{align}
g(\theta) &\geq 0,  \quad g_{\theta\theta}(\theta)\leq 0,\quad 
g_\theta(\theta)\geq 0,  \label{min 4g}
\end{align}
and the last inequality is strict for $\theta\in [\frac{\pi}{4}, \frac{\pi}{2})$.
\end{lemma}

\proof{}
If $g$ is a solution to (\ref{9}) then function
\begin{align*}
\psi(\theta)=\frac{g_\theta(\theta)}{g(\theta)},\quad g(\theta)\neq 0,
\end{align*}
satisfies the following equation:
\begin{align}
9\psi'+(9+\psi^2)[21+6\cot(2\theta)\psi]=0.
\label{min 4a}
\end{align}
Our strategy is to solve (\ref{min 4a}) first and then find the function  $g$.
To this end we will look for a solution of (\ref{min 4a}) in the interval
$I=(\pi/4,\pi/2)$ with
\begin{align}
\psi(\pi/2)=0.
\label{min 4b}
\end{align}
\blue{In order to define the function $g$  we also need  $\psi$ to be defined and positive in the whole interval $(\frac{\pi}{4},\frac{\pi}{2}]$ and $\lim_{\theta\to \frac{\pi}{4}^+}\psi(\theta)=+\infty$. Let $(\theta^*,\frac{\pi}{2}]$, $\frac{\pi}{4}\leq \theta^*$ be the maximal interval for which the solution of (\ref{min 4a}) exists.

We set $\psi_+(\theta)=-11\tan(2\theta)$. Then we have
\begin{align*}
9\psi_+'+(9+\psi_+^2)[21+6\cot(2\theta)\psi_+]& < 0, \quad \theta\in (\frac{\pi}{4},\frac{\pi}{2}],\\
\psi_+(\frac{\pi}{2})=0=\psi(\frac{\pi}{2})&,\quad \psi_+'(\frac{\pi}{2})=-22<-21=\psi'(\frac{\pi}{2}).
\end{align*}
}
Substituting $\psi_-(\theta)=-2\tan(2\theta)$ for $\psi$ in  (\ref{min 4a}) we get:
\begin{align}
9\psi_-'+(9+\psi_-^2)[21+6\cot(2\theta)\psi_-] >  0.
\label{min 4d}
\end{align}
We have $\psi(\pi/2)=\psi_-(\pi/2)=0$ and, from (\ref{min 4a}),
\begin{align*}
\psi'(\pi/2)=-21<-4= \psi_-'(\pi/2).
\end{align*}
\blue{From this we get that the maximal solution of (\ref{min 4a}) satisfies:
\begin{align}
\psi_+(\theta)=-11\tan(2\theta)>\psi(\theta)\geq \psi_-(\theta)=-2\tan(2\theta)\geq 0, \quad \theta\in (\theta^*,\pi/2),
\label{min 4e}
\end{align}
and that $\theta^*= \frac{\pi}{4}$. }
Let us now define
\begin{align}
g(\theta)=\exp\Big\{-\int_\theta^{\pi/2}\psi(t)\,dt\Big\},
\label{min 4f}
\end{align}
where $\psi$ is the unique solution of  (\ref{min 4a})--(\ref{min 4b}).  Clearly we have
$g_\theta(\pi/2)=0$ and from (\ref{min 4e}) it follows $g(\pi/4)=0$.  Thus $g$ defined in (\ref{min 4f}) is a solution of (\ref{9})--(\ref{10}).

\blue{We have  $g_\theta>0$ in $(\frac{\pi}{4},\frac{\pi}2)$, since $g_\theta=g\psi$. 
To show that  $g_\theta(\frac{\pi}{4})>0$ we will improve the upper bound on $\psi$.  Let us define:
\begin{align*}
\psi_1 =-2\tan(2\theta)+\tilde\psi, \quad{\mbox{where}}\ \tilde\psi=A\big(-\tan(2\theta)\big)^{\eta},
\end{align*}
and $\frac{2}{3}<\eta<1$, $A>1$ are to be chosen. 
Direct calculations give:
\begin{align*}
&9\psi_1'+(9+\psi_1^2)[21+6\psi_1\cot(2\theta)]=9\tilde\psi'\cos^2(2\theta)+45\cos^2(2\theta)\\
&\quad +6\tilde\psi\cot(2\theta)[4+5\cos^2(2\theta)]+36\tilde\psi\sin(2\theta)(-\cos(2\theta))\\
&\quad+9\tilde\psi^2\cos^2(2\theta)+6\tilde\psi^2\cot(2\theta)[4\sin(2\theta)(-\cos(2\theta))+\tilde\psi\cos^2(2\theta)].
\end{align*}
Using the definition of $\tilde\psi$, after some calculation we find that the last expression is negative for $\theta\in (\frac{\pi}{4},\frac{\pi}{2})$ when 
\begin{align*}
0&>-18A\eta+45(-\tan(2\theta))^{1-\eta}\cos^2(2\theta)-6A[4+5\cos^2(2\theta)]+36A\sin^2(2\theta)\\&\quad-15A^2(-\tan(2\theta))^{1+\eta}\cos^2(2\theta)-6A^3(-\tan(2\theta))^{1-2\eta}\sin(2\theta)(-\cos(2\theta)),
\end{align*}
which can be achieved if $\frac{2}{3}<\eta<1$ and $A$ is chosen sufficiently large. Since $\eta<1$ it follows that 
\begin{align*}
\psi(\theta)\leq \psi_1(\theta), \quad \theta\in(\frac{\pi}{4},\frac{\pi}{2}),
\end{align*}
hence, for certain constant $C>0$, 
\begin{align}
-C\cos(2\theta)\leq g(\theta)\leq -\cos(2\theta), \quad \theta\in [\frac{\pi}{4},\frac{\pi}{2}].
\label{min 4h}
\end{align}
In fact the inequalities in (\ref{min 4g}) are  strict for $\theta\in (\frac{\pi}{4}, \frac{\pi}{2})$. It follows in addition that:
\begin{align*}
g_\theta(\theta)\geq C\sin(2\theta), \quad \theta\in[\frac{\pi}{4}, \frac{\pi}{2}].
\end{align*}
This shows in particular $g_\theta>0$ in $[\frac{\pi}{4}, \frac{\pi}{2})$. The remaining estimate for $g_{\theta\theta}$ follows from the second order equation for $g$.}
\qed


Given the function  $g$  let us define:
\begin{equation}
\label{11}
\cos  \phi = \frac{ 3 g}{ \sqrt{ 9 g^2 +g_\theta^2}}, \quad \sin  \phi = \frac{ g_\theta }{ \sqrt{ 9 g^2 +g_\theta^2}}.
\end{equation}
We see from Lemma \ref{lemma min 1} that $\phi$ \blue{satisfies}:
\begin{equation}
\label{12}
 \phi'+ 7 + 6 \cot (2 \theta) \tan \phi=0, \qquad \phi(\frac \pi 4 )=\frac{\pi}{2},\quad \phi (\frac{\pi}{2})=0.
\end{equation}
We  need the following lemma:
\begin{lemma}
It holds
\begin{equation}
\label{fact1}
\phi' (\frac{\pi}{4})=-3,\quad   \phi' (\frac{\pi}{2})=-\frac{7}{4}, \quad \phi' (\theta) >-3 \ \mbox{for} \ \theta \in (\frac{\pi}{4}, \frac{\pi}{2}).
\end{equation}
\end{lemma}

\proof{}
To prove the first identity we observe that $\tan\phi=\frac{1}{3}\psi$ which after differentiation yields:
\begin{align}
\phi'=\frac{1}{3}\psi'\cos^2\phi=-\frac{1}{3}[21+6\cot(2\theta)\psi]\geq -3,
\label{12a}\end{align}
since $\psi(\theta)\geq  -2\tan(2\theta)$. Now considering (\ref{12}) we see that when $\theta\to\pi/4^+$ we can have $\phi'(\pi/4^+)=-3$ or $\phi'(\pi/4^+)=-4$. From (\ref{12a}) 
we get the required formula.

The second identity  follows from simple analysis near  $ \theta=\frac{\pi}{2}$.

To prove the last estimate, we suppose that there exists a point $\theta_1 \in (\frac{\pi}{4}, \frac{\pi}{2}) $ such that $ \phi' (\theta_1)=-3$. We claim that $ \phi'' (\theta_1) <0$. This gives a contradiction. (We may take $\theta_1$ to be the point closest to $\frac{\pi}{2}$. Then necessarily $ \phi' (\theta_1) \geq 0$.)  In fact, from (\ref{12}), we deduce that
$$
 2\sin (2\theta_1) \cos \phi + 3 \cos (2 \theta_1) \sin \phi=0,
$$
which is equivalent to
\begin{equation}
 5 \sin (2 \theta_1+\phi) = \sin (2\theta_1 -\phi).
\end{equation}
Note that $ 2 \theta -\phi \in (0, \pi)$ and hence $ 0 <2\theta-\phi <2\theta +\phi <\pi$. Now we compute
\begin{align*}
\phi{''} (\theta_1) &= \frac{6}{ \sin^2 \theta_1 \cos^2 \phi} (\sin2 \phi -\frac{1}{2} \sin 4 \theta_1 \phi{'})
\\
&= \frac{6}{ \sin^2 \theta_1 \cos^2 \phi} \sin (2\theta_1-\phi) \cos (2\theta_1) \cos \phi <0,
\end{align*}
which completes the proof.

\qed

\subsection{A new system of coordinates}\label{st}


In this section we will introduce a system of coordinates  in  the sector $T$  (see (\ref{def sector}))which depends on the function  $F_0$ defined above.  The idea is that the coordinate lines on the ($2$ dimensional) surface  given by the graph of $F_0$ are orthogonal. As we will see this property is extremely useful in further developments.
\begin{lemma}\label{lemma ts}
There exists a diffeomorphism $\Phi:Q\to   T$, where $Q=\{(t,s)\mid t>0, s>0\}$ such that
$\Phi(t,s)= \uu(t,s)= \big(u(t,s),v(t,s)\big)$
and $\uu$ satisfies the coupled system of differential equations
\be
\frac{\partial \uu}{\partial t} \,= \, \frac {\nabla F_0} {|\nabla F_0|^2} ,\qquad \frac {\partial \uu}{\partial s} \,= \, \frac 1{(uv)^3}  \frac { \nabla F_0^\perp }{|\nabla F_0|},
\label{coord1}\ee
where we denote
$$
\nabla F =(F_u, F_v), \qquad \nabla F^\perp  =( F_v, -F_u) .
$$
Moreover $\Phi$ maps  $(t=0,s)$ onto the line $u=v$ and $(t,s=0)$ onto $(u=0,v)$.
\end{lemma}
\proof{}
Introducing polar coordinates
$$ u = r\cos\theta, \quad v = r\sin\theta,$$
and using  \equ{coord1} we find:
\begin{align}
\begin{aligned}
&\left \{ \begin{matrix} \frac{\partial r}{\partial t} &=&\frac{F_{0,r}}{|\nabla F_0|^2}&=& \frac{3g}{r^2(9 g^2 + g_\theta^2) } \\[2mm]
                        \frac{\partial \theta }{\partial t} &=&\frac{F_{0,r}}{|\nabla F_0|^2}&=& \frac{g_\theta }{r^3(9 g^2 + g_\theta^2) } \\
         \end{matrix} \right.  , \\
&\left \{ \begin{matrix} \frac{\partial r}{\partial s}&=&\frac{8F_{0\theta}}{r^7\sin^3(2\theta)|\nabla F_0|} &=& \frac{8g_\theta }{r^6\sin^3(2\theta)\sqrt{9 g^2 + g_\theta^2} } \\[2mm]
                        \frac{\partial \theta }{\partial s}&=&\frac{-8F_{0r}}{r^7\sin^3(2\theta)|\nabla F_0|} &=& -\frac{24g}{r^7\sin^3(2\theta)\sqrt{9 g^2 + g_\theta^2} }\\
         \end{matrix}\right.\ .
\end{aligned}
\label{coord2}
\end{align}
Using the formal relations
$$
\left [ \begin{matrix} t_r & t_\theta \\ s_r & s_\theta \end{matrix} \right ] \, \left [ \begin{matrix} r_t & r_s  \\ \theta_t & \theta_s
\end{matrix} \right ]\ =  \ \left [ \begin{matrix} 1 &  0 \\ 0 &  1 \end{matrix} \right ]
$$
we  arrive in particular at the equations for $s$
\begin{align*}
3gs_r  + \frac{g_\theta} r  s_\theta &= 0, \\
 \frac{8g_\theta s_r} {r^6 \sin^3 2\theta \sqrt{9g^2 +g_\theta^2} } -
 \frac{24 g s_\theta} {r^7 \sin^3 2\theta \sqrt{9g^2 +g_\theta^2} }& =1,
\end{align*}
or
$$
\left \{ \begin{matrix} \frac{\partial s}{\partial \theta } &=& -\frac{3r^7 \sin^3 2\theta  g }{ 8 \sqrt{ 9 g^2 + g_\theta^2} } \\[2mm]
                        \frac {\partial s}{\partial  r } &=& \frac{r^6 \sin^3 2\theta g_\theta }{8 \sqrt{ 9 g^2 + g_\theta^2} } \\
         \end{matrix} \right.   \ , \qquad
$$
which are satisfied by the function
\be
s=  \frac { r^7 \sin^3 (2\theta ) g_\theta }{ 56  \sqrt{ 9g^2 + g_\theta^2}}
\label{s}\ee
because of the equation satisfied by $g$. Similarly we obtain the solution for $t$
\be t= r^3g(\theta ) . \qquad \label{t}\ee
Using the properties of the function $g$ we can directly check  that function given by  the formulas (\ref{s})--(\ref{t}) is  a diffeomorphism with the required properties.
\qed

For future references let us keep in mind  that setting $\sin\phi,\ \cos\phi$ as in formula \equ{11}, we find simply
\begin{align}
\label{fact2a}
\begin{aligned}
\frac{\partial r}{\partial s} &= \frac r{7s} \sin^2\phi,  \\
\frac{\partial \theta }{\partial s} &= - \frac 1{14s} \sin (2\phi),
\end{aligned}
\end{align}
and
\begin{align}
\label{fact3a}
\begin{aligned}
\frac{\partial r}{\partial t} &= \frac r{3t} \cos^2\phi,  \\
\frac{\partial \theta }{\partial t} &=  \frac 1{6t} \sin (2\phi). 
\end{aligned}
\end{align}

Our next goal is to express the mean curvature operator (\ref{56.1}) in terms of the variables $(t,s)$.
Denoting by $\uu'$ the matrix $(\uu_t,\uu_s)$ problem \equ{eqmc} is transformed to
\be
 (uv)^{-3} \frac {1}{\sqrt{\det\uu'{\uu'}^T}} \nabla_{t,s} \cdot \left (  \frac { (uv)^3 \sqrt{\det\uu'{\uu'}^T}} { \sqrt{1+ |\nabla F|^2} }\, (\uu'{\uu'}^T)^{-1} \nabla_{t,s} F\, \right ) \ = \ 0.
\label{56.4}\ee
From Lemma (\ref{lemma ts}) we find
\be
 \left <  \uu_t , \uu_t \right > = \frac{1}{|\nabla F_0|^2} ,\quad
\left <  \uu_t , \uu_s \right > = 0 ,\quad \left <  \uu_s , \uu_s \right > = \frac{1}{(uv)^6}:=\rho^2, \label{56.3}\ee
hence we  compute
\be
{\det }\uu'= \frac{-\rho}{|\nabla F_0|}, \quad (\uu{'} {\uu'}^T)^{-1}=
\left  ( \begin{matrix}  |\nabla F_0|^2 &    0 \\[2mm]
                            0 & \rho^{-2}
         \end{matrix} \right )\ .
\label{56.5}\ee
Then equation \equ{56.4} becomes
\be
|\nabla F_0|\partial_t  \Big( \frac{ |\nabla F_0|\partial_t F} {\sqrt{1+|\nabla F|^2}}\Big)\,
+
|\nabla F_0|\partial_s\Big( \frac{ \rho^{-2} \partial_s  F} { |\nabla F_0| \sqrt{1+|\nabla F|^2}}\Big)=0.
\label{56.6}\ee
Let us observe that:
\begin{align*}
\nabla F &=
\left < \nabla F, \frac{\nabla F_0}{|\nn F_0|} \right > \,
\frac{\nabla F_0}{|\nn F_0|}
+ \left < \nabla F, \frac{\nabla F_0^\perp }{|\nn F_0|} \right > \,
\frac{\nabla F_0^\perp }{|\nn F_0|}
\\&=
F_t \nn F_0 + \rho^{-1} F_s
\frac{\nabla F_0^\perp }{|\nn F_0|} .
\end{align*}
From this we have
\begin{align*}
1+ |\nn F|^2 &=
1+ |\nn F_0|^2\left( \,
  F_t ^2  + \frac{ \rho^{-2} F_s^2}{|\nn F_0 |^2} \, \right ) \\
  &= |\nabla F_0|^2   (  \frac{1}{|\nabla F_0|^2} + F_t^2 + \frac{\rho^{-2} F_s^2}{|\nabla F_0|^2} ).
\end{align*}
Denoting by $Q(\nabla_{t,s}F)$ the function
$$
Q(\nabla_{t,s}F)= \frac{1}{|\nabla F_0|^2} + F_t^2 + \frac{\rho^{-2} F_s^2}{|\nabla F_0|^2},
$$
we see  the mean curvature equation is equivalent to
$$
H[F]=\frac{|\nabla F_0|}{Q^{3/2}(\nabla_{t,s}F)}G[F]=0
$$
where
\begin{align}
\begin{aligned}
G[F]&= Q(\nabla_{t,s}F) F_{tt}- \frac{1}{2} \partial_t Q(\nabla_{t,s}F)  F_t
+Q(\nabla_{t,s}F) \partial_s \left(\frac{\rho^{-2} F_s}{|\nabla F_0|^2}\right)
\\
&\quad-\frac{1}{2} \partial_s Q(\nabla_{t,s}F) \frac{\rho^{-2} F_s}{|\nabla F_0|^2}.
\end{aligned}
\label{defG}
\end{align}

Now  we derive the mean curvature operator for functions of the form
$$
F = F_0 +A \varphi (t,s) = t+A \varphi (t,s),
$$
where $A$ is constant parameter.
Our goal  is to  write the resulting equation in the form of a polynomial in $A$.  In general we assume that for $r \gg1$,
\be
\label{phis}
|\varphi_t| + \frac{|\varphi_s \rho^{-1}|}{ |\nabla F_0|}
 = o(1).
 \ee
We compute
\begin{align*}
\nabla F &=
\nabla F_0 +
\left < \nabla\varphi, \frac{\nabla F_0}{|\nn F_0|} \right > \,
\frac{\nabla F_0}{|\nn F_0|}
+ \left < \nabla\varphi, \frac{\nabla F_0^\perp }{|\nn F_0|} \right > \,
\frac{\nabla F_0^\perp }{|\nn F_0|}\\
&=
\nn F_0 + \varphi_t \nn F_0 + \rho^{-1} \varphi_s
\frac{\nabla F_0^\perp }{|\nn F_0|} .
\end{align*}
Then we have
\begin{align*}
1+ |\nn F|^2 &=
1+ |\nn F_0|^2\Big[
 (1+A \varphi_t )^2  + A^2\frac{ \rho^{-2} \phi_s^2}{|\nn F_0 |^2} \Big] \\
&= |\nabla F_0|^2 \Big( 1+\frac{1}{|\nabla F_0|^2} + 2A \varphi_t +A^2 R_1\Big),
\end{align*}
where we denote
$$
R_1=  \varphi_t^2 +  \frac{\rho^{-2} \varphi_s^2}{|\nabla F_0|^2}.
$$
It is convenient to introduce
$$
R= \Big( 1+\frac{1}{|\nabla F_0|^2} + 2A \varphi_t +A^2 R_1\Big).
$$
With these notations we have:
\begin{align}
\begin{aligned}
&|\nabla F_0|^{-1} R^{3/2}H[F_0+A\varphi]\\
&=\Big[AR\partial_t^2\varphi-\frac{1}{2}(1+A\partial_t\varphi)\partial_t R+AR\partial_s\Big(\frac{\rho^{-2}\partial_s\varphi}{|\nabla F_0|^2}\Big)-\frac{1}{2}A\frac{\rho^{-2}\partial_s\varphi}{|\nabla F_0|^2}\partial_s R\Big]
\\
&= -\frac{1}{2}\partial_t|\nabla F_0|^{-2}+A\Big[|\nabla F_0|^{-2}\partial_t^2\varphi-\frac{1}{2}\partial_t|\nabla F_0|^{-2}\partial_t\varphi\\
&\qquad\quad+\partial_s\Big(\frac{\rho^{-2}\partial_s\varphi}{|\nabla F_0|^2}\Big)(1+|\nabla F_0|^{-2})-\frac{1}{2}\Big(\frac{\rho^{-2}\partial_s\varphi}{|\nabla F_0|^2}\Big)\partial_s|\nabla F_0|^{-2}\Big]
\\
&\qquad+A^2\Big[\partial_t\varphi\partial^2_t\varphi-\frac{1}{2}\partial_t R_1+2\partial_t\varphi\partial_s\Big(\frac{\rho^{-2}\partial_s\varphi}{|\nabla F_0|^2}\Big)-\Big(\frac{\rho^{-2}\partial_s\varphi}{|\nabla F_0|^2}\Big)\partial^2_{ts}\varphi\Big]\\
&\qquad+A^3\Big[R_1\partial_t^2\varphi-\frac{1}{2}\partial_t\varphi\partial_t R_1+R_1\partial_s\Big(\frac{\rho^{-2}\partial_s\varphi}{|\nabla F_0|^2}\Big)-\frac{1}{2}\Big(\frac{\rho^{-2}\partial_s\varphi}{|\nabla F_0|^2}\Big)\partial_s R_1\Big].
\end{aligned}
\label{a terms}
\end{align}
In the sequel we will refer to the consecutive term in (\ref{a terms}) as the $A^0$, $A^1$, $A^2$ and $A^3$ terms respectively.
For future references we observe that the $A^0$ term can be written as
\begin{align}
 -\frac{1}{2}\partial_t|\nabla F_0|^{-2}=|\nabla F_0|^{-1}(1+|\nabla F_0|^{-2})^{3/2}H[F_0],
 \label{a0 term}
 \end{align}
and the $A^1$ term can be written as
\begin{align}
\begin{aligned}
&\Big[\cdot\Big]=|\nabla F_0|^{-1} \tilde L_0[\varphi]\\
&\qquad - \frac{3}{2}\partial_t|\nabla F_0|^{-2}\partial_t\varphi+|\nabla F_0|^{-2}\partial_s\Big(\frac{\rho^{-2}\partial_s\varphi}{|\nabla F_0|^2}\Big)-\frac{1}{2}\Big(\frac{\rho^{-2}\partial_s\varphi}{|\nabla F_0|^2}\Big)\partial_s|\nabla F_0|^{-2},
\end{aligned}
\label{a1 term}
\end{align}
where
\begin{align}
\tilde L_0[\varphi]=|\nabla F_0|\Big[\partial_t\Big(\frac{\partial_t \varphi}{|\nabla F_0|^2}\Big)
+
\partial_s\Big( \frac{ \rho^{-2}\partial_s  \varphi} {|\nabla F_0|^2}\Big)\Big].
\label{L0}
\end{align}




\bigskip

\subsection{A refinement of the result of Bombieri, De Giorgi and Giusti }
In this  section, taking  the existence result in \cite{bdg} as the point of departure, we find the asymptotic behavior of the minimal graph. Our goal is prove the following theorem:
\begin{teo}\label{exist fmc}
There exists a solution $F=F(u,v)$ to the mean curvature equation with the following properties
\begin{equation}
F_0\leq F\leq  F_0 + \frac{\tt C}{r^\sigma}\min\{F_0,1\}, \quad r>R_0,
\end{equation}
where $0< \sigma<1 $, ${\tt C}\geq 1$, and  $R_0$, are  positive constants.
\label{thm-mc1}
\end{teo}
The rest of this section is devoted to the proof of the theorem. Our approach, which is  based on a comparison principle, relies on a refinement of the supersolution/subsolution  in \cite{bdg}.
We need the following  comparison principle:
\begin{lemma}
Let $\Omega$ be a smooth and open  bounded domain. If $F_1$ and $F_2$ satisfies
\be
H[F_1] \leq H[F_2] \ \mbox{in} \ \Omega, \ F_1 \geq F_2 \ \mbox{on} \ \partial \Omega
\ee
Then
\be
F_1 \leq F_2 \ \mbox{in} \ \Omega.
\ee
\end{lemma}
\proof{} The proof is simple since
 \[ H[F_1]- H[F_2]= \sum_{i,j} a_{ij} \frac{\partial^2}{\partial x_i x_j} (F_1-F_2)
\]
where the matrix $(a_{ij})$ is uniformly elliptic in $\Omega$. By the usual Maximum Principle, we obtain the desired result.
\qed

Let us observe that  from (\ref{min 4h}) we have
\begin{equation}
\label{mc0}
  \min (\frac{ -\cos (2\theta)}{g(\theta)} )\geq 1, \theta \in (\frac{\pi}{4}, \frac{\pi}{2}).
\end{equation}
Thus for $F_0= r^3 g(\theta)$  it holds \be
\label{mc00}
F_0= r^3 g(\theta) \leq (v^2-u^2) (v^2+u^2)^{\frac{1}{2}}.
\ee
We will now construct a subsolution to the mean curvature equation.
\begin{lemma}
\label{sub1-n}
Let $H[F]$ denote the  mean curvature operator. We have
\be
\label{sub1}
H[F_0] \geq 0.
\ee
It holds as well:
\begin{align}
H[F_0]=O(r^{-5}).
\label{sub1aa}
\end{align}
\end{lemma}
\proof{}
Since $H[F]$ and $G[F]$ (defined in (\ref{defG})) differ only by a nonnegative factor it suffices to show that
\begin{equation}
G[F_0]\geq 0.
\label{sub 1}
\end{equation}
In fact, let $F=F_0=t$, we then have
\begin{align*}
G[F_0]&=- \frac{1}{2} \partial_t Q(\nabla_{t,s}F_0)\\
&=-\frac{1}{2}\partial_t\Big(\frac{1}{|\nabla F_0|^2}\Big),
\end{align*}
where
$$
\frac{1}{|\nabla F_0|^2}= \frac{1}{ r^4 (9 g^2 + g_\theta^2)}=\frac{ r^2 \cos^2 \phi}{ 9 t^2}.
$$
By the formulas (\ref{fact3a}), we have
\begin{align}
\begin{aligned}
-\partial_t \Big(\frac{r^2 \cos^2 \phi}{9t^2}\Big)&= \frac{r^2}{9t^3} \left [{2 \cos^2 \phi} - \frac{2tr_t\cos^2\phi}{r}  + {t\phi'\theta_t\sin (2\phi)} \right ]\\
&=\frac{2 r^2\cos^2\phi}{9t^3}\left[\frac{2}3\cos^2\phi+\frac{1}3\sin^2\phi(\phi'+3)\right]
\\
&\geq 0,
\end{aligned}
\label{92 z}
\end{align}
where we have used the fact that  $\phi'(\theta)\geq -3$. Estimate (\ref{sub1aa}) follows easily from this.
This ends the proof.
\qed

By the standard theory of the mean curvature equation for each fixed $R>0$, there exists a unique solution to the following problem
\begin{equation}
\label{mc11}
\frac{1}{(uv)^3} \nabla\cdot \left ( \frac { (uv)^3 \nabla F} { \sqrt{1+ |\nabla F|^2} }\right ) = 0 \ \mbox{in} \ \Gamma_R, F=F_0 \ \mbox{on} \ \partial \Gamma_R
\end{equation}
where  $ \Gamma_R=B_R \cap T$, $T=\{u,v>0, u<v\}$.
Let us denote the solution to (\ref{mc11}) by $F_R$.

Using  (\ref{mc00}),  the comparison principle and the supersolution found in \cite{bdg}, we have
\be
\label{mc3}
  F_0 \leq F_R \leq   {\mathcal H}\Big( (v^2-u^2) + (v^2-u^2) (u^2 +v^2)^{1/2}  (1+{\tt A} (|\cos (2\theta)|)^{ \lambda -1} )\Big)
\ee
where
$$
{\mathcal H}(t) := \int_0^t \exp\Big( {\tt B} \int_{|w|}^\infty \frac{dt}{ t^{2-\lambda} (1+ t^{2\alpha \lambda -2 \alpha})}\Big) d w,
$$
$\lambda>1$ is a  positive fixed number,  $\alpha= \frac{3}{2}$, and ${\tt A}$, ${\tt B}$ are sufficiently large positive constants.
This inequality, combined with standard elliptic estimates, imply that as $R \to +\infty$, $F_R \to F$ which is a solution to the mean curvature equation $ H[F]=0$  with
\be
\label{mc4}
  F_0 \leq F \leq  {\mathcal H}\Big((v^2-u^2) + (v^2-u^2) (u^2 +v^2)^{1/2}  (1+ {\tt A} (|\cos (2\theta)|)^{ \lambda -1} )\Big).
\ee

Next we need the following key lemma:
\begin{lemma}\label{sup4}
There exists $\sigma_0\in (0,1)$ such that for each $\sigma\in (0,\sigma_0)$ there exists $a_0>1$ such that for each \blue{sufficiently large} $\tilde{A} \geq 1$, we have
\be
H[F_0+ \frac{\tilde{A} F_0 }{r^\sigma}] \leq 0, \ \ \mbox{for} \ r> a_0.
\label{mc5a}\ee
Moreover, under the same assumptions   for each \blue{sufficiently large} $A\geq 1$ we have
\be
\label{mc5}
H[F_0+ \frac{A}{r^\sigma}] \leq 0, \ \ \mbox{for} \ r> a_0 A^{\frac{1}{3+\sigma}}.
\ee
\end{lemma}
\proof{}
We will consider (\ref{mc5a}) first. We will use formula (\ref{a terms}) to write
$H[F_0+\frac{\tilde{A} F_0 }{r^\sigma}]$ multiplied by a nonnegative factor as a polynomial in
$\tilde A$. Explicit computations (\ref{a terms}) yield
\begin{align*}
|\nabla F_0|^{-1} R^{3/2}H[F_0+\frac{\tilde{A} F_0 }{r^\sigma}]&= H_0+\tilde A H_1+\tilde A^2 H_2+\tilde A^3 H_3,
\end{align*}
where
\begin{align}
\begin{aligned}
H_0&=|\nabla F_0|^{-1}(1+|\nabla F_0|^{-2})^{3/2}H[F_0]= \frac{r^2\cos^2\phi}{9t^3}\left[\frac{2}3\cos^2\phi+\frac{1}3\sin^2\phi(\phi'+3)\right],\\
H_1&=\frac{-7\sigma \cos^2 \phi}{9 t r^\sigma} ( 7 +  (2\phi^{'} -\sigma) \sin^2 \phi ) + \frac{\cos^2\phi}{ tr^{\sigma}}O(r^{-4}).
\end{aligned}
\label{mc 51}
\end{align}
In the Appendix A we show in addition that
\begin{align}
\begin{aligned}
H_2&= \frac{\cos^2\phi}{ tr^{\sigma}}O(r^{-\sigma})\leq 0,\\
H_3&=\frac{\cos^2\phi}{ tr^{\sigma}}O(r^{-2\sigma})\leq 0.
\end{aligned}
\label{mc 52}
\end{align}
Let us observe that \blue{the first} term in (\ref{mc 51}) is bounded by
\begin{align}
H_0&\leq c_1 \frac{r^2\cos^4\phi}{t^3}\leq c_1\frac{\cos^2\phi}{tr^4}.
\label{mc 53}
\end{align}
Estimate (\ref{mc 53}) follows from (\ref{mc 51}) and the fact that $\phi(\pi/4)=\pi/2$, $\phi'((\pi/4)^+)=-3$, $\phi''((\pi/4)^+)=0$. Summarizing, we have
\begin{align}
\begin{aligned}
H[F_0+\frac{\tilde{A} F_0 }{r^\sigma}]&\leq H_0+\tilde AH_1\\
&\leq \frac{-7\tilde A\sigma \cos^2 \phi}{9 t r^\sigma} ( 7 +  (2\phi{'} -\sigma) \sin^2 \phi ) + \frac{\cos^2\phi}{ tr^\sigma}O(r^{-4+\sigma})\\
&\leq 0.
\end{aligned}
\label{mc 54}
\end{align}
To prove (\ref{mc5}) we use a similar argument. Writing  $H[F_0+\frac{A}{r^{\sigma}}]$ as a polynomial in $A$ we get that the $A^0$ term is equal to $H_0$ in (\ref{mc 51}) and:
\begin{align}
H_1&=\frac{-7\sigma \cos^2 \phi}{9 g^2(\theta) r^{6+\sigma}} ( 7 +  (2\phi{'} -\sigma) \sin^2 \phi ) + \frac{1}{ r^{6+\sigma}}O(r^{-1}).
\label{mc 55}
\end{align}
The other terms satisfy:
\begin{align*}
H_2&=\frac{1}{r^{6+\sigma}}O(r^{-3-\sigma}), \quad (A^2\ \mbox{term}),\\
H_3&=\frac{1}{r^{6+\sigma}}O(r^{-6-2\sigma}), \quad (A^3\ \mbox{term}).
\end{align*}
Since $H_0=O(r^{-7})$ the lemma follows by combing the above estimates.
\qed

Now we can prove Theorem \ref{exist fmc}:
 In fact, from (\ref{mc3}),  we have
\be
F_0 \leq F_R \leq F_0 +  \frac{\tilde{A} F_0 }{r^\sigma}, \ \  \mbox{for} \ r= a_0,
\ee
if we choose $\tilde{A} \geq 1 $ such that
\begin{align}
\label{mc09}
\begin{aligned}
&\max_{\theta }  \frac{{\mathcal H}\Big( a_0 (-\cos(2\theta)) + a_0^{3/2} (-\cos (2\theta))  (1+ {\tt A} (|\cos (2\theta)|)^{ \lambda -1} )\Big)}{ (a_0^{3} + \tilde{A} a_0^{3-\sigma}) g(\theta)}\leq 1,
\end{aligned}\end{align}
which is possible since \blue{$\frac{|\cos(2\theta)|}{g(\theta)}<\infty$ (this follows from (\ref{min 4h}) and the fact that $g_\theta(\frac{\pi}{4})>0$).} Note that (\ref{mc09}) holds for any $\tilde{A}$ large.

By comparison principle in the domain $ \Gamma_R \setminus B_{a_0}$, (noting that the function $F_0 +  \frac{\tilde{A} F_0 }{r^\sigma}$ is a super-solution for $ r>a_0$ by Lemma \ref{sup4} and the function $F_0$ is a sub-solution by Lemma \ref{sub1-n}), we deduce that
\be
F_0 \leq F_R \leq F_0 +  \frac{\tilde{A} F_0 }{r^\sigma}, \ \  \mbox{in}\   \Gamma_R \setminus B_{a_0},
\label{mc07}
\ee
and hence
\be
\label{mc7}
F_0 \leq F_R \leq F_0 +   \tilde{A} r^{3-\sigma}, \ \  \mbox{in} \   \Gamma_R \setminus B_{a_0},
\ee
for $\tilde{A}$ large.

Let $A\geq 1$ be a constant to be chosen later and let  us consider the region $\Gamma_R\cap \{ r >R_0 \}$, where $ R_0= a_0 A^{\frac{1}{3+\sigma}}$.  From (\ref{mc7}), we then have
\be
F_0 \leq F_R \leq F_0 +   \tilde{A} R_0^{3-\sigma} \leq F_0+ \frac{A}{R_0^\sigma}, \ \  \mbox{for} \ r =R_0
\ee
if we choose
\be
\tilde{A}\leq \frac{A}{R_0^3}= \frac{ A}{ a_0^3 A^{\frac{3}{3+\sigma}}}= a_0^{-3} A^{\frac{\sigma}{3+\sigma}}.
\label{mc9aa}
\ee
By comparison principle applied now in  $\Gamma_R \cap \{ r >R_0 \}$, using Lemma \ref{sup4}, we then obtain
\be
\label{mc9}
F_0 \leq F_R  \leq F_0+ \frac{A}{r^\sigma}, \ \  \mbox{for} \ r  \geq R_0= a_0 A^{\frac{1}{3+\sigma}}.
\ee
The assertion of the Theorem follows now by  combing (\ref{mc07}) and (\ref{mc9}) and letting $R\to \infty$.

\qed

The second  Theorem of this section  improves the super-solution  and further refines the estimate on $F$.
\begin{teo}\label{teo mc2}
There exists $\sigma_0\in (0,1)$ such that for each $\sigma\in (0,\sigma_0)$ there exists $a_0>1$ such that for each \blue{sufficiently large} $A\geq 1$, we have
\be
\label{mc50}
H[F_0+ \frac{A \tanh ( F_0 r^{-1}) }{r^\sigma}] \leq 0, \ \ \mbox{for} \ r> a_0 A^{\frac{1}{1+\sigma}}
\ee
As a consequence there are  constants ${\tt C}$, $R_0$ such that the solution \blue{to the mean curvature equation described}  in Theorem \ref{exist fmc} satisfies:
\be
\label{mc50-1}
F_0 \leq F \leq F_0+ \frac{{\tt  C} \tanh ( F_0 r^{-1}) }{r^\sigma}, \quad\mbox{for} \ r> R_0.
\ee
\end{teo}
\proof{}
Let us prove (\ref{mc50}) first.   We will denote
\begin{align*}
F=F_0+\frac{A}{r^\sigma}\varphi(t,s), \quad \varphi(t,s)=\tanh(t/r).
\end{align*}
Note that  the $A^0$ and $A^1$ terms in (\ref{a terms}) are:
\begin{align}
\begin{aligned}
&-\frac{1}{2} \partial_t \Big(\frac{1}{|\nabla F_0|^2}\Big) + \frac{A}{|\nabla F_0|^2} \partial^2_{t} \varphi - \frac{A}{2} \partial_t \Big(\frac{1}{|\nabla F_0|^2}\Big) \partial_t \varphi\\
&\quad
+ A\Big(1+ \frac{1}{|\nabla F_0|^2}\Big) \partial_s\Big(\frac{ \rho^{-2} \varphi_s}{ |\nabla F_0|^2}\Big) -\frac{A}{2} \partial_s\Big(\frac{1}{|\nabla F_0|^2}\Big) \frac{\rho^{-2} \varphi_s}{|\nabla F_0|^2}
\\
& =|\nabla F_0|^{-1} H[F_0]+
A|\nabla F_0|^{-1}\Big[|\nn F_0| \pp_t \Big(\frac{\varphi_t}{|\nn F_0|^2}\Big) + |\nn F_0| \pp_s
\Big (\frac{\rho^{-2} \varphi_s}{|\nn F_0|^2}\Big) \Big]
\\
&\quad
-A\Big[\frac{3}{2} \partial_t\Big(\frac{1}{|\nabla F_0|^2}\Big) \partial_t \varphi  -
\frac{1}{|\nabla F_0|^2} \partial_s\Big (\frac{ \rho^{-2} \varphi_s}{|\nabla F_0|^2}\Big) +\frac{1}{2} \partial_s\Big(\frac{1}{|\nabla F_0|^2}\Big) \frac{\rho^{-2} \varphi_s}{|\nabla F_0|^2}\Big].
\label{mc 50a}
\end{aligned}
\end{align}
We have by (\ref{mc 53})
\begin{align}
H_0=|\nabla F_0|^{-1} H[F_0]& \leq c_1\frac{\cos^2\phi}{tr^4}\leq c_1\frac{\cos\phi}{r^7}.
\label{mc 50b}
\end{align}
Now we will deal with the first $A^1$ term in (\ref{mc 50a}). This term is given explicitly in (\ref{a1 term}). We recall here that in (\ref{L0}) we have defined the following operator:
\begin{align}
\tilde L_0[\varphi]:=|\nn F_0| \pp_t \Big(\frac{\varphi_t}{|\nn F_0|^2}\Big) + |\nn F_0| \pp_s
\Big (\frac{\rho^{-2} \varphi_s}{|\nn F_0|^2}\Big).
\label{mc 50c}
\end{align}
We will prove the following Lemma:
\begin{lemma}\label{lemma tl01}
There exists $\sigma_0>0$ such that for each $\sigma\in (0,\sigma_0)$ there exist $a_0>0$ and $c_0>0$  such that
\begin{align}
\tilde L_0 [ r^{-\sigma}\tanh(t/r) ]&\leq -\frac{c_{0}}{ r^{4+\sigma}}\min\{1, t/r\}, \quad r>a_0.
\label{mc 50d}
\end{align}
\end{lemma}
\proof
Let us denote:
$$
\beta (\eta)=\tanh (\eta), \ \eta= \frac{t}{r},  \quad
 \beta_1(\eta)=\beta (\eta)-\frac{1}{\sigma} \beta{'} \eta,
$$
and
\be
\varphi = \beta (\eta) r^{-\sigma}, \qquad \sigma  >0.
\label{60.03}\ee
Then we compute
$$
\partial_s \varphi= -\frac{\sigma r^{-\sigma} \sin^2 \phi}{7s}\beta_1,
$$
hence
\begin{align*}
\pp_s \Big(\frac{\rho^{-2}\pp_s \varphi}{|\nn F_0|^2}\Big)
&
= -c_1 \sigma\pp_s \left ( \frac{ r^{-\sigma} }{t^2} \cos^2\phi \beta_1 \right )
\end{align*}
where \blue{$c_1>0$.  From now on, by $c_i>0$ we will denote  generic positive constants.} We obtain
\begin{align}
\begin{aligned}
&\pp_s\Big( \frac{\rho^{-2}}{|\nn F_0|^2} \pp_s \varphi\Big)\\
&  =-\frac{c_1 \sigma r^{-6-\sigma}}{9g^2+g_\theta^2}\Big\{\beta_1\Big[1+ \frac{2\sin^2\phi}{7}\Big(\frac {-\sigma}{2}   + \phi' \Big)\Big] - \frac{\eta\sin^2\phi}{7}\beta'_1 \Big\}
\end{aligned}
\label{79}
\end{align}
On the other hand, we have
$$
\partial_t \phi_0= -\frac{ \sigma r^{-\sigma}}{3 t} \beta \cos^2 \phi +  \beta{'} (1-\frac{\cos^2\phi}{3}) r^{-\sigma-1},
$$
and
$$
 \frac{ \partial_t \phi_0}{|\nabla F_0|^2}
=  \frac{ r^{1-\sigma} \cos^2 \phi}{ 9 t^2}
\Big[\frac{-\sigma}{3}
(\frac{\beta}{\eta}) \cos^2 \phi+\Big(1-\frac{\cos^2 \phi}{3}\Big) \beta{'}\Big],
$$
hence
\begin{align}
\begin{aligned}
\partial_t\Big( \frac{1}{|\nabla F_0|^2} \partial_t \phi\Big)
&= \Big[-  \frac{ 2r^{1-\sigma} \cos^2 \phi}{ 9 t^3}+\frac{r^{1-\sigma}\sin^2(2\phi)\phi'}{63 t^3}\Big] \Big(1-\frac{\cos^2 \phi}{3}\Big) \beta{'}
\\
&\quad + \frac{ r^{-\sigma} \cos^2 \phi}{ 9 t^2}\Big [ \frac{-\sigma}{3}
\Big(\frac{\beta}{\eta}\Big)' \cos^2 \phi  +\Big( 1-\frac{\cos^2 \phi}{3}\Big) \beta'' \Big] \Big(1-\frac{\cos^2 \phi}{3}\Big)
\\
&\quad+ O\big( \frac{ \cos \phi}{ r^{8+\sigma}}\big).
\end{aligned}
\label{91}
\end{align}
The first term in (\ref{91})  is negative.  The second term can be estimated as follows
 \begin{align}
 \begin{aligned}
& \frac{ r^{-\sigma} \cos^2 \phi}{ 9 t^2}
\Big[ \frac{-\sigma}{3}
\Big(\frac{\beta}{\eta}\Big)' \cos^2 \phi  + ( 1-\frac{\cos^2 \phi}{3}) \beta''\Big ] \Big (1-\frac{\cos^2 \phi}{3}\Big)
\\
&\quad \leq   \frac{c_2}{ r^{6+\sigma}}  \Big [ \frac{-\sigma}{3}
\Big(\frac{\beta}{\eta}\Big)' \cos^2 \phi  + \frac{2}{3}  \beta{''} \Big].
\end{aligned}
\label{92}
\end{align}
Combining (\ref{79}) and (\ref{92}), we have
\begin{align}
\begin{aligned}
\tilde L_0 [ \varphi ]
&\leq \frac{c_3}{ r^{4+\sigma}} \Big\{-\sigma\beta_1
\Big[1+ \frac{2\sin^2\phi}{7}\Big(\frac {-\sigma}{2}   + \phi' \Big)\Big] + \frac{\eta\sigma\sin^2\phi}{7} \beta_1'
\\
&\qquad\qquad+ \Big [ \frac{-\sigma}{3}
\Big(\frac{\beta}{\eta}\Big)' \cos^2 \phi  + \frac{2}{3}  \beta{''} \Big] \Big\}+O\big( \frac{ \cos \phi}{ r^{6+\sigma}}\big).
\end{aligned}
\label{92 a}
\end{align}
Denoting the term in brackets above by $\tilde a$ we can  estimate as follows:
\begin{align*}
\tilde a & \leq
\beta''\Big(c_4\eta^2\sin^2\phi+\frac{2}{3}\Big)-{c_5\sigma}\Big[\beta-c_6|\beta'\eta|-c_7
\Big|\Big(\frac{\beta}{\eta}\Big)'\Big|\Big].
\end{align*}
Given small $\ve_0>0$, let  $\eta_0>0$ be such that
$$
\beta-c_6|\beta'\eta|-c_7
\Big|\Big(\frac{\beta}{\eta}\Big)'\Big| \geq \ve_0, \quad \eta\geq \eta_0,
$$
hence for $\eta>\eta_0$ we have:
\begin{align}
\tilde a\leq -c_8\ve_0\sigma, \quad \mbox{for}\ \sigma\in (0,1/2).
\label{92 b}
\end{align}
On the other hand when $0\leq \eta\leq \eta_0$ then we have
\begin{align}
\tilde a\leq -c_9\eta\Big(\frac{1}{7}\eta^2+\frac{2}{3}\Big)-c_{10}\sigma \eta\leq -c_{11}\eta,
\label{92 c}
\end{align}
where $\sigma\in (0,\sigma_0)$ with $\sigma_0>0$ small.
Finally let us consider the last term in (\ref{92 a}).  When $\eta\leq 1$ then
$$
\frac{\cos\phi}{r^{6+\sigma}}\leq \frac{c_{12}\eta}{r^{8+\sigma}},
$$
while when $1\leq \eta$ then
$$
\frac{\cos\phi}{r^{6+\sigma}}\leq \frac{1}{r^{6+\sigma}}.
$$
Summarizing the above and (\ref{92 a})--(\ref{92 c}) we have  that  for each $\sigma\in (0,\sigma_0)$,  where
$\sigma_0$ is small, there exists $r_0>0$, $c_0$ such that :
\begin{align}
\begin{aligned}
\tilde L_0 [\varphi]&\leq   -\Big(\frac{c_{13}}{ r^{4+\sigma}}-\frac{c_{14}}{r^{6+\sigma}}\Big)\min\{1, \eta\}\\
&\leq -\frac{c_{0}}{ r^{4+\sigma}}\min\{1, \eta\}, \quad r>r_0.
\label{92 d}
\end{aligned}
\end{align}
\qed
\begin{remark}\label{remark beta}
We observe that Lemma \ref{lemma tl01} remains true with $\beta(\eta)=\tanh (\eta)$ replaced by  $\beta(\eta)=1-e^{\,-\eta}$, $\eta>0$ with no change in the proof.
\end{remark}
Continuing the proof of Theorem \ref{teo mc2} we notice that
\begin{align}
\begin{aligned}
- \frac{3}{2} \partial_t (\frac{1}{|\nabla F_0|^2}) \partial_t \varphi &
 \leq  \frac{ c_{15} \cos \phi}{ r^{8+\sigma} }\\
 & \leq  \frac{c_{15} \min \{\eta, 1\} }{ r^{8+\sigma} } ,
\end{aligned}
\label{C2-1}
\end{align}
since $\cos \phi \leq  \frac{\eta}{r^2}$, and
\begin{align}
\label{C2-2}
 \frac{1}{|\nabla F_0|^2} \partial_s (\frac{ \rho^{-2} \varphi_s}{ |\nabla F_0|^2}) &\leq  \frac{c_{16}}{ r^{8+\sigma}} \min\{\eta, 1\}\\
\label{C2-3}
 -\frac{1}{2} \partial_s (\frac{1}{|\nabla F_0|^2}) \frac{\rho^{-2} \varphi_s}{|\nabla F_0|^2} &\leq  c_{17}\frac{ |\beta_1 (\eta)| }{ r^{10+\sigma}} \leq  c_{17}  \frac{\min\{\eta, 1\}}{ r^{10+\sigma}}.
\end{align}
We analyze  the $A^2$-term and $A^3$ terms  in the expansion of
$H[F_0+Ar^{-\sigma}\tanh(t/r)]$. A typical term in (\ref{a terms}) is
\begin{align}
\begin{aligned}
-\frac{1}{2} \partial_t (\frac{\rho^{-2} F_s^2}{|\nabla F_0|^2}) &
=-\frac{\sigma^2 \sin^2 \phi \cos^2 \phi}{ 27 t^3 r^{2\sigma}} [ -\sigma  \cos^2 \phi -3+  \cos (2\phi)  \phi^{'} ] \beta_1^2\\
&\quad - \frac{\sigma^2 \sin^2 \phi \cos^2 \phi}{ 9 t^3 r^{2\sigma}} 2\beta_1 \beta_1^{'}  \eta (1-\frac{\cos^2 \phi}{3})\\
&= \sin^2 \phi  \min \{\eta, 1\}O({ r^{-7-2\sigma}}).
\end{aligned}
\label{C2-4}
\end{align}
Other $A^2$ terms are estimated in a similar way.
Direct calculations show that $A^3-$term satisfy
\begin{align}
H_3= \sin^2 \phi  \min\{\eta, 1\} O(r^{-8-3\sigma}).
\label{C2-5}
\end{align}
In conclusion, we have
\begin{align}
\begin{aligned}
H[F_0+Ar^{-\sigma}\tanh(F_0/r)]&\leq \Big(\frac{c_1}{r^7} -\frac{c_{0}A}{ r^{6+\sigma}}
+\frac {c_{18}A^2}{r^{7+2\sigma}}+\frac {c_{19}A^3}{r^{7+3\sigma}}
\Big)\min\{1, \eta\}
\\
&\leq 0,
\end{aligned}
\label{C2-6}
\end{align}
if we choose $a_0$ large and $r \geq a_0 A^{\frac{1}{1+\sigma}}$. This proves (\ref{mc50}).

Now we will show (\ref{mc50-1}). From (\ref{mc7}),  we have
\be
\label{mc7-1}
F_0 \leq F_R \leq F_0 +   \tilde{A} F_0 r^{-\sigma}, \ \  \mbox{for} \ r \geq a_0
\ee
for some  $\tilde{A}\geq 1$.

Let us consider the region
$$\Sigma:= B_R   \cap \{ v>u \} \cap \{ r >R_0 \} \cap \{0 \leq  \frac{F_0}{r} <1 \},$$
 where $ R_0= a_0 A^{\frac{1}{1+\sigma}}$, and $A$ is to be chosen.  From (\ref{mc7}), we have in $\Sigma$:
\be
F_0 \leq F_R \leq F_0 +   \tilde{A} F_0 R_0^{-\sigma} \leq F_0+ \frac{A \tanh (F_0 R_0^{-1}) }{R_0^\sigma}, \ \  \mbox{for} \ r =R_0,
\ee
if we choose
\be
\tilde{A}   \leq     \frac{A}{R_0} \frac{\tanh (F_0 R_0^{-1})}{F_0 R_0^{-1}} =A^{\sigma/(1+\sigma) } a_0^{-1} \sup_{|\eta|<1} \frac{ \tanh\eta}{\eta}.
\label{mc9b}
\ee
Consider now   the boundary $ \{ \frac{F_0}{r} =1 \}$.   We have by (\ref{mc9}):
\begin{align}
\label{mc9-1}
\begin{aligned}
F_0 \leq F_R  &\leq F_0+ \frac{A\tanh(1)}{r^\sigma} \\
&\leq F_0 + \frac{A\tanh (F_0/r)}{ r^\sigma}, \ \  \mbox{for} \ r  \geq R_0 \geq   a_0 (\tanh(1) A)^{\frac{1}{3+\sigma}},  \ \mbox{and}  \ F_0/r=1,
\end{aligned}
\end{align}
if we chose (c.f. (\ref{mc9aa})):
\begin{align}
\tilde{A}\leq  a_0^{-3} (\tanh(1)A)^{\frac{\sigma}{3+\sigma}}.
\label{mc9c}
\end{align}
Choosing $A$ larger if necessary we can assume that in addition to  (\ref{mc9b}) also (\ref{mc9c}) is satisfied.
By comparison principle applied to $\Sigma$, we then obtain
\be
\label{mc9-2}
F_0 \leq F_R  \leq F_0+ \frac{A \tanh (F_0/r)}{r^\sigma}, \ \  \mbox{for} \ r  \geq R_0.
\ee
Passing to the limit $R\to \infty$ we then get:
\be
\label{mc9-3}
 F_0\leq F  \leq F_0+ \frac{A \tanh (F_0/r)}{r^\sigma}, \ \  \mbox{for} \ r  \geq R_0,
\ee
in $\Sigma$.
Combining this with the statement of Theorem \ref{exist fmc} to estimate $F$ for $r>R_0$ in the complement of $\Sigma$ we complete the proof.
\qed

\setcounter{equation}{0}
\section{Local coordinates for the minimal graph}\label{sec grad est}

The minimal graph of Bombieri, De Giorgi and Giusti,  $\Gamma=\{x_9= F(x')\}$ can also be represented  \blue{locally as}  a graph over its
tangent hyperplane $T_{p_0}\Gamma$ at  $p_0= (x_0, F(x_0))$,
with $|x_0| =R$. In other words, for each fixed $p_0\in \Gamma$ there is   a function
$G(t)$ such that, for some $\rho, a>0$,
\begin{align*}
\Gamma\cap B_\rho(p_0)=p_0+\{(t, G(t))\mid |t|<a\}
\end{align*}
 where \blue{$t=(t_1,\dots,t_8)$ are the Euclidean} coordinates on $T_{p_0}\Gamma$. More precisely,
 $F(x)$ and $G(t)$ are linked through the following relation:
\be
\left [ \begin{matrix} x\\ F(x)    \end{matrix}\right ] = \left [ \begin{matrix} x_0\\ F(x_0)    \end{matrix}\right ] + \Pi t + G(t) n(p_0)
\label{relFG}\ee
Here
$$
\Pi t = \sum_{j=1}^8 t_i \Pi_i, \quad t\in \R^8,
$$
where $\{\Pi_1, \Pi_2, \ldots, \Pi_8\}$ is a choice of an orthonormal
basis for the tangent space to the minimal graph at the point $p_0= (x_0,F(x_0))$, and
$$
n(p_0) =  \frac 1{\sqrt{1+|\nn F(x_0)|^2}} \left [ \begin{matrix} \nn F(x_0) \\ -1    \end{matrix}\right ],
$$
so that
$$
G(t) =   \frac 1{\sqrt{1+|\nn F(x_0)|^2}}\big(F(x) -F(x_0) - \nn F(x_0)\cdot (x-x_0)\big) .
$$
The implicit function theorem implies that  $G$ and $x$, \blue{given in equation
\equ{relFG}}, are smooth functions of $t$ , at least while $|t| < a$ for a sufficiently
small number $a >0$.  \blue{Clearly when $p_0$ is restricted to some fixed compact set than there exists a $\theta>0$ such that
\begin{align*}
a=\theta(1+R), \quad R=|x_0|.
\end{align*}
To show a similar bound for all $p_0\in \Gamma$ we will assume  $|x_0|=R>1$.
The  bound we are seeking amounts to estimating} (from below) the  largest $a$ so that 
$$
\sup_{|t| < a } |D_t G(t)| \ <\ +\infty.
$$
Here and below by \blue{$D_t$, $D^2_t$} etc. we will denote the derivatives with respect to the local variable $t$.
Let $n(z)$ denote unit normal at the point $z=(t,G(t))$ (with some abuse of notation $n(p_0)\equiv n(0)$).
Let us set
$$
\hat t = \frac t{|t|}
$$
and consider \blue{the following} curve  on  the minimal surface:
$$
r\mapsto  \gamma(r) := ( r\hat t , G (r\hat t )), \quad 0<r\le |t| .
$$
Then,
$$
\partial_r n(\gamma(r))  = A_\Gamma(\gamma(r)) [ (\hat t , D_t G(r\hat t)\cdot \hat t )]
$$
where $A_\Gamma$ is the  second fundamental form on $\Gamma$ \blue{and $D_t G(r\hat t)=D_tG(t)\left|_{t=r\hat t}\right.$}.
Thus
$$
|n(\gamma (r) ) - n(0) | \le    \sup_{ 0<s< r }  |A_\Gamma(\gamma(s))|  \int_0^r \, (1 +
|G'(s\hat t)| )\, ds.
$$
We will now make use of Simon's estimate (Theorem 4, p. 673 and Remark 2, p. 674 in  \cite{simon 1}) which yields:
$$
 \sup_{ 0<s< r }  |A_\Gamma(\gamma(s))|  < \frac cR,
$$
\blue{since we can assume that $|t|< \theta R$, with some small $\theta>0$.} 
In addition we have that
$$
|n(\gamma (r) ) - n(0) | \ge   \frac {|D_tG(r\hat t)| }{  1+ |D_tG( r\hat t )| },
$$
hence
$$
 \frac {|D_tG(r\hat t )| }{  1+ |D_tG( r\hat t  )| }\ \le\ \frac cR  \int_0^r \, (1 +
|D_t G(s\hat t)| )\, ds.
$$
Let us write $\ve = \frac cR $ and
$$
\psi(r) :=  \int_0^r \, (1 +
|D_t G(s\hat t)| )\, ds.
$$
The above inequality reads
$$
 1- \frac 1{\psi'(r)} \le  \ve \psi(r),
$$
or 
$$
(1 - \ve\psi(r))\psi'(r)  \le 1,
$$
so that for all sufficiently small \blue{(relative to the size of $\ve$)} $r>0 $ we have that
$$
1- (1 - \ve\psi(r))^2  \le 2\ve r.
$$
Since $\psi(0) =0$ it follows that
$$
(1-2\ve r)^{\frac 12}  \le (1 - \ve\psi(r)),
$$
hence
$$
1- \frac 1{ 1+ |D_t G(r\hat t)|} \le \ve\psi(r) \le 1 - (1-2\ve r)^{\frac 12},
$$
which implies
$$
|D_t G(t)| \le  (1-2\ve |t|)^{-\frac 12} - 1 \le  8\ve |t|,
$$
provided that
$
\ve |t| < \frac 14.
$
Hence we have established that there are positive numbers $\theta$,  $c$, independent of $R$  such that
\be
|D_t G(t)|\, \le \,  \frac cR |t| \foral |t| < \theta R\ .
\label{cotaG}
\ee
In particular, we obtain a uniform bound on $D_tG(t)$ for
$|t| \le \theta R $,
while at the same time
\be
|n(t,G(t)) - n(0) | \, \le\, \frac cR |t| \foral |t| < \theta R\ .
\label{cotan}\ee
This guarantees the fact that our minimal surfaces
indeed defines a graph over the tangent plane at $p_0$, at least for  $|t| \le \theta R $.
The quantities $x(t)$ and $G(t)$ linked by equation \equ{relFG} are thus well-defined, provided that
$|t|< \theta R$. The implicit function theorem yields in addition their differentiability. We have
\be
\left [ \begin{matrix} D_t x(t) \\ \nabla F(x)\cdot D_t x(t)     \end{matrix}\right ]  = \Pi    +    D_t G(t)n(p_0),
\label{derx}\ee
and in particular $|D_t x(t)|$ is uniformly bounded in $|t|< \theta R$. The above relation also
tells us that
\be |D_t^m x(t)|  \, \le\,  | D_t^m G(t)|\, ,\quad   m\ge 2, \quad  |t|< \theta R.
\label{boundxd}\ee

\bigskip
Let us estimate now the derivatives of $G$.
Since $ G(t)$ represents a minimal graph, we have that
\be
H[G] = \nn_t\cdot\Big( \frac{\nn_t G}{\sqrt{1+ |\nn_t G|^2} }\Big) = 0 \quad
\hbox{in}\quad B(0, \theta R)\subset  \R^8.
\label{mc1}\ee
Let us consider now the change of variable
$$
\ttt G (t) = \frac 1R G(R t) ,
$$
and observe that $\ttt G$ is bounded and satisfies
\be
H[\ttt G] = \nn_t\cdot \Big( \frac{\nn_t \ttt G}{\sqrt{1+ |\nn_t\ttt G|^2} }\Big) = 0 \quad
\hbox{in}\quad B(0, \theta ) .\label{mc2}\ee
In fact  from  (\ref{cotaG}) we have
$$
|\ttt G (t)| \le C    \foral |t|\le \theta,
$$
hence, potentially reducing $\theta$,
from  standard estimates for the minimal surface equation \blue{(see for instance \cite{{gilb}})} we find
\be
|D_t\ttt G(t)| \le   C  \foral |t|\le \theta ,
\label{gr}\ee
with a similar estimate  for $D^2_t\ttt G$,
and in general the same bound for
$D^m_t\ttt G$, $m\ge 2$ in this region. As a conclusion, using also
\equ{boundxd} we obtain
\be
|D^m_tx(t)| + |D^m_t G(t)| \le \frac C{R^{m-1}}  \foral |t|\le \theta R
\label{estimacionG}\ee
for $m=2, 3, \ldots$.
Summarizing, we have established:
\begin{lemma}\label{lem gr1}
There exists a constant $\theta>0$ such that for each $p_0=(x_0,F(x_0))\in \Gamma$ the surface
$\Gamma\cap B(p_0,\theta(1+R))$, $R=|x_0|$,  can be represented as a graph over
$T_{p_0}\Gamma$ of a smooth function $G(t)$. Moreover, denoting $\Gamma\cap B(p_0,\theta(1+R))=\{(t,G(t))\mid t\in T_{p_0}\Gamma\}$,  we have whenever $|t|\leq \theta(1+R)$:
\begin{align}
|D_tG(t)|\leq \frac{c |t|}{1+R},
\label{graph 2}
\end{align}
and
\begin{align}
|D^m_tG(t)|\leq \frac{c}{1+R^{m-1}},
\label{graph 3}
\end{align}
with some universal constant $c$.
\end{lemma}
We want to estimate with higher accuracy derivatives of $G$, in their relation with the approximated minimal graph $\Gamma_0$, $x_9 = F_0(x)$.
We shall establish next that in the situation considered above we also have that $\Gamma_0$ can be represented as the graph of a function $G_0(t)$ over the tangent plane to $\Gamma$ at the point $p_0$, at least in a ball on that plane
of radius $\theta R$ for a sufficiently small, fixed $\theta>0$ and for all large $R$. Below we let
$\nu$ and $n$ denote respective normal vectors to $\Gamma_0$ and $\Gamma$, with the convention  $\nu\cdot n \ge 0$. \blue{For convenience the situation is presented schematically in Figure \ref{fig:one}.
\begin{figure}
\vspace{0.75cm} 
\resizebox{8.0cm}{8.0cm}
{\includegraphics[angle=0]{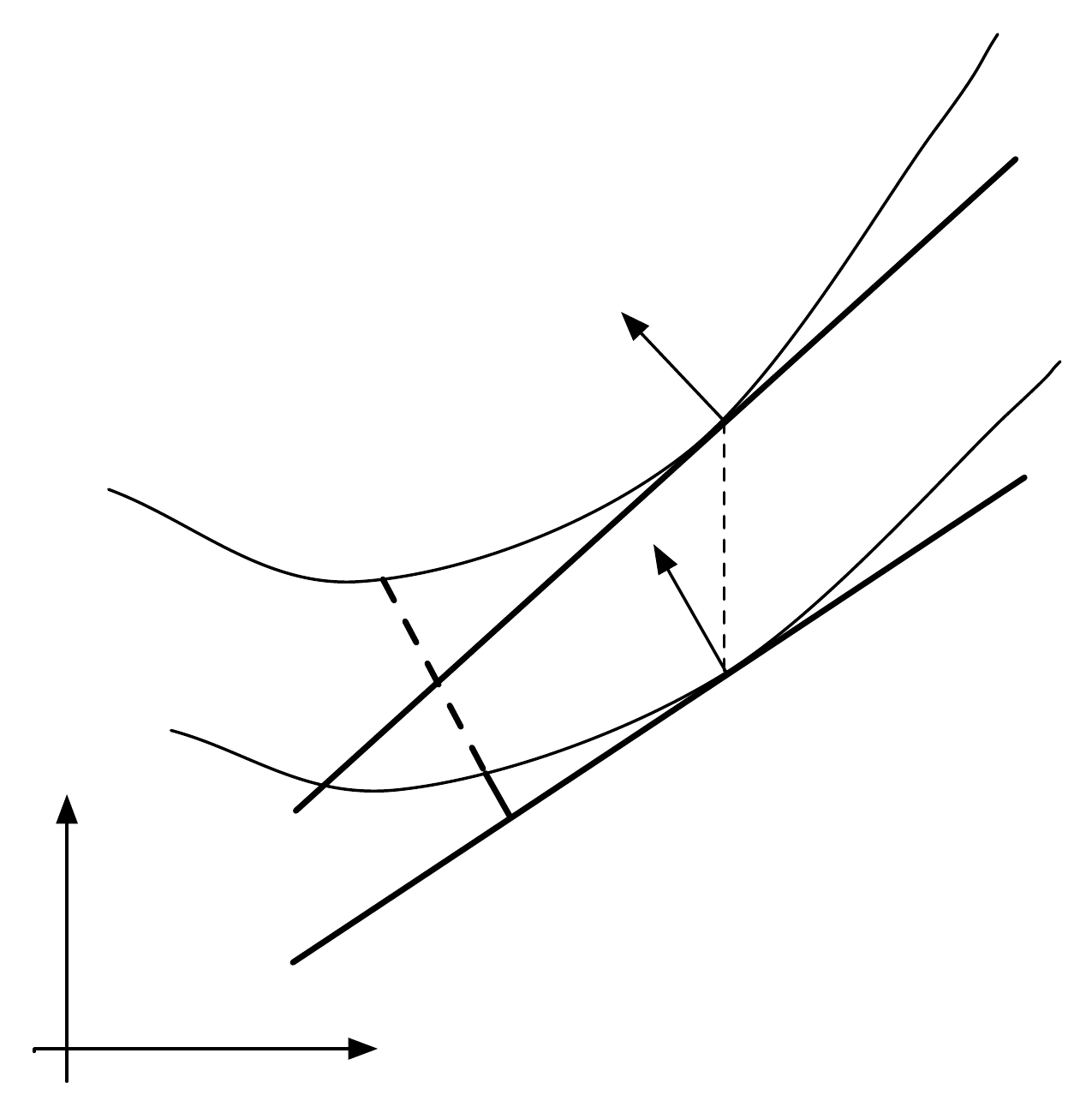}}
\put(-152,3){$\R^8$}
\put(-228,65){$x_9$}
\put(-75,83){$p_0$}
\put(-20,116){$T_{p_0}\Gamma$}
\put(-110,100){$n(p_0)$}
\put(-115,148){$\nu(q_0)$}
\put(-74,135){$q_0$}
\put(-20,180){$T_{q_0}\Gamma_0$}
\put(-8,142){$\Gamma$}
\put(-19,212){$\Gamma_0$}
\put(-148,58){$G(t)$}
\put(-164,88){$G_0(t)$}
\put(-120,52){$t$}
\caption{\small{Local configuration of the two  surfaces $\Gamma$ and $\Gamma_0$.}} \label{fig:one}
\end{figure}
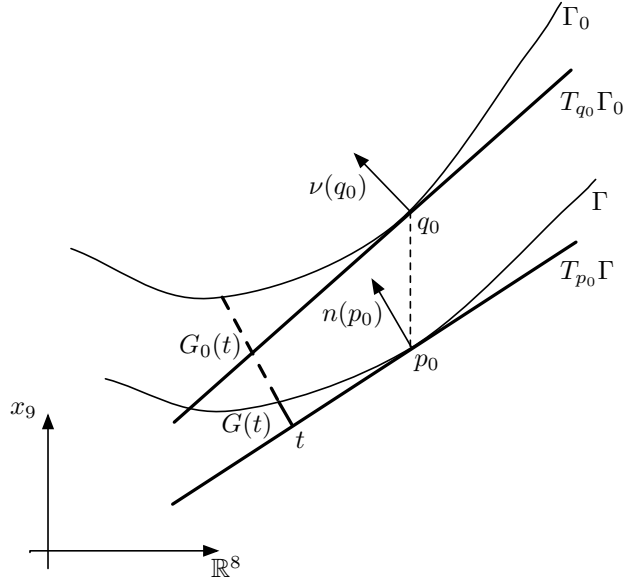}

To prove the above claim \blue{we will  show} that for fixed, sufficiently small $\theta$ we have the estimate
\be
|\nu(q)- n(p_0) |  <  C\theta  \foral    q\in \Gamma_0\cap B(p_0, \theta R) .
\label{n0}\ee

\medskip
Since by Theorem \ref{exist fmc}
$$
F(x) -F_0(x) = O( |x|^{-\sigma}),\quad \mbox{some}\ \sigma\in (0,1),
$$
we have that the points
$p_0 = (x_0, F(x_0))$ and $q_0= (x_0, F_0(x_0))$ satisfy
\be
|p_0 -q_0| \le \frac C{R^\sigma}.
\label{jkl}\ee
Let   $T_{p_0}{\Gamma}$, $T_{ q_0}{\Gamma_0}$ be the corresponding tangent hyperplanes, namely
\begin{align*}
T_ {p_0}{\Gamma}&= \{ z\in \R^9\mid (z- p_0)\cdot n(p_0) = 0\} ,\\
T_{ q_0}{\Gamma_0} &= \{ z\in \R^9\mid  (z- q_0)\cdot \nu(q_0) = 0\}  .
\end{align*}
We assume that  $n(p_0)\cdot \nu(q_0) \ge 0$.  We claim that there is a number $M>0$ such that for all large $R$,
\be
|n(p_0) - \nu(q_0) | \le  \frac {5M} {R} .
\label{jlp}\ee
Let us assume the opposite and let us consider a point $z\in T_{q_0}{\Gamma_0}$ with
$$\theta R>|z- q_0| >  \frac \theta 2  R , $$ with $\theta >0$ as in \equ{cotaG}.
Let us write $\cos\alpha = n(p_0)\cdot \nu(q_0)$ with $0\le\alpha \le \frac \pi 2$. Then, using \equ{jkl} we get
\be
\dist ( z, T_{ p_0}{\Gamma}) \,\ge \, |z-p_0|\sin\alpha\, \ge\,  (\frac \theta 2 R - R^{-\sigma} )|n(p_0) - \nu(q_0) |\ge  {M\theta}   .
\label{inc}\ee
\blue{Let now $\tilde q\in \Gamma_0$ be the point whose projection onto $T_{q_0}\Gamma_0$ is $z$. Point $\tilde q$ is unique by the analog of \equ{cotan} for the surface $\Gamma_0$. Let us denote $\tilde q=(\tilde x, F_0(\tilde x))$. Notice that $|\tilde x|\sim R$. We will also set $\tilde p=(\tilde x, F(\tilde x))\in \Gamma$. 
Since the second fundamental form of the surface $\Gamma_0$ satisfies an estimate similar to the one for $\Gamma$ we may assume, reducing $\theta$ if necessary,
$$
\dist (\tilde q , T_{q_0}\Gamma_0)\le c  \theta.
$$
Now, estimate \equ{cotaG} implies that
$$
\dist (\tilde p, T_{ p_0}{\Gamma})\le c \theta.
$$
 
If $M$ is fixed so that $M\theta$ is sufficiently large, the above two relations and \equ{jkl} are not compatible with  \equ{inc}, indeed we get:
\begin{align*}
M\theta\leq \dist ( z, T_{ p_0}{\Gamma})&\leq \dist(\tilde p,\tilde q)+\dist (\tilde p , T_{p_0}\Gamma_0)+\dist (\tilde q , T_{q_0}\Gamma_0)\\&
\le c  \theta+\dist (\tilde p, \tilde q)\\
& \leq \frac{C}{R^\sigma}+c\theta,
\end{align*}
}
hence \equ{jlp} holds.
Moreover, using estimate \equ{cotan} and the analogous estimate for the variation of $\nu$ we have the validity of  the estimate
$$
|\nu(q) - \nu(q_0) | +  |n(p) - n(p_0) |  <  C\theta\quad   \forall  p\in \Gamma\cap B(p_0, \theta R), \ \ \forall  q\in \Gamma_0\cap B(q_0, \theta R).
$$
\blue{Furthermore, we observe the that  analog of the estimate \equ{cotan} implies  that in the set $\Gamma_0\cap B(q_0, \theta R)$ the distance between  $\Gamma_0$ and its tangent plane at $q_0$  varies by no more that $c\theta$. From this and (\ref{jkl}) and   (\ref{jlp})  the desired conclusion \equ{n0} immediately follows (taking $\theta$ smaller if necessary).  Hence the function $G_0(t)$ is well-defined for $|t| <\theta R$.}

Let  us observe that
 $F_0$ and $G_0$ are linked through the following relation:
\be
\left [ \begin{matrix} \ttt x\\ F_0(\ttt x)    \end{matrix}\right ] = \left [ \begin{matrix} x_0\\ F(x_0)    \end{matrix}\right ] + \Pi t + G_0(t) n(p_0)
\label{relFG0}\ee
 By the implicit function theorem, $\ttt x$ and $G_0(t)$ define differentiable functions of $t$ for $|t|\le \theta R$.
 We shall establish derivative estimates for $G_0$ similar to those found for $G$. We claim that
\be
|D^m_t \ttt x(t)|+  |D^m_t G_0(t)| \le \frac {C}{R^{m-1}}  \foral |t|\le \theta R,
\label{estimacionG0}\ee
for $m=1,2,\ldots$. Differentiation of relation \equ{relFG0} yields
\be
\left [ \begin{matrix} \partial_j \ttt x\\ \nabla F_0(\ttt x) \partial_j \ttt x   \end{matrix}\right ] =   \Pi_j  + \partial_jG_0n(p_0).
\label{relFG1}\ee
Let $q= (\ttt x, F_0(\ttt x))$ and
$$
\nu(q) = \frac 1{\sqrt{1+ |\nn F_0(\ttt x)|^2}} \left [ \begin{matrix}  \nabla F_0(\ttt x)\\ -1    \end{matrix}\right ].
$$
From \equ{relFG1} and the fact that $\nu(q)\cdot n(p_0) \ge c>0$ we then get
$$
|\partial_jG_0(t)| \le C |\Pi_j\cdot \nu(q)| \le C.
$$
Using again relation \equ{relFG1} we also get
$$
|\partial_j\ttt x (t)| \le C .
$$
Let us differentiate again.  Now we get
\be
\left [ \begin{matrix} \partial_{jk} \ttt x\\ \nabla F_0(\ttt x) \partial_{jk} \ttt x   \end{matrix}\right ] +
\left [ \begin{matrix} 0 \\ D^2F_0(\ttt x) [\partial_j \ttt x ,\partial_j \ttt x]  \end{matrix}\right ]
=    \partial_{jk}G_0n(p_0).
\label{relFG2}\ee
Again, taking the  dot product against $\nu(p)$ we  get
$$
|\partial_{jk}G_0(t)| \le C \frac{|D^2F_0(\ttt x)|}{\sqrt{1+ |\nn F_0(\ttt x)|^2}} \le \frac CR
$$
and thus
$$
|\partial_{jk}\ttt x (t)| \le \frac CR  .
$$
Iterating this argument, using that
$$
|D^mF_0(\ttt x)|\le  C R^{3-m}, \quad m=1,2,\ldots
$$
the desired result \equ{estimacionG0} follows.


\medskip
Let us write
$$
G(t) = G_0(t) + h(t).
$$
We will estimate first the size of $h(t)$ in the ball $|t|\le
\theta R$.
We claim that we have
\be
|h(t)|\, \le\, CR^{-1-\sigma} \foral |t|\le \theta R. \label{esh}\ee
First observation we make is that when $t=0$ we have:
\begin{align}
|h(0)|=|G_0(0)|\leq \frac{C}{R^{1+\sigma}}.
\label{esh a}
\end{align}
\blue{To show this let $\tilde x$ be such that 
\begin{align*}
\left [ \begin{matrix} \ttt x\\ F_0(\ttt x)    \end{matrix}\right ] = \left [ \begin{matrix} x_0\\ F(x_0)    \end{matrix}\right ] +  G_0(0) n(p_0),
\end{align*}
and let $\tilde t$ be such that 
\begin{align*}
\left [ \begin{matrix} \ttt x\\ F(\ttt x)    \end{matrix}\right ] = \left [ \begin{matrix} x_0\\ F(x_0)    \end{matrix}\right ] + \Pi \tilde t+ G(\tilde t) n(p_0).
\end{align*}
Comparing these two expressions and using $|F(\tilde x)-F_0(\tilde x)|\sim R^{-\sigma}$ we see that $|\tilde t|\sim R^{-\sigma}$ hence, by (\ref{graph 2}) we get that $|G(\tilde t)|\sim R^{-1-2\sigma}$.
Now multiplying the above relations by $n(p_0)$ and subtracting them we infer   (\ref{esh a})} since by Theorem 4 p.673 and Theorem 5 p. 680 \cite{simon 1}, we have that
\begin{align*}
|n_9(p_0)|=\frac{1}{\sqrt{1+|\nabla F(p_0)|^2}}\leq \frac{C}{R}.
\end{align*}
To prove (\ref{esh}) now we let $p_1 = (x_1,F(x_1))\in \Gamma\cap B(p_0,\theta R)$ so that:
$$
p_1 = p_0 + \Pi t  +  G(t)n(p_0),\quad |t|\le \theta R .
$$
Then $|G(t)-G_0(t)|$ corresponds to the length of the segment
with direction $n (p_0)$ starting at $p_1$, with end on the surface
$\Gamma_0$.
Let $p_2= (x_1, F_0(x_1))$. Then
$$
|p_1 -p_2| \le C R^{-\sigma} .
$$
Let us consider the tangent plane $T_{p_2}{\Gamma_0}$ to $\Gamma_0$ at $p_2$,
with normal $\nu(p_2)$. Then,
$\Gamma_0 \cap B(p_2, CR^{-\sigma}) $ lies  within a distance $O( R^{-1-\sigma})$ from
$T_{p_2}{\Gamma_0}$, more precisely,
$$\Gamma_0 \cap B(p_2, CR^{-\sigma})\subset{\mathcal C}_R, $$ where ${\mathcal C}_R$ is the cylinder
$$ {\mathcal C}_R = \{  \bar z +  s \nu(p_2) \mid   \bar z\in T_{p_2}{\Gamma_0},\ |\bar z - p_2| \le CR^{-\sigma}, \   |s| \le
CR^{-1-\sigma}\}. $$
Using (\ref{esh a}) we may assume that $p_1 \in {\mathcal C}_R$.
 In particular, the line starting from
$p_1$ with direction $n(p_1)$ intersects $\Gamma_0$ inside this
cylinder. Since $n(p_1) \cdot\nu(p_2) \ge c>0$, the length of this
segment is of the same order as the height of the cylinder, and we then get
$$ |G(t)-G_0(t)| \le C R^{-1-\sigma},$$
hence \equ{esh} holds.

\medskip
Next we shall  improve the previous estimate. We claim that we have
\be |D^m_t  h(t)| \le \frac c{R^{m+1+\sigma }} \inn |t| <   \theta  R,
\label{eshm}\ee
for $m=0,1,2,\ldots$. Let us set
$$
\ttt G (t) = \frac 1R G(R t) , \quad \ttt G_0 (t) = \frac 1R G_0(R t),\quad  \ttt h (t) = \frac 1R h(R t).
$$
We  compute (for brevity dropping the subscript in the derivatives):
$$
\sqrt{1+|\nn\ttt G|^2} H[\ttt G] = \Delta\ttt G - \frac {D^2\ttt G\,[\nn\ttt G, \nn\ttt G]}{1+ |\nn\ttt G|^2}=0.
$$
Now,
$$
\frac {D^2\ttt G\,[\nn\ttt G, \nn\ttt G]}{1+ |\nn\ttt G|^2} = \frac {D^2h\,[\nn\ttt G, \nn\ttt G]}{1+ |\nn\ttt G|^2} + \frac {D^2\ttt G_0 \,[\nn\ttt G, \nn\ttt G]}{1+ |\nn\ttt G|^2},
$$
and
$$
\frac {D^2h\,[\nn\ttt G, \nn\ttt G]}{1+ |\nn\ttt G|^2} = \frac {D^2\ttt G_0 \,[\nn\ttt G_0, \nn\ttt G_0]}{1+ |\nn\ttt G|^2} +\frac{ D^2\ttt G_0 \,[2 \nn\ttt G_0 + \nn h, \nn h] }{1+ |\nn\ttt G|^2}.
$$
Furthermore,
\begin{align*}
\frac {D^2\ttt G_0 \,[\nn\ttt G_0, \nn\ttt G_0]}{1+ |\nn\ttt G|^2} &=  \frac {D^2\ttt G_0 \,[\nn\ttt G_0, \nn\ttt G_0]}{1+ |\nn\ttt G_0|^2} \\
&\quad-
\frac {D^2\ttt G_0 \,[\nn\ttt G_0, \nn\ttt G_0]\, ( 2 \nn\ttt G_0 + \nn h )\cdot \nn h  }{(1+ |\nn\ttt G_0|^2) (1+ |\nn\ttt G|^2)}.
\end{align*}
Collecting terms we see that
$h$ satisfies the equation
$$
\Delta h -   \frac {D^2h\,[\nn\ttt G, \nn\ttt G]}{1+ |\nn\ttt G|^2}  + b\cdot\nn h + E = 0, \quad\hbox{in } B(0, \theta),
$$
where
$$
E= \Delta\ttt G_0 - \frac {D^2\ttt G_0 \,[\nn\ttt G_0, \nn\ttt G_0]}{1+ |\nn\ttt G_0 |^2} = \sqrt{1+|\nn\ttt G_0|^2} H(\ttt G_0),
$$
and
$$
b \ =\   -\frac {D^2\ttt G_0 \,[\nn\ttt G_0, \nn\ttt G_0]\, ( 2 \nn\ttt G_0 + \nn h ) }{(1+ |\nn\ttt G_0|^2) (1+ |\nn\ttt G|^2)} +
\frac{ D^2\ttt G_0 \,[2 \nn\ttt G_0 + \nn h] }{1+ |\nn\ttt G|^2} .
$$
Notice that:
$$ |\nn\ttt G(t) | \le  C ,\quad   |\ttt h(t) |\le CR^{-2-\sigma}   \inn |t|<  \theta. $$
Also by (\ref{sub1aa})
the mean curvature of $\Gamma_0$ decays like $R^{-5}$. From
\begin{align*}
|E(t)|   &=   R\Big|\Big(\Delta G_0 - \frac {D^2 G_0 \,[\nn G_0, \nn G_0]}{1+ |\nn G_0 |^2}\Big)(Rt)\Big|\\ &=
R\sqrt{1+|\nn G_0|^2} H[G_0] (Rt)\\
&= R \sqrt{1+|\nn G_0 (Rt)|^2}H[F_0] (\ttt x(Rt)),
\end{align*}
(in the notation of \equ{relFG0})
we then find
\begin{align*}
|E(t)|  = O(R^{-4}),
\end{align*}
and, as a conclusion, reducing $\theta$ if needed,
$$
|D_t \ttt h(t)| \le \frac c{R^{2+\sigma}} \inn |t| < \theta,
$$
so that for $h$ we get accordingly
$$
|D_t h (t)| \le \frac c{R^{2+\sigma }} \inn |t| <   \theta  R.
$$
On the other hand, using \equ{estimacionG0}
we have for instance that
$$
D_tH[G_0](t) = D_xH[F_0] (\ttt x(t)) D_t\ttt x(t)  = O(R^{-6}),
$$
hence
\begin{align*}
|D_t E(t)|  = O(R^{-4}).
\end{align*}
More generally, since
$$
D_t^mH[F_0] (x) = O( |x|^{-5-m}),
$$
we get
$$
D_t^m E(t)= O( R^ {-4}).
$$
This, estimates
\equ{estimacionG0}, \equ{estimacionG}
and standard higher regularity elliptic  estimates  yield
$$
|D^m_t  \ttt h(t)| \le \frac c{R^{2+\sigma}} \inn |t| <   \theta R .
$$
Hence
$$
|D^m_t  h(t)| \le \frac c{R^{m +1 +\sigma }} \inn |t| <   \theta  R.
$$
for $m\ge 1$, as desired.

\bigskip
Now we will derive some consequences of the above estimates. First we will consider the  Fermi coordinates near $\Gamma$.  Let $x=(x',x_9)\in \R^9$ be a point in a neighborhood of $\Gamma$ and let $p(x)$ be its projection on $\Gamma$ in the direction of $n(p)$. If $\mbox{dist}\,(x,\Gamma)$ is sufficiently small then $p(x)$ is unique and we can write:
\begin{align}
x=p+z n(p),
\label{coord 1}
\end{align}
where $z=z(x)$. These $(p,z)$ are  the Fermi coordinates of $x$. \blue{They are defined as long as  the function $x\mapsto (p,z)$ is invertible. We claim  that this the case, and that the Fermi  coordinates are well defined as long as
\begin{align}
{|z|\leq \theta|A_\Gamma(p)|^{-1}},
\label{coord 2}
\end{align}
where $R=|\Pi_{\R^8}(p)|$ is the the distance of the projection of $p$ onto $\R^8$ from the origin, and $\theta$ is chosen to be a small number. 
Because of the symmetry of the surface $\Gamma$, it is enough to consider the situation in which, for certain $x=(x',x_9)$ such that $x'\in T$ we have the existence of two different points $p_1, p_2\in \Gamma\cap T\times [0,\infty)$ such that
\begin{align}
x=p_i+ z n(p_i), \quad i=1,2,
\label{coord 2a}
\end{align} 
with $z$ satisfying (\ref{coord 2}). We may assume that $|\Pi_{\R^8}(p_1)|=R_1$ is large. Then it follows:
\begin{align}
|p_1-p_2|\leq |z||n(p_1)-n(p_2)|\leq \theta |A_{\Gamma}(p_1)|^{-1}.
\label{coord 2b}
\end{align}
In the portion  of $\Gamma$ where (\ref{coord 2b}) holds we have in fact:
\begin{align}
\begin{aligned}
|p_1-p_2|&\leq |z||n(p_1)-n(p_2)|\\
&\leq \theta |A_{\Gamma}(p_1)|^{-1}\sup_{|p_1-p|\leq \theta |A_{\Gamma}(p_1)|}|A_\Gamma(p)||p_1-p_2|.
\\
&\leq C\theta\frac{R_1+1}{R_1}|p_1-p_2|.
\end{aligned}
\label{coord 2c}
\end{align}
Now the claim  follows if we take  $\theta>0$ to be sufficiently small. 
From (\ref{coord 2}) we get that the Fermi coordinates are well defined in an expanding neighborhood ${\mathcal U}_{\theta_0}$ of $\Gamma_\alpha$ defined in (\ref{def ndelta1}).}

Second we will compute the derivative of $z(x)$ with respect to $x_9$. Since $p\in \Gamma$ therefore we can write
\begin{align*}
p=(y', F(y')), \quad y'\in \R^8.
\end{align*}
Then, taking the derivative with respect to $x_9$ and using (\ref{coord 1}) we get:
\begin{align}
e_9=(\partial_{x_9}y',\nabla F(y')\cdot\partial_{x_9}y')+\partial_{x_9} z n(p), \quad \mbox{at}\  z=0.
\label{coord 3}
\end{align}
Notice that
\begin{align*}
\tau(p)= (\partial_{x_9}y',\nabla F(y')\cdot\partial_{x_9}y')\in T_p\Gamma.
\end{align*}
Multiplying (\ref{coord 3}) by $\nu(q)$, where $q\in \Gamma_0$ is a point on the segment in the direction of $n(p)$ and $\nu(q)$ is the unit normal at this point  we get
\begin{align}
\frac{-1}{\sqrt{1+|\nabla F_0(q)|^2}}=\tau(p)\cdot \nu(q)+\partial_{x_9} zn(p)\cdot \nu(q).
\label{coord 4}
\end{align}
From (\ref{coord 3}) we get as well:
\begin{align}
|\tau(p)|\leq 1+|\partial_{x_9} z|.
\label{coord 5}
\end{align}
Now let $(t, G(t))$ and $(t, G_0(t))$ be the local coordinates centered respectively at $p\in\Gamma$ and $q\in\Gamma_0$. From the above discussion we conclude that any two unit tangent vectors in the same direction, say $t_i$ differ by a factor proportional to $D_th(t)$. This means
\begin{align*}
|\tau(p)\cdot\nu(q)|\leq C|D_t h(t)|\leq \frac{C}{R^{2+\sigma}},
\end{align*}
hence using (\ref{coord 4}) we get 
\begin{align*}
|\partial_{x_9} z||n(p)\cdot \nu(q)|= \frac{1}{\sqrt{1+|\nabla F_0(q)|^2}}+O\big(\frac{1}{R^{2+\sigma}}\big).
\end{align*}
Since $|n(p)\cdot\nu(q)|>c>0$ uniformly, it follows:
\begin{align}
\frac{C^{-1}}{1+R^2}\leq |\partial_{x_9} z|\leq \frac{C}{1+R^2}.
\label{coord 6}
\end{align}
Notice that as a byproduct we get
\begin{align*}
n_9(p)=\partial_{x_9} z,
\end{align*}
hence
\begin{align}
\frac{C^{-1}}{1+R^2}\leq \frac{1}{\sqrt{1+|\nabla F(p)|^2}}\leq \frac{C}{1+R^2}.
\label{coord 6a}
\end{align}
which is a special case of the estimate in Theorem 5 p. 679 in \cite{simon 1}.

Next we will discuss the expressions for the Laplace-Beltrami operators  on $\Gamma$ and $\Gamma_0$ in terms of the local coordinates associated  with $G(t)$ and $G_0(t)$. Let us observe that the metric tensor  $\tg$ in $\Gamma\cap B(p_0,\theta R)$ satisfies the following relation:
\begin{align}
\tg_{ij}=\delta_{ij}+\partial_i G(t)\partial_j G(t)=\delta_{ij}+O(|t|^2R^{-2}), \quad |t|\leq \theta R,
\label{coord 7}
\end{align}
where $\partial_i=\partial_{t_i}$. Similar estimates hold  for the metric tensor $\tg_0$ on the surface $\Gamma_0$ expressed in the same local coordinates. In fact we have:
\begin{align}
\begin{aligned}
\tg_{0,ij}&=\delta_{ij}+\partial_i G_0(t)\partial_j G_0(t)\\
&=\tg_{ij}-\partial_i G(t)\partial_j h(t)- \partial_j G(t)\partial_i h(t)+\partial_i h(t)\partial_j h(t)
\\
&=\tg_{ij}+|t|O(R^{-3-\sigma}).
\end{aligned}
\label{coord 8}
\end{align}
Let now $f$ be a $C^2$ function defined on $\Gamma$ in a neighborhood $p_0$. We can identify this function with a function $\tilde f$ on $\Gamma_0$ through the change of variables :
\begin{align}
\tilde f(\tilde x)=f(t), \quad \tilde x=\tilde x(t).
\label{coord 9}
\end{align}
Then using (\ref{coord 8}) we have
\begin{align}
|\Delta_\Gamma f-\Delta_{\Gamma_0} \tilde f|\leq \frac{C|\nabla^2_{\Gamma} f|}{R^{2+\sigma}}+\frac{C|\nabla_{\Gamma} f|}{R^{3+\sigma}},
\label{coord 10}
\end{align}
as long as $|t|\leq \theta R$. In the sequel we will denote functions on $\Gamma$  and on
$\Gamma_0$  by the same symbol taking into account  (\ref{coord 9}).

Finally, let us consider the second fundamental form on $\Gamma$,  $A_\Gamma$ and the second fundamental form on $\Gamma_0$ denoted by $A_{\Gamma_0}$. We observe that in the local coordinates associated with the graph $G(t)$ over the tangent space we have:
\begin{align*}
|A_\Gamma(t)|=|D^2_t G(t)|, \quad\mbox{at}\  t=0.
\end{align*}
Thus
\begin{align}
\begin{aligned}
|A_{\Gamma_0}(t)|&=|A_\Gamma(t)|+O(|D^2_t h(t)|)\\
&=
|A_{\Gamma}(t)|+O(R^{-3-\sigma}),
\quad \mbox{at}\  t=0.
\end{aligned}
\label{coord 11}
\end{align}
It follows:
\begin{align}
|A_{\Gamma}-A_{\Gamma_0}|\leq \frac{C}{1+R^{3+\sigma}}.
\label{coord 12}
\end{align}

\setcounter{equation}{0}
 \section{Linear theory }
In this section we will consider the basic linearized operator and we will derive a solvability theory for the operator
 \begin{align}
 \cL(\phi)=\Delta\phi+f'(\ww)\phi,
\label{def l2}
\end{align}
already defined in (\ref{def l1}).

\subsection{Nondegeneracy of the approximate solution}
To begin with we will review some basic facts about the one dimensional version  of $\cL$.
By $w$ we will denote the heteroclinic solution to $w'' + w-w^3 =0$ such that $w(0)=0$, $w(\pm\infty)=\pm 1$ namely
$$
w(z) =\tanh\left ( \frac z{\sqrt{2}}\right ).$$
Consider the one-dimensional linear operator
 $$
L_0(\phi)=\phi_{zz}+f'(w)\phi, \quad f'(w)= 1-3w^2.
$$
We recall some well known facts about $L_0$. First notice that $L_0(w_z)=0$.
Second, writing  $\phi = w_z\psi$ we get that
\begin{align*}
L_0(\phi) &=  L_0(\psi w_z)\\
&
= w_z\psi_{zz}+ 2\psi_z w_{zz} +  w_{zzz}\psi + f'(w)w_z \psi \\
&= w_z^{-1}\,(\, w_z^2\psi_z)_z ,
\end{align*}
hence assuming that $\phi(z)$ and its derivative  decay fast enough as $ |z|\to +\infty$, we get the identity
$$
\int_\R L_0(\phi) \phi \, dz \, =\,
\int_\R [|\phi_z|^2 - f'(w)\phi^2]\,dz\, =\,  \int_\R  w_z^2|\psi_z|^2\, dz .
$$
Since $f'(w)\to -2$ as $|z|\to +\infty$ guarantees exponential decay of any bounded solution of
$L_0(\phi) =0$ therefore
any such solution must be of the form $\phi = Cw_z$, $C\in \R$.
Next we set, for $\phi(y,z)$ defined in $\R^m\times\R$,  (where $m=8$ in our case):
$$
L(\phi) = \Delta_y\phi + \phi_{zz} + f'(w)\phi = \Delta_y\phi + L_0(\phi) .
$$
The
equation $L(\phi)=0$,  has the obvious bounded solution $\phi(y,z) = w_z(z)$. Less obvious, but also true, is the converse:
\begin{lemma}\label{lemma l1}
Let $\phi$ be a bounded, smooth solution of the problem
\begin{equation}
L(\phi) = 0 \quad\hbox{in } \R^m\times \R.
\label{l3}\end{equation}
Then $\phi(y,z) = Cw_z(z) $ for some $C\in \R$.
\end{lemma}
\proof
Let $\phi$ be a bounded solution of equation \equ{l3}. We claim that $\phi$ has exponential decay in $z$, uniform in $y$.
Let $0<\sigma<1$ and let us fix $z_0 >0$ such that for all $|z| > z_0$ we have that
$$
f'(w) <  -2 + \sigma^2
$$
In addition, consider the following function
$$
g_\delta (z,y) =   e^{-\sigma(|z|-z_0)}  + \delta  \sum_{i=1}^m \cosh( \sigma y_i) .
$$
for $\delta >0$.
Then for $|z|>z_0$ we get that
$$
L(g_\delta ) =  \blue{\sigma^2 g_{{\delta}} +  f'(w)g_{{\delta}},}
$$
so that if $\sigma >0$ is fixed small enough,
we get
$$
L(g_\delta ) < 0 \quad \hbox{if } |z| >z_0 .
$$
As a conclusion, using maximum principle, we get
$$
|\phi| \le \|\phi\|_\infty \, g_\delta \quad \hbox{if } |z| >z_0 ,
$$
and letting $\delta\to 0$ we then get
$$
|\phi(y,z)| \ \le\ C\|\phi\|_\infty   e^{-\sigma|z|}  \quad \hbox{if } |z| >z_0 \ .
$$
Let us  observe the following fact: the function
$$
\ttt \phi(y,z) =
\phi(y,z) - \left ( \int_\R w_\zeta (\zeta)\, \phi(y,\zeta)\, d \zeta \right )\, \frac{ w_z(z)} {\int_\R w_\zeta^2}
$$
also satisfies $L(\ttt \phi ) = 0$ and, in addition,
\be
\int_\R   w_z (z)\, \blue{\tilde \phi}(y,z)\, d z  = 0\foral y .
\label{orti}\ee
In view of the above discussion, it turns out that the function
$$
\vp (y) := \int_\R \ttt \phi^2(y,z)\, dz
$$
is well defined. In fact so are its first and second derivatives by elliptic regularity \blue{theory applied to}  $\phi$, and differentiation under the
integral sign is thus justified. Now, observe that
$$
\Delta_y \vp (y) = 2 \int_\R \Delta_y \ttt \phi \cdot \ttt \phi \, dz  \red{+}2\int_\R |\nabla_y \ttt \phi |^2
$$
and hence
\begin{align}
\begin{aligned}
0 &= \int_\R  (L_0\ttt \phi \cdot \ttt \phi )\\
& =
\blue{\frac{1}{2}\Delta_y \vp
- \int_\R  |\nabla_y \ttt \phi |^2
 \, dz  -
\int_\R (\, |\ttt \phi _z|^2 - f'(w)\ttt \phi ^2\,)\, dz\, .}
\end{aligned}
\label{ew}
\end{align}

Let us observe that because of relation \equ{orti}, we have that
$$
\int_\R (\, |\ttt\phi _z|^2 - f'(w)\ttt \phi ^2\,)\, dz\, \ge \gamma \vp .
$$
It follows then that

$$\Delta_y \vp   -\gamma \vp   \ge 0 $$
Since $\vp$ is bounded, from maximum principle we find that
$\vp$ must be identically equal to zero.
But this means
\be
\phi(y,z) = \left ( \int_\R w_\zeta (\zeta)\, \phi(y,\zeta)\, d \zeta \right )\, \frac{ w_z(z)} {\int_\R w_\zeta^2}. \label{uuu}\ee
Then the bounded function
$$
g(y) = \int_\R w_\zeta (\zeta)\, \phi(y,\zeta)\, d \zeta
$$
satisfies the equation
\be
\Delta_y g = 0, \quad \mbox{in}\ \R^m.
\label{f}\ee
Liouville's theorem then implies that $g\equiv$ constant. Relation \equ{uuu}
 then yields
 $\phi = Cw_z(z)$ for some $C$.
This concludes the proof.\qed
%

\subsection{Linearized problem near $\Gamma_\alpha$}
Let $\theta_0>0$ be a number such that the Fermi coordinates are well defined for $x\in \R^9$ satisfying  $d(x,\Gamma)<\theta_0$.  \blue{Here $\theta_0>0$ may be taken to be the same as in the definition of the expanding neighborhood ${\mathcal U}_{\theta_0}$, see   (\ref{def ndelta1}).}  We will define the $\delta$-neighborhood of $\Gamma_\alpha$ to be
\begin{align*}
\cN_\delta=\Big\{x\in \R^9\mid d(x,\Gamma_\alpha)<\frac{\delta}{\alpha}\Big\}.
\end{align*}
\blue{Now,  we  let $\delta>0$ be such that $4\delta<\theta_0$ and consider neighborhood of the form $\cN_{4\delta}$. Observe that with this $\delta$ fixed  the approximate solution $\ww$ defined in  (\ref{pre 7}) satisfies $\ww(x)=w(z-h_\alpha(y))$, where $(y,z)$ are the Fermi coordinates of $x\in \cN_{4\delta}$.} 
We will also denote
\begin{align*}
\Gamma_{\alpha,z}= \Big\{x\in \R^9\mid d(x,\Gamma_\alpha)=z\Big\}, \quad
|z|<\frac{4\delta}{\alpha}.
\end{align*}
In $\cN_{4\delta}$ we can write the Laplacian in the local Fermi coordinates:
\begin{align}
\Delta=\Delta_{\Gamma_{\alpha, z}}+\partial_z^2-H_{\Gamma_{\alpha, z}}\partial_z,
\label{l4}
\end{align}
where $H_{\Gamma_{\alpha,z}}$ denotes the mean curvature  of the surface $\Gamma_{\alpha,z}$.
Expression (\ref{l4}) is valid only in $\cN_{4\delta}$ however it is convenient to extend it in such a way that it makes sense for all $z\in \R$. To this end let $\eta(\tau)$ be a smooth cut-off function with $\eta(\tau)=1$, for $|\tau|< 1$ and $\eta(\tau)=0$, $|\tau|>2$, and let us denote:
\begin{align*}
\eta_{\delta}^\alpha=\eta\Big(\frac{\alpha z}{\delta}\Big).\
\end{align*}
Then we define $\A_\delta^\alpha$ to be the following operator
\begin{align}
\A_\delta^\alpha=\eta_{\delta}^\alpha(\Delta_{\Gamma_{\alpha, z}}+\partial_z^2-H_{\Gamma_{\alpha, z}}\partial_z)+(1-\eta_{\delta}^\alpha)(\Delta_{\Gamma_{\alpha}}+\partial_z^2).
\label{l5}
\end{align}
This operator is defined in the set $\Gamma_\alpha\times\R$ and not just in the set $\cN_{4\delta}$.  We \blue{notice that  $\Gamma_\alpha\times\R$ can be parametrized by  $\R^9$ and it is equipped with the natural product metric.} In the sequel we will write
\begin{align*}
\A^\alpha_\delta(\phi)&=\Delta_{\Gamma_{\alpha}}\phi+\partial_z^2\phi+\cB^\alpha_\delta(\phi), \quad\mbox{where} \\
\cB^\alpha_\delta(\phi)&=\eta_{\delta}^\alpha(\Delta_{\Gamma_{\alpha, z}}-\Delta_{\Gamma_\alpha}-H_{\Gamma_{\alpha, z}}\partial_z)(\phi),
\end{align*}
We will introduce now some norms  for functions defined in $\blue{\Gamma_\alpha\times\R}$. By $dV_{\Gamma_\alpha}$ we will denote the volume element of $\Gamma_\alpha$ and we say that $f\in L^p_{loc}(\Gamma_\alpha)$, $1<p<\infty$, if
\begin{align*}
\int_{\Gamma_\alpha\cap K} |f|^p\,dV_{\Gamma_\alpha} <\infty,
\quad \forall K\subset \R^9,\  {\mbox compact}.
\end{align*}
Similarly we define $L^\infty_{loc}(\Gamma_\alpha)$ to be the set of locally bounded functions on $\Gamma_\alpha$.  In $L^p_{loc}(\Gamma_\alpha\times\R)$, $1<p< \infty$ we will introduce the following weighted norms:
\begin{align*}
\blue{\|g\|_{p,\sigma} :=\ \sup_{(y,z)\in \Gamma_\alpha\times\R} e^{\sigma |z|} \,\| g\|_{L^p\big(\Gamma_\alpha\cap B(y,1)\times (z-1,z+1)\big)}},
\end{align*}
where $\sigma$ is such that  $0<\sigma <\sqrt{2}$. In the sequel  $\sigma$ will be  taken \blue{to be smaller as necessary but always $\alpha$ independent}.  Similarly, we define
\begin{align*}
\|g\|_{\infty,\sigma} :=\ \sup_{(y,z)\in \Gamma_\alpha\times\R} e^{\sigma |z|} |g (y,z)| , \quad g\in L^\infty_{loc}(\Gamma_\alpha\times\R).
\end{align*}
\blue{Observe that these definitions are consistent in the sense that if we take formally  $p=\infty$ in the definition of the norm $\|\cdot\|_{p,\sigma}$ then the resulting norm will be equivalent to the $\|\cdot\|_{\infty,\sigma}$ we have actually defined.}

Our next goal is to  establish a solvability theory  for the following problem:
\begin{align}
\begin{aligned}
\Delta_{\Gamma_\alpha}\phi+B_\alpha(\phi)+\partial_z^2\phi+f'(w)\phi&=g+c_\alpha w_z, \quad \mbox{in}\ \Gamma_\alpha\times\R,\\
\int_{\R} \phi(y,z) w_z(z)\,dz&=0, \quad \forall y\in \Gamma_\alpha.
\end{aligned}
\label{l6}
\end{align}
where $B_\alpha(\phi)$ is a second order differential operator of the form:
\begin{align}
B_\alpha(\phi)=\nabla_{\Gamma_\alpha}\phi_z\cdot {\bf b}_{\alpha 1}+\phi_{zz}b_{\alpha 2}.
\label{l6a}
\end{align}
We assume additionally that ${\bf b}_{\alpha 1}:\Gamma_\alpha\to T\Gamma_\alpha$ and $b_{\alpha 2}:\Gamma_\alpha\to \R$ are continuous functions such that
\begin{align}
\|{\bf b}_{\alpha 1}\|_{L^\infty(\Gamma_\alpha)}+\|{b}_{\alpha 2}\|_{L^\infty(\Gamma_\alpha)}\leq C\alpha.
\label{l6b}
\end{align}
\blue{These conditions arise in a natural way as can be seen from  (\ref{l17}) and the argument that follows. } 
We observe that function $c_\alpha:\Gamma_\alpha\to \R$ in (\ref{l6}) is a parameter to be determined using the orthogonality condition.
We will show the following  a priori estimate:

\begin{lemma} \label{ap1}
There exists  $C>0$ such that for all sufficiently small $\alpha$ and all  $g\in L^p_{loc}(\Gamma_\alpha\times\R)$, $9<p\leq \infty$  any solution
 $\phi$ of problem $\equ{l6}$ with $\|\phi\|_{\infty,\sigma }<+\infty$ satisfies
\begin{align}
\|D^2\phi\|_{p,\sigma} + \|D\phi\|_{\infty,\sigma } + \|\phi\|_{\infty,\sigma } \ \le\ C\, \|g\|_{p,\sigma},
\label{l7}
\end{align}
\end{lemma}
\proof
A  remark we make is that multiplying the equation by
$w_z(z)$, integrating by parts, and using the orthogonality assumption we readily get
\begin{align*}
 c_\alpha(y) \int_\R w_z^2 \ =\ \int_\R g(y,z)w_z(z)\, dz-\int_\R B_\alpha(\phi)w_z\,dz\ , \quad \forall y\in \Gamma_\alpha.
\end{align*}
In particular, we easily check that the function $c_\alpha(y)w_z(z)$ satisfies
\begin{align*}
\|c_\alpha w_z\|_{p,\sigma} \le C\|g\|_{p,\sigma}+O(\alpha)\|D^2\phi\|_{p,\sigma},
\end{align*}
hence for the purpose of the proof we do not lose generality in assuming simply that
$c_\alpha\equiv 0$.
Next, we will prove  the existence of $C$ for which
\be
\|\phi\|_{\infty,\sigma } \ \le\ C\, \|g\|_{p,\sigma}.
\label{ass1}\ee
To establish this assertion we argue by contradiction.
Let us assume  that we have  sequences $\{\alpha_n\}$, $\{g_n\}$, $\{\phi_n\}$ for which problem \equ{l6} is satisfied (now with $c_\alpha\equiv0$),  and
\begin{align}
\|\phi_n \|_{\infty,\sigma }= 1, \quad \|g_n\|_{p,\sigma} \to 0 ,\quad \alpha_n \to 0 .
\label{l8}
\end{align}
This means that there exists a sequence
$(y_n, z_n) \in \Gamma_{\alpha_n}\times \R$  such that
\begin{align*}
 e^{\sigma |z_n|}|\phi_n (y_n, z_n) |\to 1.
\end{align*}
We consider two cases:
\begin{enumerate}
\item Sequence $|z_n|$ is bounded.
\item $\lim_{n\to \infty}|z_n|=\infty$.
\end{enumerate}

\noindent
{\it Case 1.} From Lemma \ref{lem gr1} we know that there exists a $\rho>0$ such that for each $n$ the surface $\Gamma_{\alpha_n}\cap B(y_n, \rho\alpha_n^{-1})$ can be represented as a graph of a smooth function $G_n: T_{y_n}\Gamma_{\alpha_n}\to \R$ and that moreover
\begin{align}
|G_n(t)|\leq c\alpha_n |t|, \quad |t|\leq \rho\alpha_n^{-1}, t\in T_{y_n}\Gamma_{\alpha_n}.
\label{l9}
\end{align}
In local coordinates of $\Gamma_{\alpha_n}$ given  by the graph of $G_n$ we can write
\begin{align*}
\phi_n(y, z)=\tilde \phi_n(t, z), \quad y\in \Gamma_{\alpha_n}\cap B(y_n, \rho\alpha_n^{-1}),\quad  y=(t, G_n(t)).
\end{align*}
Using estimate (\ref{graph 2}) we get that the metric tensor on $\Gamma_{\alpha_n}\cap B(y_n, \rho\alpha_n^{-1})$ expressed in terms of the coordinates $t$ satisfies:
\begin{align*}
g_n(t)=I+O(\alpha^2_n),
\end{align*}
where $I$ is the identity matrix. Therefore, over compacts of $\R^8\equiv T_{y_n}\Gamma_{\alpha_n}$, we have that
\begin{align*}
\sqrt{\det g_n(t)}\to 1, \quad g^{ij}_n(t)\to \delta_{ij},
\end{align*}
uniformly, together with its derivatives.  Writing now the equation satisfied by $\tilde\phi_n$ in the local coordinates we get that
\begin{align}
\begin{aligned}
&\frac{1}{\sqrt{\det g_n}}\partial_{t_i}(g_n^{ij}\sqrt{\det g_n}\partial_{t_j}\tilde \phi_n)+g_n^{ij}{\bf b}_{n 1, j}\partial^2_{t_i z}\tilde\phi_{n}+b_{n 2}\partial^2_z\tilde\phi_n+\partial_z^2\tilde\phi_n+f'(w)\tilde\phi_n=\tilde g_n, \\
&\qquad\qquad\qquad \mbox{in}\  B(0, \rho\alpha_n^{-1})\times \R,\\
&\int_\R\tilde\phi_n(t,z)w_z\,dz=0, \quad \mbox{in}\ B(0, \rho\alpha_n^{-1}).
\end{aligned}
\label{l10}
\end{align}
We have that
\begin{align*}
\tilde g_n&\to 0, \quad \mbox{in}\  L^p_{loc}(\R^8\times \R),\\
|{\bf b}_{n1}|+|b_{n2}|&\to 0,\ \mbox{uniformly over compacts},\\
0<c&\leq |\tilde\phi_n(t,z)|\leq C,\quad  \mbox{in}\  B(0, \rho\alpha_n^{-1})\times \R,
\end{align*}
and by standard elliptic regularity we get that
\begin{align*}
\tilde\phi_n\to \tilde \phi\neq 0,
\end{align*}
over compacts of $\R^8\times \R$. Moreover, $\tilde\phi$ is a bounded, non-zero  solution of
\begin{align*}
\Delta_t\tilde\phi+\partial_z^2\tilde\phi+f'(w)\tilde\phi&=0, \quad\mbox{in}\ \R^8\times \R,
\\
\int_\R\tilde\phi(t,z)w_z\,dz&=0, \quad \mbox{in}\  \R^8,
\end{align*}
which, by Lemma \ref{lemma l1}, implies that $\tilde\phi\equiv 0$, a contradiction.

\noindent
{\it Case 2.}
In this case the proof is similar, except that we define
\begin{align*}
\tilde\phi_n(t,z)=\blue{\cosh\big(\sigma (z_n+z)\big)\phi_n(y,z_n+z)}.
\end{align*}
Then  a similar limiting argument can be used to show that $\tilde\phi$ satisfies
\begin{align*}
\Delta_t\tilde\phi+\partial_z^2\tilde\phi+\sigma a_1(z)\partial_z\tilde\phi-(2-\sigma^2 a_2(z))\tilde\phi&=0, \quad\mbox{in}\ \R^8\times \R,
\end{align*}
where $a_j(z)$ are bounded functions. Then, if $\sigma$ is sufficiently small, by maximum principle, we get that $\tilde\phi\equiv 0$.
We have reached a contradiction again and the proof of estimate
\equ{ass1}is concluded.

It only remains the estimate for first and second derivatives.
This is  immediate from local elliptic $L^p$-theory and the already obtained weighted $L^\infty$ estimate for $\phi$.
The proof is concluded.
\qed


The next lemma establishes existence of a unique solution of problem \equ{l6} as in Lemma \ref{ap1}.
\begin{lemma}\label{ex1}
For all sufficiently small $\alpha$ and all
 $g$  with $\|g\|_{p,\sigma} <+\infty$, $9<p\leq \infty$  there exists a unique solution
 $\phi$ of problem $\equ{l6}$ with $\|\phi\|_{\infty,\sigma }<+\infty$. This solution satisfies
$$
\|D^2\phi\|_{p,\sigma} + \|D\phi\|_{\infty,\sigma } + \|\phi\|_{\infty,\sigma } \ \le\ C\, \|g\|_{p,\sigma}.
$$
\end{lemma}
\proof
\blue{We assume initially that $\|g\|_{\infty,\sigma}<\infty$. The general case of $g$ with $\|g\|_{p,\sigma}<\infty$,  $9<p<\infty$ will follow by taking a suitable sequence of functions approximating $g$.} 
First we will show that there the assertion of the Lemma holds for the following problem:
\begin{align}
\begin{aligned}
\Delta_{\Gamma_{\alpha}}\phi+\partial_z^2\phi+f'(w)\phi&=g+c_\alpha w_z, \quad \mbox{in}\
\Gamma_\alpha\times\R,\\
\int_\R \phi(y,z)\blue{w_z(z)}\,dz&=0, \quad\mbox{in}\ \Gamma_\alpha.
\end{aligned}
\label{l11}
\end{align}
We shall argue by approximations. Let us replace $g$ by
$$
g_R(y,z) = \left\{\begin{matrix}
g(y,z) &\hbox{in}&  B_R(0)\cap (\Gamma_\alpha\times\R) \\
0 &\hbox{in}&  (\Gamma_\alpha\times\R)\setminus B_R(0) \\
\end{matrix} \right.
$$
Then   problem (\ref{l11}) corresponds to
finding a minimizer of the functional
$$
J(\blue{\phi}) = \frac 12 \int_{\Gamma_\alpha\times\R} [|\nabla_{\Gamma_\alpha}\phi|^2+ |\phi_z|^2 - f'(w)\phi^2]\,dV_\alpha dz -
 \int_{\Gamma_\alpha\times\R} g_R \phi\,dV_\alpha dz.
$$
in the space  $H$ of all functions $\phi\in H^1_{loc}(\Gamma_\alpha\times\R)$ for which
$$
\|\phi\|_H^2:= \int_{\Gamma_\alpha\times\R} (|\nabla_{\Gamma_\alpha}\phi|^2+|\phi_z|^2 + \phi^2)\,dV_\alpha dz <+\infty
$$
and
$$
\int_\R \phi(y,z) w_z(z)\, dz\, = 0, \quad \forall y\in \Gamma_\alpha.
$$
Since  the orthogonality assumption implies that
$$
\int_{\Gamma_\alpha\times\R} [(\phi_z^2 - f'(w)\phi^2)]\,dV_\alpha dz \ge \gamma \int_{\Gamma_\alpha\times\R} \phi^2\,
dV_\alpha dz, \quad \mbox{with some}\ \gamma>0,
$$
therefore $J$ is coercive in $H$ by a standard argument. Consequently, the functional has a minimizer  $\phi_R$ in $H$.
Since the truncation $g_R$ has compact support,
$\phi_R$ can be approximated by minimizing $J$ on the set of functions in $H$ which vanish outside a ball $B_n(0)$ with $n\gg R$.
Calling $\phi_{R,n}$ this minimizer, we see that $\phi_{R,n}$  approaches $\phi_R$ in the sense of the $H$-norm.
Applying elliptic estimates to the equation
satisfied by $\phi_{R,n}$, we obtain that $\phi_{R,n}$ is in fact locally bounded, uniformly in $n$.
Outside the support of $g_R$ the equation
$$
\Delta_{\Gamma_\alpha}\phi_{R,n}+ \partial_z^2\phi_{R,n} + f'(w)\phi_{R,n} = 0
$$
is satisfied. Now, we recall the fact that
$$
\bar h_\alpha (x) = \frac 1{\sqrt{1+|\nn F_\alpha|^2}}, \quad F_\alpha(x)=\frac{1}{\alpha}F(\alpha x), \quad x\in \R^8,
$$
is a positive supersolution for the Laplace-Beltrami operator, indeed
$$
-\Delta_{\Gamma_\alpha} \bar h_\alpha = |A_\alpha|^2\bar h_{\alpha} >0 .$$
Observe  that $\psi_\alpha = \blue{e^{\,-\sigma(|z|-|z_0|)}}\bar h_\alpha(x)$
satisfies
$$
\Delta_{\Gamma_\alpha}\psi_\alpha+\partial^2_z\psi_{\alpha} + f'(w)\psi_\alpha
 \le \blue {[\sigma^2+f'(w)]\psi_\alpha}-|A_{\Gamma_\alpha}|^2\psi_\alpha<0
$$
\blue{for $|z|>|z_0|$ and $|z_0|$ large.} 
Using comparison  principle we then obtain that
$$ |\phi_{R,n}| \le C_R {\blue{\psi_\alpha},}
\quad \mbox{in}\  B_n(0)\subset \Gamma_\alpha\times\R,
$$ where  constant $C_R$ depends on  $\|g\|_{\blue{\infty},\sigma}$, \blue{the uniform bound on $\|\phi_{R,n}\|_{\infty, 0}$,}  and $R$ only. Letting $n\to\infty$ we obtain that the same estimate is valid for
$\phi_R$. We conclude that,
$\|\phi_R\|_{\infty,\sigma} < +\infty$. It follows from Lemma \ref{ap1} that we have the uniform control
$$
\blue{\|D^2\phi_R\|_{p,\sigma}}+\|D \phi_R\|_{\infty,\sigma} + \|\phi_R\|_{\infty,\sigma} \le C\|g_R\|_{p,\sigma}\le C\|g\|_{p,\sigma}
$$
and thus we can pass  to the limit $R\to \infty$, obtaining a function  $\phi$, which solves problem \equ{l11} with   $\|\phi\|_{\infty,\sigma}< +\infty $.
Then it  follows from Lemma \ref{ap1} that
\begin{align}
\|D^2\phi\|_{p,\sigma}+\|D \phi\|_{\infty,\sigma} + \|\phi\|_{\infty,\sigma} \le C\|g\|_{p,\sigma}
\label{l12}
\end{align}
The general result for (\ref{l6}) follows now from a straightforward \blue{perturbation} argument.
\qed

We need to introduce now a norm that involves decay  with respect  $y\in \Gamma_\alpha$.  We recall the definition of $r_\alpha$:
$$
r_\alpha(y)= \sqrt{1+\alpha^2|\Pi_{\R^8}(y)|^2}, \quad y\in \Gamma_\alpha,
$$
where $\Pi_{\R^8}:\Gamma_\alpha\to \R^8$ is the projection onto  $\R^8$ of the (embedded) graph
$\Gamma_\alpha\subset \R^9$.
Let us consider the new norms for $g$,
\begin{align*}
\|g\|_{p,\sigma,\nu} :=\ \sup_{(y,z)\in \Gamma_\alpha\times\R}  e^{\sigma |z|} \,\| r_\alpha^\nu g\|_{L^p\big(\Gamma_\alpha\cap B(y,1)\blue{\times(z,z+1)}\big)},\quad 1<p<\infty, \nu\geq 1,
\end{align*}
and
$$
\|g\|_{\infty,\sigma,\nu} :=
\sup_{(y,z)\in \Gamma_\alpha\times\R} r^\nu_\alpha(y) e^{\sigma |z|} |g (y,z)|, \quad \nu \geq 1.
$$
Then we have the  following a priori estimate.
\begin{lemma} \label{ap2}
There exists a number $C>0$ such that for all sufficiently small $\alpha$ the following holds:
Given $g$  with $\|g\|_{p,\sigma,\nu} <+\infty$, $9<p\leq \infty$, $0\leq \nu$  there exists a unique solution
 $\phi$
 of $\equ{l6}$ with  $\|\phi\|_{\infty,\sigma,\nu }< +\infty$. This solution satisfies
$$
\|D^2\phi\|_{p,\sigma,\nu } + \|D\phi\|_{\infty,\sigma ,\nu } + \|\phi\|_{\infty,\sigma,\nu } \ \le\ C\, \|g\|_{p,\sigma,\nu}.
$$
\end{lemma}
\proof
This result is in fact a direct corollary (with obvious modifications for the respective norms) of the previous lemma.
We will set
$$
\phi = r_\alpha ^{-\nu}\psi,\quad
$$
and use (\ref{l6}) to  find the equation satisfied by $\psi$:
\begin{align}
\begin{aligned}
\Delta_{\Gamma_\alpha}\psi+r^\nu_\alpha\tilde B_\alpha(r_\alpha^{-\nu}\psi)+\partial_z^2\psi+f'(w)\psi&=r^\nu_\alpha g+r^\nu_\alpha c_\alpha w_z, \quad \mbox{in}\ \Gamma_\alpha\times\R,\\
\int_{\R} \psi(y,z) w_z(z)\,dz&=0, \quad \forall y\in \Gamma_\alpha,
\end{aligned}
\label{l13}
\end{align}
where $\tilde B_\alpha$ is a second order differential operator. Let us observe that
\begin{align*}
\|\Delta_{\Gamma_\alpha} r_\alpha\|_{L^\infty(\Gamma_\alpha)}+
\alpha\|\nabla_{\Gamma_\alpha} r_\alpha\|_{L^\infty(\Gamma_\alpha)}\leq C\alpha^2.
\end{align*}
This means that essentially the same argument as in the proof of Lemma \ref{ap1} applies to yield the $L^\infty$ estimate for $\psi$ in terms of  $\|g\|_{p,\sigma,\nu}$ and then local elliptic estimates give the estimate for the derivatives. We omit the details.

\qed

The theory developed in this section allows us to define an operator $T_\alpha(g):=\phi$ where $\phi$ is a solution of (\ref{l6}). In particular, with this definition,  we have
\begin{align}
\|D^2T_\alpha(g)\|_{p,\sigma,\nu } + \|DT_\alpha(g)\|_{\infty,\sigma ,\nu } + \|T_\alpha(g)\|_{\infty,\sigma,\nu } \ \le\ C\, \|g\|_{p,\sigma,\nu}.
\label{l14}
\end{align}
\subsection{Full linearized operator}\label{full linear}
We will now use the solvability theory for the problem (\ref{l6})  to treat the full linearized operator $\cL(\phi)=\Delta\phi+f'(\ww)\phi$, defined  in (\ref{def l2}). Thus we will consider the following problem:
\begin{align}
\cL(\phi)=g, \quad \mbox{in}\ \cN_{4\delta}.
\label{l15}
\end{align}

We recall that 
\begin{align*}
{\tt w}(x)=\begin{cases} \chi\big(\frac{4\alpha z}{\theta_0r_\alpha})\big(w(z-h_\alpha)+1\big)-1, &\quad z<0,\\
\chi\big(\frac{4\alpha z}{\theta_0r_\alpha})\big(w(z-h_\alpha)-1\big)+1, &\quad 0\leq z,
\end{cases},
\end{align*}
where $h_\alpha$ is a function defined  on $\Gamma_\alpha$. In addition we will assume that, with some $\mu>8/p$ and $\nu\geq 2$:
\begin{align}
\alpha^2\|h_\alpha\|_{\infty, \nu-2}+\alpha\|\nabla_{\Gamma_\alpha}h_\alpha\|_{\infty,\nu-1}+\alpha^{8/p}\|\nabla^2_{\Gamma_\alpha}h_\alpha\|_{p, \nu}:=\|h_{\alpha}\|_{*,p,\nu}\leq C\alpha^{2+\mu},
\label{l16}
\end{align}
where $p\in (9,\infty)$, and
\begin{align*}
\|h_\alpha\|_{p,\nu} &:=\ \sup_{y\in \Gamma_\alpha}  \,\|r_\alpha^\nu h_\alpha\|_{L^p(\Gamma_\alpha\cap B(y,\theta_0\alpha^{-1}))},\quad 9<p<\infty,\quad\nu\geq 2,\\
\|h_\alpha\|_{\infty,\nu} &:=
\sup_{y\in \Gamma_\alpha}r_\alpha(y)^\nu |h_\alpha(y)|, \quad \nu \geq 2.
\end{align*}
This means that we can assume that in $\cN_{4\delta}$ we have $\ww(x)=w(z-h_\alpha)$.
We recall that we can identify  functions defined on $\Gamma_\alpha$ and those defined  on $\Gamma$ through the relation:
\begin{align*}
h_\alpha(y)=h(\alpha y), \quad y\in \Gamma,
\end{align*}
where $h:\Gamma\to \R$. This justifies the definition of the norms and the assumption (\ref{l16}).

Using the operator defined in  (\ref{l5}) we can write
\begin{align*}
\cL(\phi)=\Delta_{\Gamma_\alpha}\phi+\partial^2_z\phi+\cB^\alpha_\delta(\phi)+f'(\tt w)\phi.
\end{align*}
Now we need to make a change of variables in $\cN_{4\delta}$:
\begin{align}
\bar z=z-h_\alpha.
\label{l17}
\end{align}
We will denote $\phi(y, z)=\bar\phi\big(y, z-h_\alpha(y)\big)$. Then we have
\begin{align*}
\Delta_{\Gamma_\alpha}\phi(y,z)=\Delta_{\Gamma_\alpha}\bar \phi(y,\bar z)+\bar B_1(\bar \phi),
\end{align*}
where $\bar B_{1}(\bar \phi)$ is a linear second order differential operator (in variables $(y,\bar z)$):
\begin{align*}
\bar B_1(\bar\phi)=-2\nabla_{\Gamma_\alpha}\bar\phi_{\bar z}\cdot \nabla_{\Gamma_\alpha} h_{\alpha}+\bar\phi_{\bar z\bar z}|\nabla_{\Gamma_\alpha} h_{\alpha}|^2-\bar\phi_{\bar z}\Delta_{\Gamma_\alpha} h_{\alpha}.
\end{align*}
We  will separate the term whose coefficients depend on the first derivatives of $h_\alpha$ from the rest. Thus we will denote:
\begin{align}
{\bf B}_{1\alpha}(\bar\phi)=-2\nabla_{\Gamma_\alpha}\bar\phi_{\bar z}\cdot \nabla_{\Gamma_\alpha} h_{\alpha}+\bar\phi_{\bar z\bar z}|\nabla_{\Gamma_\alpha} h_{\alpha}|^2.
\label{l17a}
\end{align}
Notice that ${\bf B}_{1\alpha}$ is an operator of the  same  form as $B_\alpha$ in (\ref{l6a}), and whose coefficients satisfy the analog of (\ref{l6b}) due to the assumption (\ref{l16}).

With the same change of variables we now analyze:
\begin{align*}
\cB^\alpha_\delta(\phi)=\eta_{\delta}^\alpha(\Delta_{\Gamma_{\alpha, z}}-\Delta_{\Gamma_\alpha}-H_{\Gamma_{\alpha, z}}\partial_z)(\phi).
\end{align*}
Let us fix a $y\in\Gamma_\alpha$ and let $\tg_{\alpha}$ be the metric tensor in  local coordinates around $y$. Let  $\tg_{\alpha,z}$ be the corresponding metric tensor on $\Gamma_{\alpha, z}$ (i.e. around the point $y_z=y+z n_\alpha(y)\in \Gamma_{\alpha,z}$, with $n_\alpha(y)$ denoting the normal vector).
Then we can  write, keeping in mind that $z=\bar z+h_\alpha$, $\tg_{\alpha,z}=\bar\tg_{\alpha,\bar z}$ and:
\begin{align}
\begin{aligned}
\Delta_{\Gamma_{\alpha,z}}\phi&=\frac{1}{\sqrt{\det(\bar\tg_{\alpha, \bar z})}}\partial_{i}
\Big(\bar \tg_{\alpha, \bar z}^{ij}{\sqrt{\det(\bar \tg_{\alpha, \bar z})}}\partial_{j}\bar\phi\Big)\\
&\quad -\frac{1}{\sqrt{\det(\bar \tg_{\alpha, \bar z})}}\partial_{i}
\Big(\bar \tg_{\alpha, \bar z}^{ij}{\sqrt{\det(\bar \tg_{\alpha, \bar z})}}\partial_{j}h_\alpha\Big)\partial_{\bar z}\bar\phi\\
&\quad+\frac{1}{\sqrt{\det(\bar \tg_{\alpha, \bar z})}}\partial_{\bar z}\Big(\bar \tg_{\alpha,\bar z}^{ij}{\sqrt{\det(\bar \tg_{\alpha,\bar z})}}\Big)\partial_{i}h_\alpha\partial_{j}h_\alpha\partial_{\bar z}
\bar\phi\\
&\quad +\bar\tg_{\alpha,\bar z}^{ij}\partial_{j}h_\alpha\partial_{i}h_\alpha\partial_{\bar z}^2\bar\phi
-\bar \tg_{\alpha,\bar z}^{ij}\partial_{j}h_\alpha\partial_{i {\bar z}}\bar\phi-\bar \tg_{\alpha,\bar z}^{ij}\partial_{i}h_\alpha\partial_{j \bar z}\bar \phi\\
&\quad-\frac{1}{\sqrt{\det(\bar\tg_{\alpha,\bar z})}}
\partial_{\bar z}\Big(\bar \tg_{\alpha,\bar z}^{ij}{\sqrt{\det(\bar\tg_{\alpha,\bar z})}}\Big)\partial_{i}h_\alpha\partial_{j}\bar\phi,
\end{aligned}
\label{l18}
\end{align}
We notice if in  $\tg_{\alpha,z}$ we set $z=0$ then the above operator is equal to $\Delta_{\Gamma_\alpha}\phi(y,z)$. Thus we need to "expand" (\ref{l18}) in powers of $z$. To this end let us use the  local system of coordinates around $y$   given by $(t, G_\alpha(t))$, where $G_\alpha(t)=\alpha^{-1}G(\alpha t)$, and $G$ given in Lemma \ref{lem gr1} and $|t|\leq \rho\alpha^{-1}$. Since we are interested in the size of the local norms defined above we only need to consider $|t|\leq C$, where $C>0$ is a constant independent on $\alpha$. By direct calculation, using Lemma \ref{lem gr1} and (\ref{coord 7}) we get that
\begin{align}
\begin{aligned}
\tg_{\alpha,z}&=I+\frac{1}{r^2_\alpha(y)}O(\alpha^2|t|^2)+\frac{z}{r_\alpha(y)}
O(\alpha)\\
&=\tg_{\alpha, 0}+\frac{z}{r_\alpha(y)}
O(\alpha)\\
&=\tg_{\alpha, 0}+\frac{\bar z+h_\alpha}{r_\alpha(y)}
O(\alpha)\\
&=\bar\tg_{\alpha,\bar z},
\end{aligned}
\label{l19}
\end{align}
with similar estimates for the derivatives. It follows:
\begin{align}
\bar \tg^{ij}_{\alpha,\bar z}=\tg^{ij}_{\alpha,0}+\frac{\bar z+h_\alpha}{r_\alpha(y)}
O(\alpha).
\label{l20}
\end{align}
Setting, with some abuse of notation,
\begin{align}
\Delta_{\Gamma_{\alpha, \bar z}}h_\alpha=\frac{1}{\sqrt{\det(\bar \tg_{\alpha, \bar z})}}\partial_{i}
\Big(\bar \tg_{\alpha, \bar z}^{ij}{\sqrt{\det(\bar \tg_{\alpha, \bar z})}}\partial_{j}h_\alpha\Big)
\label{l20a}
\end{align}
we see that
\begin{align}
(\Delta_{\Gamma_{\alpha, z}}-\Delta_{\Gamma_\alpha})\phi={\bf B}_{2\alpha}(\bar \phi)-[(\Delta_{\Gamma_{\alpha, \bar z}}-\Delta_{\Gamma_\alpha})h_\alpha]\partial_{\bar z}\bar\phi,
\label{l21}
\end{align}
where ${\bf B}_{2\alpha}(\bar \phi)$ is a second order linear differential operator in $\bar \phi$ whose coefficients depend on $h_\alpha, \nabla_{\Gamma_\alpha}h_\alpha, \bar z$ and that can be estimated as follows:
\begin{align}
|{\bf B}_{2\alpha}(\bar \phi)|\leq \Big[\frac{O(\alpha)(|\bar z|+|h_\alpha|)}{1+r_\alpha(y)}+|\nabla_{\Gamma_\alpha}h_\alpha|\Big](|\nabla\bar\phi|+|D^2\bar\phi|).
\label{l22}
\end{align}
We will consider now the term:
\begin{align}
\begin{aligned}
H_{\Gamma_{\alpha, z}}\partial_z\phi&=z|A_{\blue{\Gamma_{\alpha}}}|^2\partial_{\bar z}\bar\phi+
z^2{\mathcal R}_{\alpha}\partial_{\bar z}\bar\phi\\
&=(\bar z+h_\alpha)|A_{\Gamma_{\alpha}}|^2\partial_{\bar z}\bar\phi+(\bar z+h_\alpha)^2{\mathcal R}_{\alpha}\partial_{\bar z}\bar\phi,
\end{aligned}
\label{l23}
\end{align}
where $|A_{\Gamma_{\alpha}}|^2$ is the norm of the second fundamental form on $\Gamma_\alpha$, which satisfies:
\begin{align}
|A_{\Gamma_{\alpha}}|^2= \alpha^2|A_{\Gamma}|^2\leq \frac{C\alpha^2}{r_\alpha^2(y)}.
\label{l24}
\end{align}
(Above $|A_{\Gamma}|$ is the norm of the second fundamental form on $\Gamma$).
Term  ${\mathcal R}_{\alpha}$ comes from the Taylor expansion of $H_{\Gamma_{\alpha, z}}$ (see (\ref{blue 1})--(\ref{hgammaz})) and it has the form\blue{:
\begin{align}
\begin{aligned}
{\mathcal R}_\alpha&=\frac{1}{z^2}[H_{\Gamma_{\alpha,z}}-z|A_{\Gamma_\alpha}|^2]\\
&=\frac{1}{z^2}\sum_{i=1}^8\Big(\frac{\kappa_i}{1-z\kappa_i}-z\kappa_i^2\Big)\\
&=\sum_{i=1}^8 \big[\kappa_i^3+O(|z|\kappa_i^4)],
\end{aligned}
\label{restimate}
\end{align}
and can be  bounded as follows:
\begin{align}
|{\mathcal R}_{\alpha}|\leq \frac{C\alpha^3}{r^3_\alpha(y)}.
\label{l25}
\end{align}}
From this we see that  (\ref{l15}) can be equivalently written in the form:
\begin{align}
\begin{aligned}
\Delta_{\Gamma_{\alpha}}\bar\phi+\partial_{\bar z}^2\bar\phi+{\bf B}_{1\alpha}(\bar\phi)+f'(w)\bar\phi&=\bar g+\partial_{\bar z}\bar \phi\Delta_{\Gamma_{\alpha}} h_\alpha
-\eta_\delta^\alpha{\bf B}_{2\alpha}(\bar \phi)
\\
&\quad+ \eta_\delta^{\alpha}\big[(\Delta_{\Gamma_{\alpha, \bar z}}-\Delta_{\Gamma_{\alpha}}) h_\alpha\big]\\
&\quad +\eta_\delta^{\alpha}(\bar z+h_\alpha)|A_{\Gamma_{\alpha}}|^2\partial_{\bar z}\bar \phi
\\
&\quad + \eta_\delta^{\alpha}(\bar z+h_\alpha)^2{\mathcal R}_{\alpha}\partial_{\bar z}\bar \phi,
\end{aligned}
\label{l26}
\end{align}
in $\cN_{4\delta}$. It follows that it is natural to extend $\cL$ outside of $\cN_{4\delta}$ by letting
\begin{align*}
{\cL}(\phi)=\Delta_{\Gamma_{\alpha}}\bar\phi+\eta_{4\delta}^\alpha{\bf B}_{1\alpha}(\bar\phi)+\partial_{\bar z}^2\bar\phi+f'(w)\bar\phi, \quad \mbox{in}\ \Gamma_\alpha\times \R \setminus\cN_{4\delta},
\end{align*}
since $\eta_{4\delta}^\alpha\equiv 1$ in $\cN_{4\delta}$.
We will denote the extended operator ${\cL}$ by $\bar{\cL}$. With this extension we can consider (\ref{l26}) as an equivalent problem to $\bar{\cL}(\bar\phi)=\bar g$ in $\Gamma_\alpha\times \R$. Now we can use the results of the previous   subsection to show the following analog of Lemma \ref{ap2}:
\begin{proposition} \label{prop ap2}
There exists a number $C>0$ such that for all sufficiently small $\alpha$ and $\delta$ the following holds:
Given $\bar  g$  with $\|\bar g\|_{p,\sigma,\nu} <+\infty$, $9<p\leq \infty$, $2\leq \nu$  there exists a unique solution
$\bar \phi$
of
\begin{align}
\begin{aligned}
\bar{\cL}(\bar \phi)&=\bar g+\bar c_\alpha w_{\bar z}, \quad \mbox{in}\ \cN_\infty,\\
\int_\R\bar\phi(y,\bar z)w_{\bar z}(\bar z)\,d\bar z&=0, \quad \mbox{in}\ \Gamma_\alpha,
\end{aligned}
\label{l27}
\end{align}
with  $\|\bar\phi\|_{\infty,\sigma,\nu }< +\infty$. This solution satisfies
$$
\|D^2\bar\phi\|_{p,\sigma,\nu } + \|D\bar\phi\|_{\infty,\sigma ,\nu } + \|\bar\phi\|_{\infty,\sigma,\nu } \ \le\ C\, \|\bar g\|_{p,\sigma,\nu}.
$$
\end{proposition}
\proof
Using estimates (\ref{l20a})--(\ref{l26}) we see that the right hand side of (\ref{l26}) can be estimated by:
\begin{align*}
\|\bar g\|_{p,\sigma,\nu}+C(\delta+\alpha)(\|D^2\bar\phi\|_{p,\sigma,\nu } + \|D\bar\phi\|_{\infty,\sigma ,\nu } + \|\bar\phi\|_{\infty,\sigma,\nu }).
\end{align*}
Using this and a straightforward fixed point argument we obtain our result. The details are left to the reader.

\qed

\setcounter{equation}{0}
\section{The nonlinear problem}

\subsection{The gluing}
Let us recall (see section \ref{sec preliminaries}) that we are looking for the solution of (\ref{ac}) in the form:
\begin{align}
u_\alpha=\ww(x)+\phi(x).
\label{glue 1}
\end{align}
Substituting in (\ref{ac}) we get for  the function $\phi$
\begin{align}
\cL(\phi)=S[\ww]+N(\phi),
\label{glue 2}
\end{align}
where
\begin{align}
S[\ww]=-\Delta\ww-f(\ww), \quad N(\phi)=-[f(\ww+\phi)-f(\ww)-f'(\ww)\phi], \quad f(\ww)=\ww(1-\ww^2).
\label{glue 3}
\end{align}

Let us us also recall (see section \ref{sec imp}) that we introduced an improvement of the initial approximation, namely the function $\ww_1$ defined in (\ref{imp 5})--(\ref{imp 6}).
We look for a solution $\ttt\phi$ of equation \equ{glue 2} in the  form
\begin{align}
\ttt\phi =\tt w_1 +\eta^\alpha_{2\delta} \phi + \psi,
\label{ualpha2}
\end{align}
so that the following  must be satisfied:
\begin{align}
\begin{aligned}
\eta^{{\alpha}}_{2\delta}(\Delta \phi + f'(\ww )\phi)  + 2\nabla\eta^\alpha_{2\delta} \nabla\phi    +  \phi \Delta \eta^\alpha_{2\delta}  + \Delta\psi + f'(\ww) \psi &=S[\ww]+\tL(\tt w_1)\\
&\quad+
 \NNN (\ww_1+\eta^\alpha_{2\delta} \phi + \psi ),
\end{aligned}
\label{non 2}
\end{align}
where
\begin{align*}
\tL({\tt w}_1)=-\Delta\ww_1 - f'(\ww )\ww_1 .
\end{align*}
For brevity in what follows we will write:
\begin{align*}
\NN(\phi)&=\NNN(\ww_1+\phi)=-[f(\ww+\ww_1+\phi)-f(\ww)-f'(\ww)(\ww_1+\phi)],\\
\tS(\ww,\ww_1)&=S[\ww]+\tL(\ww_1).
\end{align*}
This equation is satisfied provided that $(\phi,\psi)$ satisfies a coupled system of  nonlinear elliptic equations:
\begin{align}
\Delta \phi + f'(\ww )\phi   = \eta_\delta^\alpha \big[\tS(\ww,{\tt w}_1)+\NN (\phi +\psi )-\big(2 + f'(\ww)\big)\psi\big] ,\inn \cN_{4\delta},
\label{sis1}
\\
\begin{aligned}
\Delta \psi \, -\, 2\psi  \, &= \,
(1- \eta^\alpha_\delta)\big[\tS(\ww,{\ww}_1) + \NN(\eta^\alpha_{2\delta}\phi+\psi)-\big (2+ f'(\ww)\, \big)\psi\big] \\ &\quad -2\nabla\eta^\alpha_{2\delta} \nabla\phi   \, - \, \phi \Delta \eta^\alpha_{2\delta}, \inn \R^9\ .
\end{aligned}
\label{sis2}
\end{align}
Let us consider the extension $\tilde \cL$ of the linear operator $\cL$ introduced in the previous section. Then   equation \equ{sis1} is equivalent to:
\begin{align}
\tilde \cL(\bar \phi)= \eta_\delta^\alpha \big[\bar \tS(w,w_1)+\NN (\bar\phi +\bar \psi )-\big(2 + f'(w)\big)\bar\psi\big], \inn \blue{\Gamma_\alpha\times \R},
\label{non 3}
\end{align}
where
\begin{align*}
\bar\phi(y,\bar z)=\phi(y,\bar z+h_\alpha), \quad \bar\psi=\psi(y,\bar z+h_\alpha),
\end{align*}
according with the change of variables $\bar z=z-h_\alpha$ introduced in the previous section. The error term $\tS(\ww,\ww_1)$ expressed in these variables is denoted above by $\bar \tS(w,w_1)$ \blue{(observe that in the support of $\eta_\delta^\alpha$ we have  $\ww=w$, $\ww_1=w_1$).}

Now we will recast the system (\ref{sis1})--(\ref{sis2}) as a fixed point problem for $\bar\phi$.
Let us consider a given function
$\bar \phi$ such that $\|\nabla\bar \phi\|_{\infty,\sigma, \nu}+ \|\bar \phi\|_{\infty,\sigma, \nu}$ is small. We set $\phi=\bar\phi(y,z-h_\alpha)$ in $\cN_{4\delta}$ and $\phi\equiv 0$ in $\R^9\setminus \cN_{4\delta}$ and write \equ{sis2} as
\be
\Delta \psi-2\psi =  M(\psi,\phi)+ P(\phi), \inn \R^9,
\label{sis22}
\ee
where
\begin{align*}
P(\phi)&= (1- \eta^\alpha_\delta)\big( \tS(\ww,\ww_1)+ \NN (\eta^\alpha_{2\delta}\phi)\big)+ 2\nabla\eta^\alpha_{2\delta} \nabla\phi +\phi \Delta \eta^\alpha_{2\delta}, \\
M(\psi,\phi)&=
(1- \eta^\alpha_\delta)[(2+ f'(\www)\, )\psi+\NN (\eta^\alpha_{2\delta}\phi +\psi ) -\NN (\eta^\alpha_{2\delta}\phi)\,].
\end{align*}
{For functions in $\R^9$ we define  the following norms:
$$
 \|Q\|_{p,\nu} :=\sup_{x\in \R^9}\|r_\alpha^\nu Q\|_{L^p(B(x,1))}  < +\infty, \quad
 r_\alpha(x)=\sqrt{1+\alpha^2|\Pi_{\R^8}(x)|^2}.
$$
We notice   that the function $P$ above satisfies the following estimate:
\begin{align}
\begin{aligned}
 \|P(\phi)\|_{p,\nu}\,& \le\,  C e^{\,-{\sigma\delta}/{\alpha} }\, \big[\|\tS(\ww,\ww_1)\|_{p,\sigma, \nu}+e^{\,-{\sigma\delta}/{\alpha} }( \|\nabla\phi\|_{\infty,\sigma, \nu}+\|\phi\|_{\infty,\sigma, \nu})\big] \\
 &\leq C e^{\,-{\sigma\delta}/{\alpha} }\, \big[\|\tS(\ww,\ww_1)\|_{p,\sigma, \nu}+ e^{\,-{\sigma\delta}/{\alpha} }(\|\nabla\bar\phi\|_{\infty,\sigma, \nu}+\|\bar\phi\|_{\infty,\sigma, \nu})\big].
\end{aligned}
\label{p}
\end{align}
Motivated by (\ref{p}) we will establish our next result:
\begin{lemma}\label{lemma psi1}
Let us consider the linear problem
\be
\Delta \psi \, -\, 2\psi  = Q(x) \, \inn \R^9.
\label{sis222}
\ee
There exists a number $C>0$ such that if $Q(x)$ satisfies
$$
\|Q\|_{p,\nu}  < +\infty,
$$
with $9<p< \infty$, then equation $\equ{sis222}$ has a unique bounded solution $\psi$, which defines a linear operator in $Q$ and satisfies
$$
  \|\nabla \psi \|_{\infty,\nu} + \|\psi \|_{\infty,\nu} \le
C \|Q\|_{p,\nu} .
 $$
\end{lemma}
\begin{proof}
Let us write
$$ \mu(x) = r_\alpha^{-\nu}(x),\quad\psi(x) = \mu(x) \ttt \psi(x). $$
Then the equation in terms of $\ttt\psi$ reads
$$
 \Delta \ttt \psi - 2\ttt\psi +\mu^{-1} (\ttt \psi\Delta \mu + \nabla\mu\cdot \nabla\ttt \psi)  = \mu^{-1}Q.
$$
We notice that:
$$
|\mu^{-1} \Delta\mu| \leq C\alpha^2 r_\alpha^{-2}, \quad
|\mu^{-1}\nabla\mu| \leq C\alpha r_\alpha^{-1}.
$$
Thus the equation becomes,
$$
\Delta\ttt \psi - 2\ttt\psi +  a^\alpha_1\cdot\nn \ttt \psi + a^\alpha_2\ttt\psi   =  \ttt Q(y), \quad
\ttt Q= \mu^{-1}Q,
$$
where $a^\alpha_i=o(1)$, as $\alpha\to 0$, uniformly.
Let us approximate
$\ttt Q$ by  bounded functions,
$$
\ttt Q_n =  \min\{ |\ttt Q| , n\}\, \hbox{sign}\, {\ttt Q},
$$
and consider the unique bounded solution $\ttt\psi_n$ of
$$
\Delta \ttt \psi_n - 2\ttt\psi_n +  a^\alpha_1\cdot \nn \psi_n + a^\alpha_2\ttt\psi_n   =  \ttt Q_n (y).
$$
We claim that there exists a $C>0$, independent of $\ttt Q$ and $\alpha$  such that
\be
\|\ttt \psi_n\|_\infty + \|\nabla\ttt \psi_n \|_\infty \, \le\, C\|\ttt Q_n\|_{p,0} \le  C\|\ttt Q\|_{p,0}
\label{bdd}\ee
Assuming the opposite, we have sequences $\ttt\psi_n$, $\ttt Q_n$ such that
$$
\|\ttt\psi_n\|_\infty + \|\nabla\ttt \psi_n \|_\infty =1, \quad \|\ttt Q_n\|_{p,0}\to 0, \quad n\to \infty.
$$
Let $x_n$ be such that
$$
(|\ttt \psi_n| + |\nabla\ttt \psi_n |)(x_n) \ge \frac 12
$$
Then, by local elliptic estimates, we get that the function $x\mapsto \psi_n(x_n + x)$ converges locally over compacts
 to a nontrivial, bounded  solution $\ttt\psi$ of the equation
$$ \Delta_y \ttt \psi - 2\ttt\psi  = 0.$$
We have reached a contradiction that proves the  estimate \equ{bdd}.
Passing to the limit, the lemma readily follows.
\end{proof}

Using the above lemma, we can apply contraction mapping principle to conclude that there is a unique  solution $\psi =\psi (\phi)$
of equation \equ{sis22}
which in addition satisfies
\be
\|\nabla \psi \|_{\infty,\nu} + \|\psi \|_{\infty,\nu}  \le \, C e^{\,-{\sigma\delta}/ {\alpha} }\, \big[ \|\tS(\ww,\ww_1)\|_{p,\sigma, \nu} +e^{\,-{\sigma\delta}/{\alpha} }( \|\phi\|_{\infty,\sigma, \nu}+ \|\nabla \phi\|_{\infty,\sigma, \nu})\big].
\label{oper}
\ee
In addition, we will  check that $\psi$ is  a Lipschitz function
in the considered norms, both in $\phi$ and in $h_\alpha$.
Substituting in equation \equ{non 3} we get that our full problem has been reduced to solving the nonlinear, nonlocal problem
\begin{align}
\tilde \cL(\bar \phi)= \eta_\delta^\alpha \big[\bar \tS(w,w_1)+\NN (\bar\phi +\bar \psi(\bar \phi) )-\big(2 + f'(w)\big)\bar\psi(\bar\phi)\big], \inn \Gamma_\alpha\times \R.
\label{non 4}
\end{align}
}

\subsection{The projected nonlinear problem}
We consider the projected version of problem \equ{non 4}
\begin{align}
\begin{aligned}
\tilde \cL(\bar \phi)&= \bar c_\alpha w_{\bar z}+\eta_\delta^\alpha\bar S(w,w_1)+\bar\NNN (\bar\phi), \inn \Gamma_\alpha\times \R,
\\
\int_\R  \bar\phi(y,\bar z) w_{\bar z}(\bar z)\, d\bar z\, &=0, \inn \Gamma_\alpha,
\end{aligned}
\label{proj}
\end{align}
where for convenience we have denoted:
\begin{align*}
\bar \NNN(\bar \phi)=\eta_\delta^\alpha\big[\NN (\bar\phi +\bar \psi(\bar \phi) )-\big(2 + f'(w)\big)\bar\psi(\bar\phi)\big].
\end{align*}
In the sequel we will use notation:
\begin{align}
\|\bar\phi\|_{*,p,\sigma, \nu}=\|D^2\bar \phi\|_{p,\sigma, \nu}+\|\nabla\bar \phi\|_{\infty,\sigma, \nu}+\|\bar \phi\|_{\infty,\sigma, \nu}, \quad 9<p<\infty.
\label{non 5}
\end{align}
\begin{lemma}\label{lemma non 1}
Let $\bar \phi$ be a given function such that
\begin{align}
\|\bar\phi\|_{*,p,\sigma, \nu}<\infty.
\label{non 5a}
\end{align}
Then mapping  $\bar\NNN(\bar\phi)$ satisfies:
\begin{align}
\|\bar \NNN(\bar\phi)\|_{p,\sigma, \nu}\leq C(\|\bar\phi\|^2_{*,p,\sigma, \nu} +\blue{e^{\,-\sigma\delta/\alpha}\|\bar\phi\|_{*,p,\sigma, \nu}}+
\|\bar \tS(\ww,\ww_1)\|_{p,\sigma, \nu}).
\label{non 6}
\end{align}
In addition for any  functions $\bar\phi_k$, $k=1,2$, satisfying (\ref{non 5a}) we have:
\begin{align}
\|\bar \NNN(\bar\phi_1)-\bar \NNN(\bar\phi_2)\|_{p,\sigma, \nu}\leq C[\|\bar\phi_1\|_{*,p,\sigma, \nu}
+\|\bar\phi_2\|_{*,p,\sigma, \nu}+e^{\,-\sigma\delta/\alpha}]\|\bar\phi_1-\bar\phi_2\|_{*,p,\sigma, \nu}.
\label{non 7}
\end{align}
\end{lemma}
\proof
Let us consider the solution of  \equ{sis22} denoted, after the change of variables, by
$\bar \psi(\bar \phi)$. Using (\ref{oper}) we get
\begin{align}
\begin{aligned}
\|\bar\NNN(\bar\phi)\|_{p, \sigma,\nu}&\leq C[\|\bar\phi\|_{\infty,\sigma, \nu}^2+e^{\,2\sigma\delta/\alpha}\|\bar\psi\|_{\infty,\nu}^2+\blue{e^{\,\sigma\delta/\alpha}}\|\bar\psi\|_{\infty,\nu}]\\
&\leq C[\|\bar\phi\|_{*,p,\sigma, \nu}^2+\blue{e^{\,-\sigma\delta/\alpha}\|\bar\phi\|_{*,p,\sigma, \nu}}+\|\bar \tS(\ww,\ww_1)]\|_{p,\sigma, \nu}].
\end{aligned}
\label{non 8}
\end{align}
(Above we have used the fact that $\|\bar \tS(\ww,\ww_1)\|_{p,\sigma, \nu}$, is small, see Proposition \ref{prop2} to follow). This shows (\ref{non 6}).

To prove (\ref{non 7}) we denote $\bar\psi_k=\bar\psi(\bar\phi_k)$, $k=1,2$ and use (\ref{oper}) again to get
\be
\|\nabla \bar\psi_1-\nabla\bar\psi_2 \|_{\infty,\nu} + \|\bar\psi_1-\bar\psi_2\|_{\infty,\nu}  \le \, C e^{\,-{2\sigma\delta}/ {\alpha} }\, \big[\|\bar\phi_1-\bar\phi_2\|_{\infty,\sigma, \nu}+ \|\nabla\bar \phi_1-\nabla\bar\phi_2\|_{\infty,\sigma, \nu})\big].
\label{oper1}
\ee
Estimate (\ref{non 7}) follows readily from this.

\qed

We will now show the main result of this section.
\begin{proposition}\label{prop2}
Under the assumption $2\leq \nu\leq 3$ we have that
\begin{align}
\|\bar \tS(\ww,\ww_1)]\|_{{p,\sigma,\nu}} \le C\alpha^2, \quad 9<p<\infty.
\label{non 9}
\end{align}
As a consequence, for all $\blue{\alpha}$\comment{typo} sufficiently small, problem $\equ{proj}$ has a unique solution
$\bar \phi$ with
\begin{align}
 \|\bar\phi\|_{*,p,\sigma,\nu}  \le\ C\alpha^2.
 \label{non 10}
 \end{align}
In addition $\bar\phi$ depends in a Lipschitz way on $h_\alpha$ in natural norms, namely we have:
\begin{align}
\|\bar\phi^{(1)}-\bar\phi^{(2)}\|_{*,p,\sigma,\nu}  \le\ C\alpha^{-8/p}\|h_\alpha^{(1)}-h_\alpha^{(2)}\|_{*,p,\nu},
\label{non 10a}
\end{align}
where $\bar\phi^{(k)}$, $k=1,2$ are solutions of \equ{proj} with $h_\alpha=h_\alpha^{(k)}$.
\end{proposition}
\proof
We begin by proving (\ref{non 9}). Let us write:
\begin{align*}
\bar \tS(\ww,\ww_1)&=\eta^\alpha_{\delta} \bar \tS(\ww,\ww_1)+(1-\eta^\alpha_{\delta}) \bar \tS(\ww,\ww_1)
\\
&:=-E_1-E_2.
\end{align*}

Notice that in terms of the original Fermi coordinates of $\Gamma_\alpha$ we have in
$\cN_{\delta}$:
\begin{align*}
E_1&=\Delta w+f(w)+\Delta w_1+f'(w)w_1\\
&=\Delta_{\Gamma_{\alpha, z}}(w+w_1)-H_{\Gamma_{\alpha, z}}\partial_z(w+w_1)+\partial_z^2(w+w_1)+f(w)+f'(w)w_1,
\end{align*}
where $w=w(z-h_\alpha)$. We will decompose:
\begin{align*}
E_1&=\Delta_{\Gamma_{\alpha,z}}(w+w_1)+[-H_{\Gamma_{\alpha, z}}\partial_zw+\bar z|A_{\Gamma_\alpha}|^2w_{\bar z}]-H_{\Gamma_{\alpha, z}}\partial_zw_1\\
&:=E_{11}+E_{12}+E_{13}.
\end{align*}
To estimate $E_{11}$ we use the expression of the Laplace-Beltrami operator in local coordinates as in (\ref{l18})--(\ref{l19}). Thus we get, changing to $\bar z=z-h_\alpha$, denoting $\tilde w=w+w_1$  and remembering that $w+w_1=w(\bar z)+w_1(\bar z)$:
\begin{align}
\begin{aligned}
\Delta_{\Gamma_{\alpha,z}}\tilde w&=-\frac{1}{\sqrt{\det(\bar \tg_{\alpha, \bar z})}}\partial_{i}
\Big(\bar \tg_{\alpha, \bar z}^{ij}{\sqrt{\det(\bar \tg_{\alpha, \bar z})}}\partial_{j}h_\alpha\Big)\partial_{\bar z}\tilde w\\
&\quad+\frac{1}{\sqrt{\det(\bar \tg_{\alpha, \bar z})}}\partial_{\bar z}\Big(\bar \tg_{\alpha,\bar z}^{ij}{\sqrt{\det(\bar \tg_{\alpha,\bar z})}}\Big)\partial_{i}h_\alpha\partial_{j}h_\alpha\partial_{\bar z}\tilde w\\
&\quad +\bar\tg_{\alpha,\bar z}^{ij}\partial_{j}h_\alpha\partial_{i}h_\alpha\partial_{\bar z}^2\tilde w,
\end{aligned}
\label{non 11}
\end{align}
We observe that, fixing a $y\in \Gamma_\alpha$ and considering the norms in the local variables $(t, G_\alpha(t))$ around $y$, we get :
\begin{align}
e^{\sigma |z|}\|r^\nu_\alpha\partial_{ij}h_\alpha \tilde w_z\|_{L^p(B(0,1))}\leq C\alpha^{2+\mu-8/p}\leq C\alpha^2,
\label{non 12}
\end{align}
because  the assumption (\ref{l16}) we have made about $h_\alpha$ and the fact that $\mu>p/8$.
Likewise we get
\begin{align}
e^{\sigma |z|}\|r^\nu_\alpha\partial_{i}h_\alpha\partial_jh_\alpha \tilde w_z\|_{L^p(B(0,1))}\leq C \alpha^{2+2\mu}.
\label{non 13}
\end{align}
This and (\ref{l19}) gives:
\begin{align}
\|E_{11}\|_{p,\sigma, \nu}\leq C\alpha^2.
\label{non 14}
\end{align}
Now we turn our attention to $E_{12}$. Using (\ref{l23})--(\ref{l25}) we get locally   in $\cN_{\delta}$:
\begin{align}
\begin{aligned}
H_{\Gamma_{\alpha, z}}\partial_z w-\bar z|A_{\blue{\Gamma_\alpha}}|^2w_{\bar z}
&=h_\alpha|A_{\Gamma_\alpha}|^2\partial_{\bar z}w+(\bar z+h_\alpha)^2{\blue{\mathcal R}}_{\alpha}\partial_{\bar z}w\\
&\leq \frac{C\alpha^2|h_\alpha|+\alpha r_\alpha^{-1}}{r^2_\alpha(y)}e^{\,-\sigma|\bar z|},
\end{aligned}
\label{non 15}
\end{align}
hence using the assumption $2\leq \nu\leq 3$ we get:
\begin{align}
\|E_{12}\|_{p,\sigma,\nu}\leq C\alpha^2.
\label{non 16}
\end{align}
Since
\begin{align*}
|w_1|\leq C\frac{\alpha^2}{r_\alpha^2}e^{\,-\sigma|\bar z|},
\end{align*}
therefore we get immediately
\begin{align}
\|E_{13}\|_{p,\sigma,\nu}\leq C\alpha^2.
\label{non 16a}
\end{align}
In order to estimate $E_2$ we will assume that $\bar z+h_\alpha>0$ and write:
\begin{align*}
\ww=\chi\big(\frac{4\alpha (\bar z+h_\alpha)}{\theta_0r_\alpha}\big) w(\bar z)+1-\chi\big(\frac{4\alpha (\bar z+h_\alpha)}{\theta_0r_\alpha}\big).
\end{align*}
Using the Fermi coordinates we can write:
\begin{align}
\begin{aligned}
(1-\eta_\delta^\alpha)\big(\Delta \ww+f(\ww)\big)&=(1-\eta^\alpha_\delta)[\partial_z^2\ww+f(\ww)]+(1-\eta^\alpha_\delta)\Delta_{\Gamma_{\alpha, z}}\ww\\
&\qquad-(1-\eta^\alpha_\delta)H_{\Gamma_{\alpha, z}}\partial_z\ww\\
&=(1-\eta^\alpha_\delta)(E_{21}+E_{22}+E_{23}).
\end{aligned}
\label{non 17}
\end{align}
Let us consider the  term denoted by $E_{21}$.  We have
\begin{align*}
E_{21}&=f\big(\chi\big(\frac{4\alpha (\bar z+h_\alpha)}{\theta_0r_\alpha}\big) w(\bar z)+1-\chi\big(\frac{4\alpha (\bar z+h_\alpha)}{\theta_0r_\alpha}\big)\big)-\chi\big(\frac{4\alpha (\bar z+h_\alpha)}
{\theta_0r_\alpha}\big)f(w(\bar z))
\\&\quad+2\partial_{\bar z}\chi\big(\frac{4\alpha (\bar z+h_\alpha)}{\theta_0r_\alpha}\big)\partial_{\bar z}w(\bar z)\\
&\quad +\partial^2_{\bar z}\chi\big(\frac{4\alpha (\bar z+h_\alpha)}{\theta_0r_\alpha}\big)\big[w(\bar z)-1\big]\\
&=A_1+A_2+A_3.
\end{align*}
To estimate $A_1$ we write
\begin{align*}
&\chi\big(\frac{4\alpha (\bar z+h_\alpha)}{\theta_0r_\alpha}\big) w(\bar z)+1-\chi\big(\frac{4\alpha (\bar z+h_\alpha)}{\theta_0r_\alpha}\big)\\
&\qquad=w(\bar z)+\big[1-\chi\big(\frac{4\alpha (\bar z+h_\alpha)}
{\theta_0r_\alpha}\big)\big](1-w(\bar z)\big]\\
&\qquad=w(\bar z)+ \big[1-\chi\big(\frac{4\alpha (\bar z+h_\alpha)}{\theta_0r_\alpha}\big)\big](1-w(\bar z)\big]O(e^{\,-\sqrt{2}|\bar z|})\\
&\qquad=w(\bar z)+O(e^{\,-\sigma|\bar z|})e^{\,-(\sqrt{2}-\sigma)\theta_0r_\alpha/\alpha}.
\end{align*}
Notice that, with some $\tilde\sigma>0$ we have
\begin{align*}
 e^{\,-(\sqrt{2}-\sigma)\theta_0r_\alpha/\alpha}\leq C e^{\,-\tilde \sigma\theta_0 r_\alpha/\alpha}e^{\,-\tilde \sigma\theta_0/\alpha},
\end{align*}
hence, using the fact that $f(w)$ is exponentially small as well in the support  of the function $1-\chi\big(\frac{4\alpha (\bar z+h_\alpha)}{\theta_0r_\alpha}\big)$, we get:
\begin{align*}
|A_1|\leq C e^{\,-\sigma|\bar z|}e^{\,-\tilde \sigma\theta_0 r_\alpha/\alpha}e^{\,-\tilde\sigma\theta_0 /\alpha},
\end{align*}
form which it follows
\begin{align}
\|A_1\|_{p,\sigma, \nu}\leq e^{\,-\tilde \sigma\theta_0/4\alpha}\leq C\alpha^2.
\label{non 18}
\end{align}
Terms $A_2$, and $A_3$ above are estimated in a similar way. To estimate the remaining term in $E_2$, namely $E_{22}$ and $E_{23}$, we use the same general approach.  The key point here is the fact that, with $k\geq 1$:
\begin{align}
\begin{aligned}
\|\nabla_{\Gamma_{\alpha}} r^{-k}_\alpha\|_{\infty,2}&\leq C\alpha, \\
\|\nabla^2_{\Gamma_{\alpha}} r^{-k}_\alpha\|_{p,2}&\leq C \alpha^{2-8/p},
\end{aligned}
\label{non 18a}
\end{align}
and the exponential smallness of $w(\bar z)\pm 1$ in the support of  $\chi'\big(\frac{4\alpha (\bar z+h_\alpha)}{\theta_0r_\alpha}\big)$.
Finally, the remaining terms in $E_2$ are handled similarly since $w_1(\bar z)\sim e^{\,-\sigma|\bar z|}$. The details are omitted.

Now, using (\ref{non 9}),  Lemma \ref{lemma non 1} and Proposition \ref{prop ap2} we show the existence of a unique solution $\bar \phi$  of (\ref{proj}) by a fixed point argument. The estimate (\ref{non 10}) is deduced from this as well.

Next we will prove that $\bar \phi$ is  Lipschitz  as a function of $h_\alpha$. To apply the general theory developed and in particular Proposition  \ref{prop ap2} let us fix functions $h^{(k)}_\alpha$, $k=1,2$ satisfying (\ref{l16}) and denote by $\bar \phi^{(k)}$ solutions of the respective nonlinear projected problems (\ref{proj}).  We notice that the functions $\bar\phi^{(k)}$ are defined in the same domain $\Gamma_\alpha\times \R$ however the linear parts of the  equations they solve are different, since  the coefficients of the differential operators involved expressed in local coordinates depend on
$h^{(k)}_\alpha$ as well. Thus we will denote the respective linear operators by $\tilde \cL^{(k)}$. We will also write $\tilde\phi=\bar\phi^{(1)}-\bar\phi^{(2)}$. With these notations we have that $\tilde\phi$ is a solution of:
\begin{align}
\begin{aligned}
&\tilde\cL^{(1)}\tilde\phi=\big\{\eta_\delta^\alpha(\bar z+h^{(1)}_\alpha)\bar \tS(w^{(1)}+w_1^{(1)})-\eta_\delta^\alpha(\bar z+h^{(2)}_\alpha)\bar \tS(w^{(2)}+w_1^{(2)})\big\}\\
&\qquad\qquad+(\tilde\cL^{(2)}-\tilde \cL^{(1)})\bar\phi^{(2)}+ (\bar c^{(1)}_\alpha -\bar c^{(2)}_\alpha)w_{\bar z}\\
&\qquad\qquad +\bar\NNN (\bar\phi^{(1)})-\bar\NNN(\bar\phi^{(2)}), \inn \blue{\Gamma_\alpha\times \R},
\\
&\int_\R  \tilde\phi(y,\bar z) w_{\bar z}(\bar z)\, d\bar z\, =0, \inn \Gamma_\alpha.
\end{aligned}
\label{non 19}
\end{align}
We will begin with estimating the following term:
\begin{align*}
\tilde E=\bar \tS(\ww^{(1)}+\ww_1^{(1)})-\bar \tS(\ww^{(2)}+\ww_1^{(2)}).
\end{align*}
This term is particularly important because the Lipschitz character of $\bar\psi$ follows from the Lipschitz property of $\tilde E$.  Let us further decompose:
\begin{align*}
\tilde E&=\chi\big(\frac{8\alpha\bar z}{\theta_0r_\alpha}\big)\{\bar \tS(\ww^{(1)}+\ww_1^{(1)})-\bar\tS(\ww^{(2)}+\ww_1^{(2)})\}\\
&\quad+\big[1-
\chi\big(\frac{8\alpha\bar z}{\theta_0r_\alpha}\big)\big]\{\bar \tS(\ww^{(1)}+\ww_1^{(1)})-\bar\tS(\ww^{(2)}+\ww_1^{(2)})\}\\
&=\tilde E_1+\tilde E_2.
\end{align*}
Notice that in the support of $\chi\big(\frac{8\alpha\bar z}{\theta_0r_\alpha}\big)$ we can assume (since $h^{(k)}_\alpha$ is small) that $\ww=w(\bar z)$, $\ww_1=w_1$.  Then we get:
\begin{align*}
\tilde E_1&=\chi\big(\frac{\alpha\bar z}{8\theta_0r_\alpha}\big)\{\Delta_{\Gamma^{(1)}_{\alpha,\bar z}}-\Delta_{\Gamma^{(2)}_{\alpha,\bar z}}\}(w+w_1)\\
&\quad  -\chi\big(\frac{\alpha\bar z}{8\theta_0r_\alpha}\big)\{H_{\Gamma^{(1)}_{\alpha, \bar z}}-
H_{\Gamma^{(2)}_{\alpha, \bar z}}\}\partial_{\bar z}w
\\
&=\tilde E_{11}+\tilde E_{12}.
\end{align*}
Using formula (\ref{non 11}) we get:
\begin{align}
\begin{aligned}
\|\tilde E_{11}\|_{p,\sigma,\nu}&\leq C\|\nabla^2_{\Gamma_{\alpha}}(h_\alpha^{(1)}-h_\alpha^{(2)})\|_{p,\nu}\\
&\quad+C (\|\nabla_{\Gamma_\alpha} h_\alpha^{(1)}\|_{\infty, \nu-1}+\|\nabla_{\Gamma_\alpha} h_\alpha^{(2)}\|_{\infty, \nu-1})\|\nabla_{\Gamma_\alpha}(h_\alpha^{(1)}-h_\alpha^{(2)})\|_{\infty, \nu-1}\\
&\quad +C (\|\nabla^2_{\Gamma_\alpha} h_\alpha^{(1)}\|_{p, \nu}+\|\nabla^2_{\Gamma_\alpha} h_\alpha^{(2)}\|_{p, \nu})\|h_\alpha^{(1)}-h_\alpha^{(2)}\|_{\infty, \nu-2}\\
&\leq C\|\nabla^2_{\Gamma_{\alpha}}(h_\alpha^{(1)}-h_\alpha^{(2)})\|_{p,\nu}\\
&\quad+C\alpha^{1+\mu}\|\nabla_{\Gamma_\alpha}(h_\alpha^{(1)}-h_\alpha^{(2)})\|_{\infty, \nu-1}+C\alpha^{2-8/p+\mu}\|h_\alpha^{(1)}-h_\alpha^{(2)}\|_{\infty, \nu-2}\\
&\leq C\alpha^{-8/p}\|h_\alpha^{(1)}-h_\alpha^{(2)}\|_{*,p,\nu},
\end{aligned}
\label{non 20}
\end{align}
(see (\ref{l16}) for the definition of $\|\cdot\|_{*,p,\nu}$).
Using similar argument as in (\ref{non 15}) we get as well:
\begin{align}
\|\tilde E_{12}\|_{p,\sigma, \nu}\leq C\alpha^2\|h_\alpha^{(1)}-h_\alpha^{(2)}\|_{\infty,\nu-2}.
\end{align}
To estimate $\tilde E_2$ we follow the same approach, again using (\ref{non 18a}) and the exponential smallness of $w(\bar z)\pm 1$ in the support of $1-
\chi\big(\frac{8\alpha\bar z}{\theta_0r_\alpha}\big)$.
As a consequence we get that
\begin{align}
\|\bar \tS(\ww^{(1)}+\ww_1^{(1)})-\bar \tS(\ww^{(2)}+\ww_1^{(2)})\|_{p,\sigma,\nu}\leq C\alpha^{-8/p}\|h_\alpha^{(1)}-h_\alpha^{(2)}\|_{*,p,\nu}.
\label{non 21}
\end{align}
From this and (\ref{oper}), denoting $\tilde \psi=\psi^{(1)}-\psi^{(2)}$, we get
\begin{align}
\begin{aligned}
\|\nabla\tilde \psi \|_{\infty,\nu} + \|\tilde\psi \|_{\infty,\nu} & \leq \,
C e^{\,-{\sigma\delta}/ {\alpha} }\, \big\|S^{(1)}[\ww]-S^{(2)}[\ww]\|_{p,\sigma, \nu}
\\&\quad +Ce^{\,-{2\sigma\delta}/{\alpha} }\big[\|\tilde\phi\|_{\infty,\sigma, \nu}+ \|\nabla\tilde \phi\|_{\infty,\sigma, \nu}\big]\\
&\leq C e^{\,-{\sigma\delta}/ {\alpha} }\alpha^{-8/p}\|h_\alpha^{(1)}-h_\alpha^{(2)}\|_{*,p,\nu}\\
&\quad +Ce^{\,-{2\sigma\delta}/{\alpha} }\big[\|\tilde\phi\|_{\infty,\sigma, \nu}+ \|\nabla\tilde \phi\|_{\infty,\sigma, \nu}\big].
\end{aligned}
\label{non 22}
\end{align}
Another important term to estimate in (\ref{non 19}) is
\begin{align}
\tilde E_3=(\tilde\cL^{(2)}-\tilde \cL^{(1)})\bar\phi^{(2)}.
\label{non 23a}
\end{align}
It is a matter of   rather tedious but  standard calculations to show that:
\begin{align}
\begin{aligned}
\|\tilde E_3\|_{p,\sigma,\nu}&\leq C\alpha^{-8/p}
\|h_\alpha^{(1)}-h_\alpha^{(2)}\|_{*,p,\nu}\|\bar\phi^{(2)}\|_{*,p,\sigma,\nu}\\
&\leq \alpha^{2-8/p}
\|h_\alpha^{(1)}-h_\alpha^{(2)}\|_{*,p,\nu}.
\end{aligned}
\label{non 23}
\end{align}
Here we use the fact that the coefficients of the derivatives  in the expressions in local coordinates  for $\tilde \cL^{(k})$ are smooth functions of $h_\alpha^{(k)}$ and that all terms involved have a total of at most 3 derivatives summing up both derivatives of  $h_\alpha^{(k)}$ and $\bar\phi^{(2)}$.

Using the a priori estimate for $\tilde \cL^{(1)}$ to estimate $\tilde\phi$ in (\ref{non 19}) we obtain the required estimate from (\ref{non 21})--(\ref{non 23}) and the Lipschitz character of the nonlinear term $\bar\NNN(\bar\phi^{(1)})-\bar\NNN(\bar\phi^{(2)})$. This ends the proof.

\qed

\medskip
This results of Proposition \ref{prop2} allow us to reduce the full nonlinear problem  to one dependent on $h_\alpha$. Indeed,  using the definition of $\phi, \psi$ and the fact that $u_\alpha=\ww+\eta_{2\delta}^\alpha\phi+\psi$ (see (\ref{glue 1})--(\ref{ualpha2})) we see that instead of the nonlinear problem
(\ref{ac}) we have found, for given $h_\alpha$ functions $u_\alpha$, $c_\alpha$ such that
\begin{align}
\Delta u_\alpha+f(u_\alpha)=\eta_{2\delta}^\alpha \bar c_\alpha w_{\bar z}(\bar z), \quad \bar z=z-h_\alpha, \inn \R^9.
\label{non 24a}
\end{align}
If we can adjust $h_\alpha$ in such a way that
\begin{align}
\bar c_\alpha\equiv 0,
\label{non 24}
\end{align}
then $u_\alpha$ in (\ref{non 24a}) is a solution we are looking for.
The theory we have already derived allows to derive a  relatively simple form of the {\it reduced problem} (\ref{non 24}). In the next section we will see that  it amounts to a
nonlocal PDE  for $h_\alpha$ which involves the Jacobi operator on $\Gamma_\alpha$ applied to $h_\alpha$ as its leading term.

\setcounter{equation}{0}
\section{Derivation of the reduced problem}\label{section red}

To derive the reduced problem we will go back to (\ref{proj}). Multiplying the equation by $w_{\bar z}(\bar z)$ and integrating over $\R$ with respect to $\bar z$  we get the following identity:
\begin{align*}
\int_{\R} \tilde \cL(\bar \phi) w_{\bar z}\, d\bar z=\bar c_\alpha\int_\R w^2_{\bar z}+\int_\R\eta_{\delta}^\alpha \bar \tS(w+w_1)w_{\bar z}\,d\bar z+\int_\R\bar \NNN(\bar \phi)w_{\bar z}\,d\bar z,
\end{align*}
hence (\ref{non 24}) is equivalent to:
\begin{align}
\int_\R\eta_{\delta}^\alpha \bar \tS(w+w_1)w_{\bar z}\,d\bar z=-\int_{\R} \tilde \cL(\bar \phi) w_{\bar z}\, d\bar z+\int_\R\bar \NNN(\bar \phi)w_{\bar z}\,d\bar z.
\label{red 1}
\end{align}
We will now calculate more explicitly all terms involved in (\ref{red 1}).

We will begin with
\begin{align*}
\int_\R\eta_{\delta}^\alpha \bar \tS(w+w_1)w_{\bar z}\,d\bar z&=\int_\R \eta_\delta^\alpha\Delta_{\Gamma_{\alpha, z}}w w_{\bar z}\,d\bar z -\int_R \eta_\delta^\alpha (H_{\Gamma_{\alpha, z}}-\bar z|A_{\Gamma_\alpha}|^2)w_{\bar z}^2\\
&\quad+\int_\R \eta_\delta^\alpha\Delta_{\Gamma_{\alpha, z}}w_1 w_{\bar z}\,d\bar z\\
&=M_1+M_2+M_3.
\end{align*}
Using the  local representation for $\Delta_{\Gamma_{\bar z}}w$ given  in (\ref{non 11}) we get
\begin{align*}
M_1&=-\int_\R\frac{1}{\sqrt{\det(\bar \tg_{\alpha, \bar z})}}\partial_{i}
\Big(\bar \tg_{\alpha, \bar z}^{ij}{\sqrt{\det(\bar \tg_{\alpha, \bar z})}}\partial_{j}h_\alpha\Big)w_{\bar z}^2\,d{\bar z}\\
&\quad+\int_\R\frac{1}{\sqrt{\det(\bar \tg_{\alpha, \bar z})}}\partial_{\bar z}\Big(\bar \tg_{\alpha,\bar z}^{ij}{\sqrt{\det(\bar \tg_{\alpha,\bar z})}}\Big)\partial_{i}h_\alpha\partial_{j}h_\alpha w^2_{\bar z}\,d\bar z\\
&\quad +\int_\R\bar\tg_{\alpha,\bar z}^{ij}\partial_{j}h_\alpha\partial_{i}h_\alpha w_{\bar z}^3\,d\bar z\\
&= M_{11}+M_{12}+M_{13}.
\end{align*}
We will start with:
\begin{align}
M_{11}=-c_0\Delta_{\Gamma_\alpha} h_\alpha+\tB_{\alpha 1}(h_\alpha),\quad  c_0=\int_\R w_z^2\,dz.
\label{red 2}
\end{align}
Let us fix $y_0\in \Gamma_\alpha$. The local norm of the the second order differential operator
$\tB_{\alpha 1}$ can be estimated as follows:
\begin{align}
\begin{aligned}
&\|r^{\nu+1}_\alpha\tB_{\alpha 1}(h_\alpha)\|^p_{L^p(\Gamma_\alpha\cap B(y,\theta_0\alpha^{-1}))}
\\
&\quad\leq
C\alpha \int_{B(0,2\theta_0\alpha^{-1})}r^{p(\nu+1)}_\alpha(y(t))|\nabla_{\Gamma_\alpha} h_\alpha|^{2p}\frac{dt}{[1+r_\alpha(y_0)]^p}\\
&\qquad+C\alpha \int_{B(0,2\theta_0\alpha^{-1})}r^{p(\nu+1)}_\alpha(y(t))|\nabla^2_{\Gamma_\alpha} h_\alpha|^{p}\frac{|h_\alpha|^p\, dt}{[1+r_\alpha(y_0)]^p}\\
&\qquad+C\alpha \int_{B(0,2\theta_0\alpha^{-1})}r^{p(\nu+1)}_\alpha(y(t))|\nabla_{\Gamma_\alpha} h_\alpha|^{p}\frac{dt}{[1+r_\alpha(y_0)]^{2p}}.
\end{aligned}
\label{red 3}
\end{align}
Notice that in the ball $B(0,2\theta_0\alpha^{-1})$ we have:
\begin{align}
\frac{1+r_\alpha(y(t))}{1+r_\alpha(y_0)}\leq C,
\label{red 3a}
\end{align}
by Lemma \ref{lem gr1}.
Hence, from the definition of the $\|\cdot\|_{*,p,\nu}$-norm and the assumption we have made on
$h_\alpha$, see (\ref{l16}), we get that:
\begin{align}
\|\tB_{\alpha 1}(h_\alpha)\|_{p, \nu+1}\leq  C\alpha^{1-8/p}\|h_\alpha\|_{*,p,\nu}.
\label{red 4}
\end{align}
Similarly we have, setting $\tB_{\alpha 2}(h_\alpha)=M_{12}+M_{13}$:
\begin{align}
\|\tB_{\alpha 2}(h_\alpha)\|_{p, \nu}\leq C\alpha^{1-8/p}\|h_\alpha\|_{*,p,\nu}.
\label{red 5}
\end{align}
To estimate $M_2$ we first use the expansion (\ref{l23}) to find:
\begin{align*}
M_2&=-h_\alpha|A_{\Gamma_\alpha}|^2\int_\R w_{\bar z}^2\,d\bar z-\int_\R(\bar z+h_\alpha)^2{\mathcal R}_{\alpha}w_{\bar z}^2\,d\bar z\\
&=-c_0h_\alpha|A_{\Gamma_\alpha}|^2-\int_\R\bar z^2{\mathcal R}_{\alpha}w_{\bar z}^2\,d\bar z+ \int_\R(2\bar z h_\alpha+h_\alpha^2){\mathcal R}_{\alpha}w_{\bar z}^2\,d\bar z
\end{align*}
Observe that here \blue{${\mathcal R}_{\alpha}={\mathcal R}_{\alpha}(y, \bar z+h_\alpha)$. We can further Taylor expand this function in terms of $(\bar z+h_\alpha)$ to get
\begin{align}
{\mathcal R}_{\alpha}(y,\bar z + h_\alpha)={\mathcal R}_{1,\alpha}(y)+(\bar z+h_\alpha){\mathcal R}_{2,\alpha}(y, \bar z+h_\alpha),
\label{red 5a}
\end{align}}
where
\begin{align}
{\mathcal R}_{1,\alpha}\sim\frac{ \alpha^3}{(1+r_\alpha^3)}, \quad
{\mathcal R}_{2,\alpha}\sim \frac{\alpha^4}{1+r_\alpha^4},
\label{red 5b}
\end{align}
by \blue{formula (\ref{hgammaz})} and Lemma \ref{lem gr1}. Then we can write, denoting $c_1=\int_\R \bar z^2 w_{\bar z}^2$,
\begin{align}
\begin{aligned}
M_2&=-c_0h_\alpha|A_{\Gamma_\alpha}|^2-c_1{\mathcal R}_{1,\alpha}\\
&\quad -\int_\R[\bar z^2(\bar z-h_\alpha){\mathcal R}_{2,\alpha}+(2\bar z h_\alpha+h_\alpha^2){\mathcal R}_{\alpha}]
w_{\bar z}^2\,d\bar z\\
&=-c_0h_\alpha|A_{\Gamma_\alpha}|^2-c_1{\mathcal R}_{1,\alpha}+\tB_{\alpha 3}(h_\alpha)
\end{aligned}
\label{red 5c}
\end{align}
We observe that since $\nu\in [2,3]$ therefore, from (\ref{red 5b}), we have:
\begin{align}
\|{\mathcal R}_{1,\alpha}\|_{p,\nu}\leq C\alpha^{3-8/p}.
\label{red 5d}
\end{align}
We notice that this is the only term that is of  order in $r^{\nu}_\alpha$, since  the rest of the terms computed so far (and those evaluated below) have weights  $r_\alpha^{\nu+1}$ in their norms.

From (\ref{red 5b})  we get:
\begin{align}
\|\tB_{\alpha 3}(h_\alpha)\|_{p, \nu+1}\leq C\alpha^{4-8/p}+C\alpha^{1-8/p}\|h_\alpha\|_{*,p,\nu},
\label{red 6}
\end{align}
Now we will estimate the terms involved in the projection of $\bar \cL(\bar\phi)$ onto $w_{\bar z}$. Using   the same notation as in in (\ref{l26}) we get, after integration by parts and also using the orthogonality condition
\begin{align}
\begin{aligned}
\int_\R\bar \cL(\bar\phi)w_{\bar z}\,d\bar z&=-\int_\R {\bf B}_{1\alpha}(\bar\phi)w_{\bar z}\,d\bar z\blue{+\int_\R \partial_{\bar z}\bar \phi\Delta_{\Gamma_{\alpha}} h_\alpha w_{\bar z}\,d\bar z}\\
&\quad
+\int_\R\eta^\delta_\alpha{\bf B}_{2\alpha}(\bar \phi)w_{\bar z}\,d\bar z\\
&\quad-\int_\R\eta^\delta_{\alpha}(\Delta_{\Gamma_{\alpha, \bar z}} h_\alpha+(\bar z+h_\alpha)|A_{\Gamma_\alpha}|^2)\partial_{\bar z}\bar \phi w_{\bar z}\,d\bar z
\\
&\quad - \int_\R\eta^\delta_{\alpha}(\bar z+h_\alpha)^2{\mathcal R}_{\alpha}\partial_{\bar z}\bar \phi w_{\bar z}\,d\bar z
\\
&=I_1+I_2+I_3+I_4.
\end{aligned}
\label{red 7}
\end{align}
Using the explicit formula for ${\bf B}_{1\alpha}(\bar\phi)$ given in (\ref{l17a}) we get after integrating by parts once with respect to $\bar z$:
\begin{align*}
\Big|\int_\R{\bf B}_{1\alpha}(\bar\phi)w_{\bar z}\,d\bar z\Big|\leq C|\nabla_{\Gamma_\alpha} h_\alpha|\int_\R  \big|D\bar\phi\big| \big|w_{\bar z\bar z}\big|\,d\bar z,
\end{align*}
 It follows that if $\nu\in [2,3]$ then:
\begin{align}
\begin{aligned}
&\|r^{\nu+1}_\alpha \int_\R{\bf B}_{1\alpha}(\bar\phi)w_{\bar z}\,d\bar z \|_{L^p(\Gamma_\alpha\cap B(y_0,\theta_0\alpha^{-1}))}\\
&\quad\leq C\|r_\alpha^{\nu-1}\nabla_{\Gamma_\alpha} h_\alpha\|_{\infty}\Big\{\int_{B(0,\theta_0\alpha^{-1})}r_\alpha^{p\nu}(y(t))\sup_{\bar z}\big[e^{\,p\sigma|\bar z|}|D\bar \phi| ^p\big]\,dt\Big\}^{1/p}
\\
&\quad\leq C\alpha^{-1-8/p}\|h_\alpha\|_{*,p,\nu}\|\bar \phi\|_{*,p, \sigma,\nu},
\end{aligned}
\label{red 8}
\end{align}
\blue{Similarly, we get:
\begin{align}
\|r^{\nu+1}_\alpha \int_\R \partial_{\bar z}\bar \phi\Delta_{\Gamma_{\alpha}} h_\alpha w_{\bar z}\,d\bar z\|_{L^p(\Gamma_\alpha\cap B(y_0,\theta_0\alpha^{-1}))}\leq C\|h_\alpha\|_{*,p,\nu}\|\bar \phi\|_{*,p, \sigma,\nu},
\label{red 8 a}
\end{align}
hence,
\begin{align}
\|r^{\nu+1}_\alpha I_1\|_{L^p(\Gamma_\alpha\cap B(y_0,\theta_0\alpha^{-1}))}\leq C\|h_\alpha\|_{*,p,\nu}\|\bar \phi\|_{*,p, \sigma,\nu}\leq C \alpha^{-1-8/p}\|h_\alpha\|_{*,p,\nu}\|\bar \phi\|_{*,p, \sigma,\nu}.
\label{red 8 b}
\end{align}}
Using (\ref{l22}) we get as well
\begin{align}
\|r^{\nu+1}_\alpha I_2\|_{L^p(\Gamma_\alpha\cap B(y_0,\theta_0\alpha^{-1}))}\leq C\alpha^{1-8/p}(1+\alpha^{-2}\|h_\alpha\|_{*,p,\nu})\|\bar \phi\|_{*,p, \sigma,\nu}.
\label{red 9}
\end{align}
We can also estimate jointly:
\begin{align}
\begin{aligned}
&\|r^{\nu+1}_\alpha I_3\|_{L^p(\Gamma_\alpha\cap B(y_0,\theta_0\alpha^{-1}))}+\|r^{\nu+1}_\alpha I_4\|_{L^p(\Gamma_\alpha\cap B(y_0,\theta_0\alpha^{-1}))}\\
&\quad\leq C\alpha^{2-8/p}\|\bar\phi\|_{*,p,\sigma,\nu}+C\alpha^{-8/p}\|\bar\phi\|_{*,p,\sigma,\nu}\|h_\alpha\|_{*,p,\nu}.
\end{aligned}
\label{red 11}
\end{align}
Finally, denoting
\begin{align*}
\int_\R\bar \NNN(\bar \phi)w_{\bar z}\,d\bar z=I_5,
\end{align*}
we get that
\begin{align}
\|r^{\nu+1}_\alpha I_5\|_{L^p(\Gamma_\alpha\cap B(y_0,\theta_0\alpha^{-1}))}
\leq C\|\bar\phi\|^2_{*,p,\sigma,\nu}.
\label{red 12}
\end{align}
Summarizing (\ref{red 1})--(\ref{red 12}) we get that $h_\alpha$ must be a solution of the following problem:
\begin{align}
\Delta_{\Gamma_\alpha}h_\alpha+|A_{\Gamma_\alpha}|^2 h_\alpha=c_1{\mathcal R}_{1,\alpha}+
{\cF}_\alpha(h_\alpha, \nabla_{\Gamma_\alpha}h_\alpha, \nabla^2_{\Gamma_\alpha} h_\alpha),
\label{red 13}
\end{align}
where the first term on the right hand side of (\ref{red 13}) satisfies (\ref{red 5d}) and \blue{formula (\ref{hgammaz}) and Lemma \ref{lem gr1} it is explicitly given by:
\begin{align}
{\mathcal R}_{1,\alpha}(y)=\sum_i^8\kappa^3_i(y),
\label{def r1alpha}
\end{align}
and} ${\cF}_\alpha$ is a nonlinear and nonlocal function of $h_\alpha$ and its first and second  derivatives that satisfies:
\begin{align}
\|\cF_\alpha\|_{p,\nu+1}\leq C\alpha^{1-8/p}\|h_\alpha\|_{*,p,\nu}+C\alpha^{3-8/p}.
\label{red 14}
\end{align}
The rest of this paper is devoted to solving the reduced problem (\ref{red 13}).
A natural way to do this is to argue by approximations on expanding balls $B_R$, as we have done  before in this paper. However an extra difficulty  in the case of the reduced problem is to derive a priori estimates (independent on $R$) for the Jacobi opertator in (\ref{red 13}). To deal with this problem we will consider an approximate Jacobi operator, which is the mean curvature linearized around $\Gamma_{0,\alpha}=\{x_9=\frac{1}{\alpha}F_0(\alpha x')\}$, rather than $\Gamma_\alpha$.

At this point we will also use the symmetry of the minimal graph. Let us recall that in reality $\Gamma_\alpha$ is a graph of a function $F_\alpha$ that satisfies
\begin{align*}
F_\alpha(u,v)=-F_\alpha(v,u),
\quad u^2=x_1^2+\cdots+x_4^2, v^2=x_5^2+\cdots+x_8^2.
\end{align*}
It is therefore natural to make the following assumption on $h_\alpha$:
\begin{align}
h_\alpha(u,v)=-h_\alpha(v,u).
\label{red 15}
\end{align}
\blue{Observe that  in particular the principal curvatures of $\Gamma$ satisfy
\begin{align*}
\kappa_i(u,v)=-\kappa_i(v,u),
\end{align*}
hence:
\begin{align}
{\mathcal R}_{1,\alpha}(u,v)=-{\mathcal R}_{1,\alpha}(v,u).
\label{red 17 0}
\end{align}}
Notice that the Fermi coordinate $z$ depends on $x$ only through $(u,v,x_9)$ and:
\begin{align*}
z(u,v,x_9)=-z(v,u,-x_9).
\end{align*}
From this it follows:
\begin{align*}
\ww(u,v,x_9)=-\ww(v,u,-x_9), \quad \ww_1(u,v,x_9)=-\ww_1(v,u,-x_9),
\end{align*}
and
\begin{align}
c_\alpha(u,v)=-c_\alpha(u,v), \quad \phi(u,v,x_9)=-\phi(v,u,-x_9).
\label{red 16}
\end{align}
In all, the right hand side of (\ref{red 13}) has the same type of symmetry as $h_\alpha$:
\begin{align}
{\cF}_\alpha(u,v)=-{\cF}_\alpha(v,u).
\label{red 17}
\end{align}
\blue{To put it differently: the procedure  that leads to determining $u_\alpha$ for a given $h_\alpha$  can be done in the sector $T\{u>0, v>0, v>u\}$ first and then the resulting solution can be extended to the whole space by using the natural symmetries of the minimal graph.}

\setcounter{equation}{0}
\section{Solvability theory for the Jacobi operator}\label{sec superl}

\subsection{The approximate Jacobi operator}

In this and the following section we will consider the Jacobi operator associated to a fixed  minimal surface setting the scaling parameter $\alpha=1$. We will denote:
\begin{align}
\Gamma=\{x_9=F(x')\}.
\label{superl 0}
\end{align}
By $\blue{A_\Gamma}$ we will denote the second fundamental form on $\Gamma$.
The Jacobi operator $\cJ$ is given by:
\begin{align}
\cJ(h)=\Delta_{\Gamma} h+|A_\Gamma|^2 h.
\label{superl 1}
\end{align}
\blue{A convenient form of the Jacobi operator is obtained using the natural parametrization of $\Gamma$ given by $\Gamma=\{(x',F(x'))\mid x'\in \R^8\}$. In these coordinates we get 
\begin{align}
\cJ(h)=H'[F](h\sqrt{1+|\nabla F|^2}),
\label{superl 1 a}
\end{align}}
where $H[F]$ is the mean curvature operator and $H'[F]$ is its linearization around $F$, namely:
\begin{align}
H'[F](\varphi)=\nabla\cdot\Big(\frac{\nabla\varphi}
{\sqrt{1+|\nabla F|^2}}-\frac{\nabla F(\nabla F\cdot\nabla\varphi)}
{(1+|\nabla F|^2)^{3/2}}\Big).
\label{superl 2}
\end{align}
We will  define the norms:
\begin{align}
\begin{aligned}
\|f\|_{p,\nu} &:=\ \sup_{y\in \Gamma}  \,\|r^\nu f\|_{L^p(\Gamma\cap B(y,\theta_0))},\quad 9<p<\infty,\quad\nu\geq 2,\\
\|f\|_{\infty,\nu} &:=
\sup_{y\in \Gamma}r(y)^\nu |f(y)|, \quad \nu \geq 2,
\end{aligned}
\label{superl 2a}
\end{align}
which are analogous to the weighted norms defined above,  and will be useful to treat question of invertibility  of the Jacobi operator.

We will go back now to the expression of $H$ in terms of the variables $(t,s)$ introduced in section
\ref{st}. We recall that:
\begin{align*}
H[F]=|\nabla F_0|\partial_t  \Big( \frac{ |\nabla F_0|\partial_t F} {\sqrt{1+|\nabla F|^2}}\Big)\,
+
|\nabla F_0|\partial_s\Big( \frac{ \rho^{-2} \partial_s  F} { |\nabla F_0| \sqrt{1+|\nabla F|^2}}\Big),
\end{align*}
where
\begin{align*}
\nabla F=F_t \nn F_0 + \rho^{-1} F_s
\frac{\nabla F_0^\perp }{|\nn F_0|}, \quad \rho=\frac{1}{(uv)^3},
\end{align*}
\blue{(see Lemma \ref{lemma ts} for the definition of the coordinates $u=u(t,s), v=v(t,s)$).}
The linearized mean curvature operator expressed in these variables takes  form:
\begin{align}
\begin{aligned}
H'[F](\varphi)&=|\nabla F_0|\partial_t  \Big(\frac{\partial_t\varphi(1+|\nabla F|^2)-\partial_t F(\nabla F\cdot\nabla\varphi)}{(1+|\nabla F|^2)^{3/2}}\Big)\\
&\quad+|\nabla F_0|\partial_s \Big(\frac{\rho^{-2}}{|\nabla F_0|}\frac{\partial_s\varphi(1+|\nabla F|^2)-\partial_s F(\nabla F\cdot\nabla\varphi)}{(1+|\nabla F|^2)^{3/2}}\Big).
\end{aligned}
\label{superl 3}
\end{align}
Let us now consider the linearized  of  mean curvature operator obtained by linearizing around the surface $\Gamma_0=\{(x',F_0(x'))\mid x'\in \R^8\}$, namely:
\begin{align}
H'[F_0](\varphi)=|\nabla F_0|\partial_t  \Big(\frac{|\nabla F_0|\partial_t\varphi}{(1+|\nabla F_0|^2)^{3/2}}\Big)+|\nabla F_0|\partial_s  \Big(\frac{\rho^{-2}}{|\nabla F_0|}\frac{\partial_s\varphi}{\sqrt{1+|\nabla F_0|^2}}\Big),
\label{superl 4}
\end{align}
where we have used the fact that
\begin{align*}
\partial_t F_0=1, \quad \partial_s F_0=0, \quad \nabla\varphi\cdot\nabla F_0=\partial_t\varphi|\nabla F_0|^2.
\end{align*}
We notice that for  $\Gamma_0$ 
the relation analogous to (\ref{superl 1 a}) holds, namely:
\begin{align}
\Delta_{\Gamma_0}h+|A_{\Gamma_0}|^2h=H'[F_0](h\sqrt{1+|\nabla F_0|^2}),
\label{superl 5}
\end{align}
where $\blue{A_{\Gamma_0}}$ is the second fundamental form on $\Gamma_0$. We will refer to the operator
defined above as the approximate Jacobi operator and denote it by $\blue{\tJ}$. The  reader should keep in mind that $\tJ$ as  the Jacobi operator associated to $\Gamma_0$ "approximates" $\cJ$.

\subsection{Supersolutions  for the operator $\tJ$}
In this section, we obtain supersolutions for the operator $\tJ(h)$
which is equivalent to finding supersolutions for
$H'[F_0](\varphi)$,
where $\varphi = h\sqrt{1+|\nn F_0|^2}$. Let us recall  the  definition of the set $T$:
\begin{align*}
T=\{(u,v)\mid u>0, v>0, u<v\},
\end{align*}
(see (\ref{def sector}) ) and the fact that in Lemma \ref{lemma ts} we have associated $T$ with the set $Q=\{t>0,s>0\}$.
\begin{lemma}
\label{lemma2}
For $ \sigma \in (-1, 0)$ and $ \sigma_1 \in [0, 1]$, there exist $r_0$ and $C>0$ such that  in the set
$T\cap\{R >r_0\}$ we have:
\begin{equation}
\label{mc19}
H'[F_0](r^\sigma t^{\sigma_1}) + \frac{C(g(\theta))^{\sigma_1}  }{r^{4-\sigma-3 \sigma_1
 }} \leq 0.
\end{equation}
Likewise  for $ \sigma \in (-1, 0)$ and $ \sigma_1 \in (0, 1)$, there holds in  $T\cap \{r >R_0\}$
\begin{equation}
\label{mc20}
H'[F_0](r^\sigma t^{\sigma_1}) + \frac{C}{r^{4-\sigma- \sigma_1 }} \leq 0.
\end{equation}

\end{lemma}

\noindent
\begin{remark}\label{rem superl 1} Note that $ 4-\sigma-3 \sigma_1 \in (1,5)$ with the choice of $ \sigma, \sigma_1$, while $ 4-\sigma-\sigma_1 \in (3, 5)$.
\end{remark}
\proof{}

Let us write
\begin{align*}
H'[F_0](\varphi)=\tilde L_0(\varphi)+\tilde L_1(\varphi),
\end{align*}
where
\begin{align}
\begin{aligned}
\tilde L_0(\varphi)&=|\nabla F_0|\partial_t  \Big(\frac{\partial_t\varphi}{|\nabla F_0|^2}\Big)+|\nabla F_0|\partial_s  \Big(\frac{\rho^{-2}\partial_s\varphi}{{|\nabla F_0|^2}}\Big),
\\
\tilde L_1(\varphi)&=\tilde L_0(\varphi)-H'[F_0](\varphi).
\end{aligned}
\label{superl 6}
\end{align}
We have
\begin{align*}
\pp_t ( r^{\sigma} t^{\sigma_1} ) & =
\sigma r^{\sigma-1}t^{\sigma_1} \frac{\pp r}{\pp t}  +\sigma_1 r^{\sigma }  t^{\sigma_1-1}\\
& =
\sigma r^{\sigma -1} t^{\sigma_1}\frac{ r}{ 3 t} \cos^2\phi  + \sigma_1r^{\sigma}  t^{\sigma_1-1}
\\
&=\sigma r^{\sigma} t^{\sigma_1}\frac{ 1}{ 3 t} \cos^2\phi  + \sigma_1 r^{\sigma}  t^{\sigma_1-1}
\end{align*}
and
\begin{align*}
\partial^2_{t} ( r^{\sigma} t^{\sigma_1} ) =r^{\sigma} \sigma_1 (\sigma_1-1) t^{\sigma_1-2} + O( r^{\sigma-6} t^{\sigma_1})
\end{align*}
Hence
\begin{align}
\begin{aligned}
|\nn F_0|  \pp_t \left ( \frac 1{|\nn F_0|^2} \partial^2_t ( r^{\sigma} t^{\sigma_1} \right ) &= \frac{1}{|\nn F_0|} \pp_{tt} ( r^{\sigma} t^{\sigma_1}) + |\nn F_0|  \partial_t  ( \frac 1{|\nn F_0|^2}) \pp_t ( r^{\sigma} t^{\sigma_1})\\
&
=  r^{-2 +\sigma} \sigma_1 (\sigma_1-1) t^{\sigma_1-2} + O( r^{\sigma-8} t^{\sigma_1}),
\end{aligned}
\label{mc23-2}
\end{align}
where we have used the fact that $|\nn F_0| \partial_t (\frac{1}{|\nn F_0|^2}) \sim -\frac{\cos \phi}{r^5}$.
Then we compute
$$
|\nn F_0| \pp_s \Big( \frac{\rho^{-2}}{|\nn F_0|^2} \pp_s r^\sigma \Big) =
|\nn F_0| \pp_s \Big( \frac{\rho^{-2}}{|\nn F_0|^2} \sigma  r^{\sigma -1} \pp_s r \Big) .
$$
Using formula \equ{s}  the above quantity thus equals
$$
C_0\sigma |\nn F_0|   \pp_s \left ( \frac{ r^{12}\sin^6 2\theta }{ r^4 (9g^2 + g_\theta^2 )} r^{\sigma -1} \frac r{s}\sin^2 \phi \right )
= C_1 \sigma |\nn F_0|\pp_s \left ( \frac{ s r^\sigma }{t^2} \cos^2 \phi \right ),
$$
where $C_i>0$ are generic positive constants, from now on. Now, using (\ref{fact1}), we obtain
\begin{align*}
&\pp_s \left ( \frac{ s r^\sigma }{t^2} \cos^2 \phi \right )\\
 &\quad=
\frac 1{t^2} \left ( r^\sigma \cos^2\phi + s\sigma  r^{\sigma  -1} \frac r{7s} \sin^2\phi \cos^2\phi
+ sr^\sigma  (-\sin 2\phi ) ( -\frac{\sin 2\phi} {14 s } ) \phi'\right )
\\
&=\quad\frac{ r^\sigma }{t^2} \left ( \cos^2\phi + \frac \sigma 7 \sin^2\phi \cos^2\phi + \frac{\sin^22\phi}{14}\phi' \right )
\\
&=\quad\frac{ r^\sigma }{t^{2}} \cos^2\phi \, \left ( 1 + \frac \sigma 7 \sin^2\phi  + \frac{2 \sin^2\phi}{7}\phi' \right ),
\end{align*}
hence
\begin{align*}
&t^{\sigma_1}|\nn F_0| \pp_s ( \frac{\rho^{-2}}{|\nn F_0|^2} \pp_s r^\sigma  )\\
&\qquad\qquad = \sigma C_3 \frac{r^{2+\sigma} }{t^{2-\sigma_1} } \cos^2\phi \sqrt{9g_\theta^2+g^2}\left ( 1 + \frac \sigma 7 \sin^2\phi  + \frac{2 \sin^2\phi}{7}\phi' \right ),
\end{align*}
and finally
\begin{align}
\label{78}
\begin{aligned}
|\nn F_0| \pp_s ( \frac{\rho^{-2}}{|\nn F_0|^2} \pp_s r^\sigma  )
&=\sigma C_3 \frac{t^{\sigma_1}}{ r^{4-\sigma}}\frac{1}{\sqrt{9g_\theta^2+g^2}} \left ( 1 + \frac \sigma 7 \sin^2\phi  + \frac{2 \sin^2\phi}{7}\phi' \right )
\\
&=\sigma C_3 \frac{t^{\sigma_1}}{ r^{4-\sigma}} a_1(\theta),
\end{aligned}
\end{align}
where $a_1(\theta)>0$. Then we obtain:
\begin{align}
\begin{aligned}
\tilde{L}_0 [ r^\sigma t^{\sigma_1}] &=
 r^{-4+\sigma+3\sigma_1}g^{\sigma_1} \Biggl(\frac{\sigma_1(1-\sigma_1)}{r^{-4} g^2}+\sigma C_3 a_1(\theta)\Biggr)+O( r^{-8+\sigma+3\sigma_1}) g^{\sigma_1}
\\
& \leq - \frac{Cg^{\sigma_1}}{ r^{4-3\sigma_1 -\sigma} },
\end{aligned}
\label{superl 7}
\end{align}
which proves (\ref{mc19}) with $H'[F_0]$, replaced by $\tilde L_0$.

For $-1 <\sigma <0$ we have $ \phi^{'} \geq -3$, hence
\be
\label{mc23-1}
|\nn F_0| \pp_s \Big( \frac{\rho^{-2}}{|\nn F_0|^2} \pp_s r^\sigma \Big)
 \leq  - \frac{C_4}{ r^{4-\sigma}}.
\ee
Combining  this and (\ref{mc23-2}), we obtain
\begin{align}
\begin{aligned}
\tilde{L}_0 [ r^\sigma t^\sigma_1]
&\leq -C \Biggl(  r^{-2 +\sigma}  t^{\sigma_1-2} + \frac{t^{\sigma_1} }{ r^{4-\sigma}} \Biggr)
\\
& \leq - \frac{C}{ r^{4-3\sigma_1 -\sigma} } \Biggl( \frac{ (\cos \phi)^{\sigma_1-2}}{r^4} + (\cos \phi)^{\sigma_1} \Biggr)
\\
&\leq - \frac{C_1  }{r^{4-\sigma- \sigma_1 }}.
\end{aligned}
\label{superl 8}
\end{align}
This proves (\ref{mc20}) with $H'[F_0]$ replaced by $\tilde L_0$.

To finish the proof one needs to estimate $\tilde L_1(r^\sigma t^{\sigma_1})$ and show that this term is of smaller order than $\tilde L_0(r^\sigma t^{\sigma_1})$. This is straightforward since
\begin{align*}
\frac{1}{|\nabla F_0|}-\frac{1}{\sqrt{1+|\nabla F_0|^2}}\sim \frac{1}{|\nabla F_0|^3}, \quad r\gg 1.
\end{align*}
We leave the details to the reader.
\qed

By $\tilde T$ we denote the following sector:
\begin{align}
\tilde T=\{(u,v)\mid v>0, |u|<v\}\subset \R^8.
\label{superl 9}
\end{align}
Since all functions involved in the proof of Lemma \ref{lemma2} are even with respect to $u$ in the set $\tilde T$ we immediately obtain:
\begin{corollary}\label{cor superl 1}
For $ \sigma \in (-1, 0)$ and $ \sigma_1 \in [0, 1]$, there exist $r_0$ and $C>0$ such that  in the set
$\tilde T\cap\{r >r_0\}$ we have:
\begin{equation}
\label{superl 10}
\tJ\big((1+|\nabla F_0|^2)^{-1/2}r^\sigma t^{\sigma_1}\big) + \frac{C_1 (g(\theta))^{\sigma_1}  }{r^{4-\sigma-3 \sigma_1
 }} \leq 0.
\end{equation}
Likewise  for $ \sigma \in (-1, 0)$ and $ \sigma_1 \in (0, 1)$, there holds in  $\tilde T\cap \{r >R_0\}$
\begin{equation}
\label{superl 11}
\blue{\tJ}\big((1+|\nabla F_0|^2)^{-1/2}r^\sigma t^{\sigma_1}\big) + \frac{C_1  }{r^{4-\sigma- \sigma_1 }} \leq 0.
\end{equation}
\end{corollary}

\subsection{The outer problem for $\tJ$}

We will use the supersolutions derived above to treat the following problem:
\begin{align}
\begin{aligned}
\tJ(h)&=f, \quad \mbox{in}\ \tilde T\cap \{r>R_0+1\},\\
h&=0, \quad \mbox{on}\ \partial(\tilde T\cap \{r>R_0+1\}),
\end{aligned}
\label{superl 11a}
\end{align}
where $R_0>r_0$ is fixed.
We will solve this problem by an  approximation scheme in extending domains:
\begin{align}
\begin{aligned}
\tJ(h_R)&=f, \quad \mbox{in}\ \tilde T\cap \{R>r>R_0+1\},\\
h_R&=0, \quad \mbox{on}\ \partial(\tilde T\cap \{R>r>R_0+1\}).
\end{aligned}
\label{superl 11b}
\end{align}
In this section we will consider the weighted norms defined in (\ref{superl 2a}) with $\Gamma$ replaced by $\Gamma_0$.
As for the right hand side of (\ref{superl 11a})  we assume that  one of the following holds:
\begin{enumerate}
\item Either $\nu=3$
\begin{align}
\|f\|_{\infty, \nu}<\infty, \quad \mbox{and}\ |f|\leq C\frac{g(\theta)^{\sigma_1}}{r^3}, \quad r>R_0,
\label{superl 11c}
\end{align}
with some $\sigma_1\in (1/3, 2/3)$;
\item or $\nu\geq 3+\mu$, $\mu\in (\frac{2}{3},1)$ and
\begin{align}
\|f\|_{\infty, \nu}<\infty.
\label{superl 11d}
\end{align}
\end{enumerate}
\begin{lemma}\label{lemma superl 2a}
Let $f$ be such that at least one of the two conditions (\ref{superl 11c}) or (\ref{superl 11d}) is satisfied. Then there exists a solution $h$ of (\ref{superl 11a}) such that:
\begin{align}
\|h\|_{\infty, \nu'-2}+\|\nabla_{\Gamma_0} h\|_{\infty, \nu'-1}+\|\nabla_{\Gamma_0}^2 h\|_{p,\nu'}\leq C\|f\|_{\infty,\nu},
\label{superl 11e}
\end{align}
where $\nu'\leq \nu$ satisfies:
\begin{align}
\nu'&=
\begin{cases}
3, \quad \mbox{if (\ref{superl 11c}) holds},\\
3+\mu',   \quad 0<\mu'<3\mu-2, \quad \mbox{if (\ref{superl 11d}) holds}.
\end{cases}
\label{superl 11f}
\end{align}
\end{lemma}
\proof{} We will solve (\ref{superl 11b}) and then take the limit $R\to \infty$. To fix attention we will consider $f$ such that (\ref{superl 11c}) holds, the other case being  similar.

We observe that an easy consequence of Lemma \ref{lemma2} is that (\ref{superl 11b}) has a unique solution for all $R>R_0+1$. Indeed taking $\sigma$, $\sigma_1$ such that $\sigma+3\sigma_1=2$ (say $\sigma=-\frac{1}{2}$, $\sigma_1=\frac{5}{6}$) we see that there is a bounded, positive supersolution of (\ref{superl 11b}) of the form:
\begin{align}
h_{\sigma, \sigma_1}=\frac{r^\sigma t^{\sigma_1}}{\sqrt{1+|\nabla F_0|^2}}.
\label{superl 11f1}
\end{align}
This mean that the homogeneous version of (\ref{superl 11b}) has only a trivial solution. By a similar argument we can prove that the operator $\tJ$ is non-degenerate for the outer problem (\ref{superl 11a}).  This means that the only \blue{vanishing at $\infty$} solution of the homogeneous version of (\ref{superl 11a}) is necessarily equal to $0$.

Now let $h_R$ be a solution of (\ref{superl 11b}). We claim that there exists $C>0$, independent on $R$ such that
\begin{align}
\|h_R\|_{\infty, 1}\leq C\|f\|_{\infty, 3}.
\label{superl 11g}
\end{align}
We will argue by contradiction. If (\ref{superl 11g}) does not hold then there exist sequences $R_n$,  $h_{R_n}$ and $f_n$ such that:
\begin{align}
\|f_n\|_{\infty, 3}\to 0, \quad \mbox{while}\ \|h_{R_n}\|_{\infty, 1}=1.
\label{superl 11h}
\end{align}
Taking function
\begin{align*}
h^+=C(1+\|f\|_{\infty, 3})h_{\sigma, \sigma_1}, \quad \sigma+3\sigma_1=1,
\end{align*}
with  a suitable constant $C$ (dependent on $R_0$ only) as a supersolution we see that for all $R_n$ sufficiently large  the supremum of $r h_{R_n}$  must be attained in a fixed compact set. Passing now to the limit we obtain a nontrivial solution of (\ref{superl 11a}) which contradicts   the non-degeneracy of $\tJ$. This proves  estimate (\ref{superl 11g}). The assertion of the Lemma follows now by elliptic estimates applied to the function $\tilde h_{R}=rh_R$. The proof is complete.
\qed

\subsection{An approximation scheme  for  the Jacobi operator}

We will consider the following problem
\begin{align}
\begin{aligned}
\cJ(h)&=f, \quad \inn \tilde T,\\
h&=0, \quad \mbox{on}\ \partial \tilde T.
\end{aligned}
\label{superl 12}
\end{align}
In this section we will in general assume that:
\begin{align}
\|f\|_{\infty,\nu}<\infty, \quad \mbox{with some}\ \nu> 4.
\label{superl 13}
\end{align}
We will solve (\ref{superl 12}) by approximations. For each sufficiently  large  $R$ we will consider:
\begin{align}
\begin{aligned}
\cJ(h_R)&=f, \inn \tilde T\cap B_R(0), \\
h_R&=0, \quad \mbox{on}\ \partial(\tilde T\cap B_R(0)).
\end{aligned}
\label{superl 14}
\end{align}
Our goal is to show the following:
\begin{proposition}\label{prop superl 1}
Consider a family of solutions of (\ref{superl 14}), $\{h_R\}$ with $f$ satisfying (\ref{superl 13}). As $R\to \infty$,  $h_R$ converges along a subsequence to a solution of (\ref{superl 12}). Moreover, denoting this solution by $h$,   there holds:
\begin{align}
\|h\|_{\infty,\nu'-2}+\|\nabla_{\Gamma} h\|_{\infty,\nu'-1}+\|\nabla^2_\Gamma h\|_{p,\nu'}\leq C\|f\|_{\infty,  \nu},
\label{superl 15}
\end{align}
where $4<\nu'<\nu$ and $p\in(9,\infty)$.
\end{proposition}
The proof of this Proposition follows from a series of Lemmas.
First  we will use Corollary \ref{cor superl 1}, and in particular (\ref{superl 11}) to show:
\begin{lemma}\label{lemma superl  1}
There exists $r_1>0$ such that   for $ \sigma \in (-1, 0)$ and $ \sigma_1 \in (0, 1)$, we have  in  $\tilde T\cap \{r >r_1\}$
\begin{equation}
\label{superl 16}
\blue{\cJ}\big((1+|\nabla F|^2)^{-1/2}r^\sigma t^{\sigma_1}\big) + \frac{C_1  }{r^{4-\sigma- \sigma_1 }} \leq 0.
\end{equation}
\end{lemma}
\proof{}
The proof is based on comparing the expressions for $\cJ$ and $\tJ$ in local coordinates and using formulas (\ref{coord 10}) and (\ref{coord 12}). We omit the details.
\qed

Next we will show that the operator $\cJ$ is non-degenerate:
\begin{lemma}\label{lemma superl 2}
Let $h$, such that
\begin{align}
\|h\|_{\infty, 2+\mu}<\infty, \quad \mbox{for some}\ \mu>0,
\label{superl 17}
\end{align}
be a solution of the following problem:
\begin{align}
\begin{aligned}
\cJ(h)&=0, \inn \tilde T, \\
h&=0, \quad \mbox{on}\ \partial\tilde T.
\end{aligned}
\label{superl 18}
\end{align}
Then we have $h\equiv 0$.
\end{lemma}
\proof{}
Let $\ve>0$ be a small number and let
\begin{align}
h_\ve(x)=\frac{\ve}{\sqrt{1+|\nabla F(x)|^2}}.
\label{superl 19}
\end{align}
Observe that $h_\ve$ satisfies:
\begin{align*}
\cJ(h_\ve)=0,
\end{align*}
and also, by (\ref{coord 6a}),
\begin{align*}
h_\ve\geq \frac{C\ve}{1+r^2}.
\end{align*}
It follows that for all  sufficiently large $R$ we have
\begin{align*}
h_\ve\geq h, \quad \mbox{on}\ \partial(\tilde T\cap B_R(0)),
\end{align*}
hence by comparison principle we have
\begin{align*}
h_\ve\geq h,\quad  \mbox{in}\ \tilde T.
\end{align*}
Taking $\ve\to 0$ the assertion of the Lemma follows.
\qed
\proof{(of the Proposition \ref{prop superl 1})}
Since for $R>0$ the function $h_\ve$ defined (\ref{superl 19}) is a positive supersolution of (\ref{superl 14}) therefore for each $R>0$ there exists a unique solution $h_R$ of (\ref{superl 14}).

We claim that there exists $C>0$ independent of $R$ such that for all $f$ we have:
\begin{align}
\|h_R\|_{\infty, 2+\mu'}\leq C\|f\|_{\infty, 4+\mu},
\label{superl 20}
\end{align}
where $\nu=4+\mu$  and $0<\mu'<\mu$.
To prove the claim we will argue by contradiction. Assume then that there exist $f_n$, $R_n$ and $h_{R_n}$ such that $h_{R_n}$ is a solution of (\ref{superl 14}) in $\tilde T\cap B_{R_n}(0)$ with $f=f_n$ and that
\begin{align}
\|f_n\|_{\infty, 4+\mu}\to 0, \quad \mbox{while}\ \|h_{R_n}\|_{\infty, 2+\mu'}=1.
\label{superl 21}
\end{align}
Using Lemma \ref{lemma superl  1} we will  construct a supersolution of (\ref{superl 14}) in the set $\tilde T\cap\{R_0<r<R\}$. Let $0<\mu'<\mu$ be given and let :
\begin{align*}
0<\ve=\frac{\mu-\mu'}{2},
\end{align*}
be fixed. Further let:
\begin{align*}
\sigma_1=\frac{\ve}{2}, \quad \sigma=-\mu+\frac{\ve}{2}.
\end{align*}
Then we have
$0>\sigma>-1$,  $1>\sigma_1>0$, and
\begin{align}
\sigma+3\sigma_1=-\mu+2\ve=\mu'.
\label{superl 21a}
\end{align}
With $\sigma$, $\sigma_1$ as above  function
\begin{align}
h^+_{R_n}=C(1+\|f_n\|_{\infty, 4+\mu})\frac{r^\sigma t^{\sigma_1}}{\sqrt{1+|\nabla F_0|^2}},
\label{superl 22}
\end{align}
with some constant $C>0$ is a positive supersolution of (\ref{superl 14}) in $\tilde T\cap\{R_0<r<R_n\}$. This means that  there exists $R'>R_0$ such that for all sufficiently large $R_n$ we have
\begin{align*}
|h_{R_n}|\leq  \frac{1}{2 r^{2+\mu'}}, \quad \blue{R'<r<R_n},
\end{align*}
which means that the supremum of $|h_{R_n} r^{2+\mu'}|$ is taken on a compact set contained in $B_{R'+1}(0)$. This allows us to pass to the limit $n\to \infty$ and conclude that the limiting function $\tilde h$ satisfies the assumptions of Lemma \ref{lemma superl 2} and hence $\tilde h\equiv 0$. This is a contradiction with the fact that $\|\tilde h\|_{\infty, 2+\mu'}=1$. The proof of the claim complete.

The assertion of the Proposition follows now by a standard argument. We omit the details.
\qed

\subsection{A gluing procedure for the reduced problem}
Given $f$ such that either (\ref{superl 11c}) or (\ref{superl 11d}) is satisfied we consider
\begin{align}
\begin{aligned}
\cJ(h)&=f, \quad \inn \tilde T, \\
h&=0, \quad \mbox{on}\ \partial\tilde T.
\end{aligned}
\label{superl 23}
\end{align}
Let us notice that the theory of the previous section allows us to invert the Jacobi operator (and solve (\ref{superl 23})   in a space of functions whose decay is faster than $r^{-4}$. However we expect that the right hand side of the reduced problem decays only like $r^{-3}$.
In order to deal with this difficulty we  will use a gluing procedure that will describe in what follows.
\begin{proposition}\label{prop suprl 3}
There exists a solution of problem (\ref{superl 23}) such that
\begin{align}
\|h\|_{\infty, \nu'-2}+\|\nabla_{\Gamma} h\|_{\infty, \nu'-1}+\|\nabla_{\Gamma}^2 h\|_{p, \nu'}\leq C\|f\|_{\infty, \nu},
\label{superl 24}
\end{align}
where
\begin{align}\nu'&=
\begin{cases}
3, \quad \mbox{if (\ref{superl 11c}) holds},\\
3+\mu',   \quad 0<\mu'<3\mu-2, \quad \mbox{if (\ref{superl 11d}) holds}.
\end{cases}
\label{superl 25}
\end{align}
\end{proposition}
\proof{} Let $\tilde h$  be the solution of the outer problem (\ref{superl 11a}). We will look for a solution of (\ref{superl 23}) in the form:
\begin{align}
h=\eta_R\tilde h+\tth,
\label{superl 26}
\end{align}
where $\eta_R$ is a cut of function such that $\eta_R(r)=0$ if $r<R$ and $\eta_R(r)=1$ for $r>R+1$, for some fixed $R>R_0+2$. Notice that in principle  function $\tilde h$ is defined on $\Gamma_0$ rather than on $\Gamma$ but  using the $(u,v)$ coordinates we can assume that $\tilde h$ is a function on $\Gamma$. Then we have:
\begin{align}
\begin{aligned}
\cJ(\tth)&=f-\cJ(\eta_R\tilde h), \quad \inn \tilde T, \\
\tth&=0, \quad \partial \tilde T.
\end{aligned}
\label{superl 26a}
\end{align}
 We have:
\begin{align}
\begin{aligned}
f-\cJ(\eta_R\tilde h)&=\eta_R\big(\tJ(\tilde h)-\cJ(\tilde h)\big)-\tilde h\Delta_\Gamma \eta_R-2\nabla_\Gamma\tilde h\cdot \nabla_\Gamma\eta_R+(1-\eta_R)f\\
&=\tilde f.
\end{aligned}
\label{superl 27}\end{align}
Observe that the last three terms   in (\ref{superl 27}) are compactly supported. On the other hand, using (\ref{coord 10}) and (\ref{coord 12}), we get that  if $\|\tilde h\|_{\infty, \nu'-2}<\infty$ then
\begin{align}
\|\eta_R\big(\tJ(\tilde h)-\cJ(\tilde h)\big)\|_{\infty, \nu'+1+\varsigma}<\infty,
\label{superl 28}
\end{align}
with some $\varsigma>0$. This means that
\begin{align}
\|\tilde f\|_{\infty, \nu}<\infty,
\label{superl 29}
\end{align}
with some $\nu>4$. This allows to  use Proposition \ref{prop superl 1} to solve (\ref{superl 26a}). Combining this with  the results of Proposition \ref{lemma superl 2a} we end the proof.
\qed

{

\subsection{The inverse of the Jacobi operator in $L^{p,\nu}$}
Notice that so far we have assumed that the right hand sides of the problems involving the operators $\cJ, \tJ$ are bounded in $L^{\infty, \nu}$. However in the case $\nu\geq 4$ we have  to deal with the right hand sides in $L^{p,\nu}$, where $p>9$. Now we will show how to overcome this technical difficulty.
We will prove first:
\begin{lemma}\label{lemma  superl 3}
Let us consider problem (\ref{superl 11a}) but now assuming that
with some  $\nu\geq 3+\mu$, $\mu\in (\frac{2}{3},1)$ we have
\begin{align}
\|f\|_{p, \nu}<\infty, \quad p>9.
\label{superl 30}
\end{align}
There exists a number $C>0$ such that for each $f$ with $\|f\|_{p,\nu}<\infty$ there is a solution
$h$ to problem (\ref{superl 11a})
with $\|h\|_{p,\nu'-2}<+\infty$, where $\nu'\leq \nu$ satisfies:
\begin{align}
\nu'=3+\mu',   \quad 0<\mu'<3\mu-2,
\label{superl 31}
\end{align}
This solution  satisfies the estimate
\begin{align}
\|\nabla^2_{\Gamma_0} h\|_{p,\nu'} + \|\nabla_{\Gamma_0} h\|_{\infty,\nu'-1} + \|h\|_{\infty,\nu'-2} \ \le \ C\, \|\red{f}\|_{p,\nu}.
\label{superl 32}
\end{align}
\end{lemma}
\proof{}
Let us set
$$
h =  r^{-\nu} \psi, \ttt f=r^\nu f,
$$
so that Problem (\ref{superl 11a})  reads
\begin{align}
\begin{aligned}
{\tJ} (\psi ) +r^{\nu}\psi\Delta_{\Gamma_0} r^{-\nu}+2 r^\nu\nabla_{\Gamma_0}\psi\cdot
\nabla_{\Gamma_0} r^{-\nu}&=\ttt f \inn \tilde T\cap\{r>R_0+1\},
\\
\psi &= 0\onn \partial (\tilde T\cap\{r>R_0+1\}).
\end{aligned}
\label{super 33}
\end{align}
We will denote:
\begin{align*}
\ttt\J(\psi)={\tJ} (\psi ) +r^{\nu}\psi\Delta_{\Gamma_0} r^{-\nu}+2 r^\nu\nabla_{\Gamma_0}\psi\cdot
\nabla_{\Gamma_0} r^{-\nu}.
\end{align*}
Let us consider now the following problem
for $\ttt f\in L^\infty$.
\begin{align}
\begin{aligned}
\la^2 \ttt {\J} ( \psi ) - M \psi &=  \ttt f  \inn \tilde T\cap\{r>R_0+1\}, \\
\psi &=0 \onn \partial (\tilde T\cap\{r>R_0+1\}),
\end{aligned}
\label{yy}
\end{align}
where $\la>0$, and $M>0$ is such that
\begin{align*}
\sup_{\Gamma_0}|A_{\Gamma_0}|^2<\frac{M}{2}.
\end{align*}
We easily check that there is a $\la_0>0$ such that whenever $\la<\la_0$
$$
\la^2 \ttt {\J} ( 1) -M \blue{<-} \frac {M}{2}  \inn \tilde T\cap\{r>R_0+1\},
$$
and hence problem \equ{yy} has a unique, bounded solution.
Let us scale out $\la$ by setting $\psi_\lambda(y)=\psi(\lambda y)$, where $y\in \Gamma_{0,\lambda}$. Then equation \equ{yy} takes the form,
\begin{align}
\begin{aligned}
\ttt {\J}_\la  ( \psi_\lambda ) -M \psi_\lambda &=  \ttt f_\lambda  \inn \tilde T_\lambda\cap\{r>\lambda^{-1}(R_0+1)\} ,
\\
\psi &=0 \onn \partial (\tilde T_\lambda\cap\{r>\lambda^{-1}(R_0+1)\}),
\end{aligned}
\label{yyy}
\end{align}
where
\begin{align*}
\ttt {\J}_\la  (\psi_\lambda) =
\Delta_{\Gamma_{0,\la}}\psi_\la+|A_{\Gamma_{0,\la}}|^2 \psi_\la+r^{\nu}\psi\Delta_{\Gamma_{0,\la}} r^{-\nu}+2 r^\nu\nabla_{\Gamma_{0,\la}}\psi\cdot
\nabla_{\Gamma_{0,\la}} r^{-\nu}.
\end{align*}
We claim that
there exists a number $C>0$ such that for all sufficiently small $\la$ the following holds:
any  bounded solution
 $\psi_\lambda $ of problem $\equ{yyy}$ satisfies
the a priori estimate
\be
\lambda^{-{8/p}}\|\nabla_{\Gamma_{0,\lambda}}^2\psi_\lambda \|_{p,0} +
\lambda\|\nabla_{\Gamma_{0,\la}}\psi_\lambda\|_{\infty,0} + \lambda^2\|\psi_\lambda\|_{\infty,0} \
\le\ C\la^{2{-8/p}}\|\ttt f_\la\|_{p,0}.
\label{ccc}\ee
We will prove  the existence of $C$ for which
\be
 \|\psi_\la \|_{\infty,0} \ \le\ C\,\la^{-{8/p}} \|\ttt f_\la\|_{p,0}.
\label{asss1}\ee
Assuming the opposite,  we have  sequences $\la  =\la _n$, $\ttt f_n$, $\psi _n$ for which problem \equ{yyy} is satisfied and
$$
   \|\psi _n \|_{\infty,0 }= 1, \quad \lambda^{-{8/p}}_n\|\ttt f_n\|_{p,0} \to 0 ,\quad \la _n \to 0 .
$$
Let us assume that
$y_n\in \Gamma_{0,\lambda_n}\cap\{r>\lambda_n^{-1}(R_0+1)\}$ is such that
$$ |\psi _n (y_n) |\to 1.
$$
Let us consider the local system of coordinates around $y_n$ given by the graph of the function $G_{0,n}(t)$, i.e
\begin{align*}
\tilde \Gamma_{\la_n}\cap B(y_n,\theta_0|y_n|)=\big\{\big(t, G_{0,n}(t)\big)\mid |t|< \theta_0|y_n|\big\}.
\end{align*}
and
define
 $$
\tilde
\psi _n (t) =\psi_n\big(y_n+(t,G_{0,n}(t)\big).
$$
Let us observe that  the components of the metric tensor associated with these local coordinates satisfy:
$$
\tg_{\la_n}=I+\frac{1}{r^2(\la_ny_n)}O(\la_n^2|t|^2),
$$
hence, locally over compacts:
$$
\tg_{\la_n}^{ij}\to \delta_{ij}.
$$
uniformly.
Let us observe further that:
$$
|A_{\Gamma_{0,\la_n}}|^2\leq \frac{M}{2},
$$
because of the definition of $M$.  Thus we can assume that, as $\la_n\to 0$ we have uniformly over compact sets:
\begin{align*}
\lim_{n\to \infty}|A_{\Gamma_{0,\la_n}}|^2\to \ttt a^*(t),
\end{align*}
where
\begin{align*}
|\ttt a^*(t)|\leq \frac{M}{2}.
\end{align*}
Furthermore, expressing the other coefficients in the definition of $\ttt {\J}_\la$ in local coordinates we get that:
\begin{align*}
|r^\nu\Delta_{\Gamma_{0,\la n}} r^{-\nu}|+|r^\nu\nabla_{\Gamma_{0,\la_n}} r^{-\nu}|\leq \frac{C\la_n}{r(\la_n y_n)},
\end{align*}
uniformly over  compact sets.

Standard elliptic estimates  give local uniform $C^1$
bound for $\tilde \psi _n$, which implies that
we may assume
$$
\tilde \psi _n \to \tilde \psi \ne 0,
$$
locally in $C^1$-sense over compacts.
We get in the limit
the equation
\begin{align}
\Delta \tilde\psi+\ttt a^*(t)\tilde\psi-M\tilde\psi=0, \quad \inn \R^8.
\label{superl 34}
\end{align}
Since by our assumption  $\tilde\psi$ is bounded, maximum principle yields  $\ttt\psi  =0$.
We have reached a contradiction, hence estimate
\equ{asss1} holds true.
The estimates for first and second derivatives
follow from local elliptic $L^p$-theory,
and the proof of the a priori estimate \equ{ccc} is concluded.

\medskip
Now, given $\blue{\ttt f}$ with $\|\blue{\ttt f}\|_{p,\nu}<+\infty$, existence of a solution to problem
\equ{yy} which satisfies estimate
\equ{ccc} follows  by approximating $\ttt f$ by a sequence of bounded functions whose $\|\ \|_{p,0} $-norm is controlled by that of $\ttt f$. The a priori
estimate itself yields uniqueness of such a solution.
Let us translate the result obtained in terms of $\blue{\psi=r^{\nu} h}$ in original variable. We have found that, fixed $\la>0$ sufficiently small,
there is a $C=C_\la>0$ such that
given $f$ with $\|f\|_{p,\nu}<+\infty$,  there is a unique solution $h := \tau(f)$ with $\|h\|_{\infty,\nu}<+\infty$ to the problem
\begin{align}
\begin{aligned}
\la^2{\tJ} (h) -M h &=  f,  \inn \tilde T\cap\{r>R_0+1\},\\
h &=0, \onn \partial(\tilde T\cap\{r>R_0+1\}),
\end{aligned}
\label{yyyy}
\end{align}
that satisfies the estimate
\be
\|\nabla_{\Gamma_\lambda}^2h\|_{p,\nu} + \|\nabla_{\Gamma_0} h\|_{\infty,\nu} + \|h\|_{\infty,\nu} \ \le \ C_\lambda
\|f\|_{p,\nu}.
\label{ap5}\ee
Let us consider now our original problem \equ{superl 11a}, and let us decompose
$$
h= \la^2\tau(f) + \ttt h .
$$
Then:
\begin{align*}
\tJ(h)&=\la^2\tJ\big(\tau(f)\big)+\tJ(\ttt h)\\
&=\tJ(\ttt h)+M\tau(f)+f.
\end{align*}
The equation in terms of $\ttt h$ reads
\begin{align}
\begin{aligned}
{\mathcal J} ( \ttt h ) &= -M\tau(f)    \inn \tilde T\cap\{r>R_0+1\},
\\
h &=0 \onn \partial(\tilde T\cap\{r>R_0+1\}).
\end{aligned}
\label{yyyyy}
\end{align}
Since $\|\tau(g)\|_{\infty,\nu}< +\infty$, Lemma \ref{lemma superl 2a}
yield the existence of a unique solution $\ttt h$ to problem $\equ{yyyyy}$ such that
$$
\|\nabla_{\Gamma_0}^2 \ttt h\|_{p,\nu'}+\|\nabla_{\Gamma_0} \ttt h\|_{\infty, \nu'-1}+\|\ttt h\|_{\infty,\nu'-2}  \le CM\|\tau(f)\|_{\infty,\nu} .
$$
This and  the a priori estimate \equ{ap5} concludes the proof of the proposition.
\qed

\begin{proposition}\label{prop  superl 4}
Let us consider problem (\ref{superl 23}) but now assuming that
with some  $\nu\geq 3+\mu$, $\mu\in (\frac{2}{3},1)$ we have
\begin{align}
\|f\|_{p, \nu}<\infty, \quad p>9.
\label{superl 35}
\end{align}
There exists a solution of problem (\ref{superl 23}) such that
\begin{align}
\|h\|_{\infty, \nu'-2}+\|\nabla_{\Gamma} h\|_{\infty, \nu'-1}+\|\nabla_{\Gamma}^2 h\|_{p, \nu'}\leq C\|f\|_{p, \nu},
\label{superl 36}
\end{align}
where
\begin{align}\nu'&=
3+\mu',   \quad 0<\mu'<3\mu-2.
\label{superl 37}
\end{align}
\end{proposition}
\proof{}
In order to establish (\ref{superl 36}) we first solve the outer problem (\ref{superl 11a}) using the regularization procedure  described in the Lemma \ref{lemma superl 3}. Second, we have to deal with the problem:
\begin{align}
\begin{aligned}
\cJ(h)&=f, \quad \inn \tilde T, \\
h&=0, \quad \mbox{on}\ \partial\tilde T.
\end{aligned}
\label{superl 38}
\end{align}
where now $f$ satisfies:
\begin{align*}
\|f\|_{p, \nu+1+\sigma}<\infty,
\end{align*}
with some $\sigma>0$. Notice that  this is basically the same problem as (\ref{superl 23}) except that the right hand side is only an $L^{p,\nu+1+\sigma}$ function. At this point we need to use the regularization procedure similar to the one described in Lemma \ref{lemma superl 3}. Namely we solve:
\begin{align}
\begin{aligned}
\lambda^2\cJ(h)-Mh&=f, \quad \inn \tilde T, \\
h&=0, \quad \mbox{on}\ \partial\tilde T.
\end{aligned}
\label{superl 39}
\end{align}
 The existence of a solution of this problem, such that:
\begin{align*}
\|\nabla_{\Gamma}^2 h\|_{p, \nu+1+\sigma}+\|\nabla_\Gamma h\|_{\infty,\nu+1+\sigma}+\|h\|_{\infty, \nu+1+\sigma}\leq  C_\lambda\|f\|_{p,\nu+1+\sigma}
\end{align*}
can be shown using essentially the same argument as the one in the proof of Lemma \ref{lemma superl 3}.  We omit the details. After this step  we conclude the proof of the proposition.

\qed

We will now summarize our results stating them in the form more suitable for the reduced problem on $\Gamma_\alpha$. Let us recall (see (\ref{red 13}))  that the basic problem we need to solve is:
\begin{align}
\begin{aligned}
\cJ_\alpha&=f_\alpha,\inn \tilde T,\\
h_\alpha&=0, \onn \partial\tilde T,
\end{aligned}
\label{superl 40}
\end{align}
where
\begin{align}
\cJ_\alpha(h_\alpha)=\Delta_{\Gamma_\alpha}h_\alpha+|A_{\Gamma_\alpha}|^2 h_\alpha.
\label{superl 41}
\end{align}
As for the right hand side of (\ref{superl 40})  we assume that  one of the following holds:
\begin{enumerate}
\item Either $\nu=3$
\begin{align}
\|f_\alpha\|_{\infty, \nu}<C\alpha^3, \quad \mbox{and}\ |f_\alpha|\leq C\frac{\alpha^3g(\theta)^{\sigma_1}}{r_\alpha^3}, \quad r>R_0,
\label{superl 42}
\end{align}
with some $\sigma_1\in (1/3, 2/3)$;
\item or $\nu=3$,  and
\begin{align}
\|f_\alpha\|_{p, \nu+1}<C\alpha^{3-8/p},
\label{superl 43}
\end{align}
which is consistent with (\ref{red 14}).
\end{enumerate}
First we state a counterpart of Proposition:
\begin{proposition}\label{prop suprl 5}
Let us assume that (\ref{superl 42}) holds.
There exists a solution of problem (\ref{superl 40}) such that
\begin{align}
\alpha^2\|h_\alpha\|_{\infty, 1}+\alpha\|\nabla_{\Gamma_\blue{\alpha}} h_\alpha\|_{\infty, 2}+\alpha^{-8/p}\|\nabla_{\Gamma_\blue{\alpha}}^2 h\|_{p, 3}\leq C\|f_\alpha\|_{\infty, 3}.
\label{superl 44}
\end{align}
\end{proposition}
Next we state a suitable  modification of Proposition \ref{prop  superl 4}.
\begin{proposition}\label{prop  superl 6}
Let us consider problem (\ref{superl 40}) but now assuming that (\ref{superl 43}) holds.
There exists a solution of problem (\ref{superl 40}) such that
\begin{align}
\|h_\alpha\|_{*,p,3}\leq C\alpha^{-8/p}\|f_\alpha\|_{p, 4},
\label{superl 45}
\end{align}
where
\begin{align*}
\|h_\alpha\|_{*,p,3}=\alpha^2\|h_\alpha\|_{\infty, 1}+\alpha\|\nabla_{\Gamma_\alpha} h\|_{\infty, 2}+\alpha^{-8/p}\|\nabla_{\Gamma_\alpha}^2 h\|_{p, 3},
\end{align*}
(c.f.  definition of $\|\cdot\|_{*,p,\nu}$ in  (\ref{l16})).
\end{proposition}

\setcounter{equation}{0}
\section{Resolution of  the reduced problem}\label{sec sred}

\subsection{Improvement of the initial approximation of $h_\alpha$}
Let us recall the reduced problem derived in Section \ref{section red} (see (\ref{red 13}), (\ref{red 15})):
\begin{align}
\begin{aligned}
\cJ(h_\alpha)&=c_1{\mathcal R}_{1,\alpha}+
{\cF}_\alpha(h_\alpha, \nabla_{\Gamma_\alpha}h_\alpha, \nabla^2_{\Gamma_\alpha} h_\alpha),\inn \tilde T,
\\
h_\alpha&=0, \onn  \partial\tilde T.
\end{aligned}
\label{sred 1}
\end{align}
Notice that  the  boundary conditions imposed above allow to solve (\ref{red 15}) with the symmetry condition (\ref{red 15}) simply by extending the solution of (\ref{sred 1}) to the whole space.

Our plan is to solve (\ref{sred 1}) in two steps:
\begin{enumerate}
\item  We  find the leading order term in the expansion of $h_\alpha$.
\item We use a fixed   point argument to determine $h_\alpha$.
\end{enumerate}
In this section we will perform Step 1.  To begin with let $\Gamma_{0,\alpha}$ be  the surface:
\begin{align*}
\Gamma_{0,\alpha}=\Big\{x_9=\frac{1}{\alpha}F_0(\alpha x')\Big\},
\end{align*}
and
\begin{align*}
\Gamma_{0,\alpha, z}=\{x\in \R^9\mid \mbox{dist}\,(x,\Gamma_{0,\alpha}=z\}.
\end{align*}
The mean curvature of $\Gamma_{0,\alpha, z}$ can be expanded as follows:
\begin{align}
H_{\Gamma_{0,\alpha, z}}&=H_{\Gamma_{0,\alpha}}+z|\tilde A_{\Gamma_{0,\alpha}}|^2+\frac{1}{2} z^2 {\mathcal R}_{01,\alpha}+O\big(\frac{\alpha^4|z|^3}{1+r_\alpha^4}\big).
\label{sred 2}
\end{align}
Let us consider  ${\mathcal R}_{01,\alpha}$. This term is given by
\begin{align*}
\sum_{i=0}^8 \kappa^3_{0i,\alpha}, 
\end{align*}
where 
 principal $\kappa^3_{0i,\alpha}$ are the principal curvatures of $\Gamma_{0,\alpha}$. It follows that it has the same symmetry as the function $F_0$, in other words:
\begin{align*}
{\mathcal R}_{01,\alpha}(u,v)=-{\mathcal R}_{01,\alpha}(v,u),
\end{align*}
hence ${\mathcal R}_{01,\alpha}(u,u)=0$ and  there exists a  $\sigma_1\in (\frac{1}{3},\frac{2}{3})$ such that
\begin{align}
|{\mathcal R}_{01,\alpha}|\leq \frac{C\alpha^3 g(\theta)^\sigma_{1}}{1+r_\alpha^3}.
\label{sred 3}
\end{align}
Also we notice that a similar term in the expansion of $H_{\Gamma_{\alpha,z}}$, denoted above by ${\mathcal R}_{1,\alpha}$ is also equal to the sum of the cubes the principal curvatures of $\Gamma_{\alpha}$. It follows that (see (\ref{def r1alpha}), (\ref{red 17 0})):
\begin{align}
|{\mathcal R}_{1,\alpha}-{\mathcal R}_{01,\alpha}|\leq \frac{C\alpha^{4+\sigma}}{1+r_\alpha^{4+\sigma}},
\label{sred 4}
\end{align}
with some $\sigma>0$, using (\ref{eshm}).

Now going back to (\ref{sred 1}) we let:
\begin{align}
h_{\alpha}=\tilde h_{\alpha}+\tt h_{\alpha},
\label{sred5}
\end{align}
where
\begin{align}
\begin{aligned}
\cJ(\tilde h_\alpha)&=c_1{\mathcal R}_{1,\alpha},\inn \tilde T,
\\
\tilde h_\alpha&=0, \onn  \partial\tilde T.
\end{aligned}
\label{sred 6}
\end{align}
\begin{lemma}\label{lemma sred 1}
For each sufficiently small $\alpha$ there exists a solution of (\ref{sred 6}) such that
\begin{align}
\|\tilde h_\alpha\|_{*,p,3}\leq C\alpha^{3},
\label{sred 7}
\end{align}
provided that $p>9$ is taken sufficiently large.
\end{lemma}
\proof{}
We will write $\tilde h_\alpha=\tilde h_{1,\alpha}+\tilde h_{2,\alpha}$, where
\begin{align}
\begin{aligned}
\cJ(\tilde h_{1,\alpha})&=c_1{\mathcal R}_{01,\alpha},\inn \tilde T,\\
h_{1,\alpha}&=0,\onn \tilde T,
\end{aligned}
\label{sred 8}
\end{align}
and
\begin{align}
\begin{aligned}
\cJ(\tilde h_{2,\alpha})&=c_1 ({\mathcal R}_{1,\alpha}-{\mathcal R}_{01,\alpha}),\inn \tilde T,\\
h_{2,\alpha}&=0,\onn \tilde T.
\end{aligned}
\label{sred 9}
\end{align}
We notice that problem (\ref{sred 8}) has a solution satisfying (\ref{sred 7}) because of Proposition \ref{prop suprl 5} and estimate (\ref{sred 3}). In addition problem (\ref{sred 9}) has a solution satisfying:
\begin{align*}
\|\tilde h_{2,\alpha}\|_{*,p, 3}\leq C \alpha^{4+\sigma-8/p}\leq C\alpha^4,
\end{align*}
if  $p$ is  large, by Proposition \ref{prop superl 6}.
This completes the proof of the Lemma.
\qed

\subsection{The fixed point argument}

With $\tilde h_\alpha$ given by  Lemma \ref{lemma sred 1} we will use the theory of solvability for the Jacobi operator to define a map $\cT_\alpha$ on a subset of  a space of function whose
$\|\cdot\|_{*,p,3}$ norm is bounded into itself. Let $\mu>0$ be a fixed small number and let $p>9$ be large so that
\begin{align*}
\mu<1-32/p.
\end{align*}
Let us set:
\begin{align}
\cB_{\alpha^{2+\mu}}=\{\tf_\alpha\mid \|\tf_\alpha\|_{*,p,3}\leq \alpha^{2+\mu}\}.
\label{sred 10}
\end{align}
Given an $\tf\in B_{\alpha^{2+\mu}}$ we let $\tth_\alpha$ to be a solution of:
\begin{align}
\begin{aligned}
\cJ(\tth_\alpha)&=
\tilde {\cF}_\alpha(\tf_\alpha, \nabla_{\Gamma_\alpha}\tf_\alpha, \nabla^2_{\Gamma_\alpha}\tf_\alpha),\inn \tilde T,
\\
\tth_\alpha&=0, \onn  \partial\tilde T,
\end{aligned}
\label{sred 11}
\end{align}
where:
\begin{align*}
\tilde {\cF}_\alpha(\tf_\alpha, \nabla_{\Gamma_\alpha}\tf_\alpha, \nabla^2_{\Gamma_\alpha}\tf_\alpha)= {\cF}_\alpha(\tf_\alpha+\tilde h_\alpha, \nabla_{\Gamma_\alpha}(\tf_\alpha+\tilde h_\alpha), \nabla^2_{\Gamma_\alpha}(\tf_\alpha+\tilde h_\alpha)).
\end{align*}
Now we define:
\begin{align*}
\cT_\alpha(\tf_\alpha)=\tth_\alpha.
\end{align*}
We observe that by (\ref{red 14}) we have:
\begin{align}
\begin{aligned}
\|\tilde \cF_\alpha\|_{p,4}&\leq C\alpha^{1-8/p}\|\tf_\alpha\|_{*,p,3}+C\alpha^{1-8/p}\|\tilde h_\alpha\|_{*,p,3}+C\alpha^{3-8/p}\\
&\leq C\alpha^{3+\mu-8/p}+C\alpha^{4-8/p}+C\alpha^{3-8/p}\\
&\leq \alpha^{2+\mu+16/p},
\end{aligned}
\label{sred 12}
\end{align}
provided that $\alpha$ is taken sufficiently small, since $3-32/p>2+\mu$.
Then Proposition \ref{prop superl 6} implies that $\cT_\alpha$ is well defined, indeed since by (\ref{superl 45}) we have:
\begin{align*}
\|\tth_\alpha\|_{*,p,3}&\leq C\alpha^{-8/p}\|\tilde \cF_\alpha\|_{p,4}\\
&\leq C\alpha^{2+\mu+8/p}\\
&<\frac{1}{2}\alpha^{2+\mu},
\end{align*}
again by the choice of $p$, taking $\alpha$ small enough.
We will now prove:
\begin{lemma}\label{lemma sred 2}
Mapping $\cT_\alpha$ has a unique fixed point in $\cB_{\alpha^{2+\mu}}$.
\end{lemma}
\proof{}
In view of (\ref{sred 12}) to use Banach fixed point theorem we need to show that $\cT_\alpha$ is a contraction map.  Let $\tf^{(j)}_\alpha\in \cB_{\alpha^{2+\mu}}$ be fixed and $\tth^{(j)}_\alpha=\cT_\alpha(\tf_\alpha^{(j)})$, $j=1,2$. We claim that
\begin{align}
\||\tilde \cF_\alpha(\tf^{(1)}_\alpha)-\tilde \cF_\alpha(\tf^{(2)}_\alpha)\|_{p,4}\leq C\alpha^{1-8/p}\|\tth^{(1)}_\alpha-\tth^{(2)}_\alpha\|_{*,p,3}.
\label{sred 13}
\end{align}
This amounts to calculations similar as in the Section \ref{section red} but taking into account two  solutions $\phi^{(j)}$, $j=1,2$ of the projected nonlinear problem and subtracting the resulting projections. The key estimates are  (\ref{non 10a}) and also  (\ref{non 23a})--(\ref{non 23}).  The somewhat tedious details are omitted here. From (\ref{sred 13}) we conclude that $\cT_\alpha$ is Lipschitz with a constant proportional to $\alpha^{1-8/p}$. Taking $\alpha$ smaller if necessary we show that $\cT_\alpha$ is a contraction map.

\qed

\section{Conclusion of the proof of the main theorem}

Let us summarize the results of our considerations so far. Given the solution to the nonlinear projected problem $\phi$ and the corresponding solution $h_\alpha$ to the reduced problem  found above we have found a  function $u_\alpha$ such that
\begin{align*}
u_\alpha=\tt w+\tt w_1 +\eta^\alpha_{2\delta} \phi + \psi(\phi),
\end{align*}
and
\begin{align*}
\Delta u_\alpha +u_\alpha(1-u_\alpha)=0, \inn \R^9.
\end{align*}
Clearly $u_\alpha$ is a bounded function.  Also  $u_\alpha$ obeys the  symmetry of the minimal graph $\Gamma_\alpha$:
\begin{align}
u_\alpha(u,v,x_9)=-u(v,u,-x_9),
\label{conc 1}
\end{align}
from which it follows in particular
\begin{align*}
u_\alpha(0)=0.
\end{align*}
To finish the result of the theorem, we need to prove  that
the solution $u_\alpha$ of the Allen-Cahn equation obtained this way is
in fact monotone in the $x_9$-direction.

Observe that the function $\psi_\alpha := \partial_{x_9} u_\alpha$ is a solution of the linear problem
$$
\Delta \psi_\alpha + f'(u_\alpha) \psi_\alpha = 0 .
$$
We claim that the construction yields that inside any bounded  neighborhood of $\Gamma_\alpha$   of the form $\cN_M=\{\mbox{dist}\,(x,\Gamma_\alpha)<M\}$ we have
$\psi_\alpha>0$.  Indeed,
\begin{align*}
\partial_{x_9}u_\alpha&=\partial_{x_9} w(z-h_\alpha)+O\big(\frac{\alpha^2}{1+r_\alpha^2}\big)\\
&=w'(z-h_\alpha)\partial_{x_9} z+O\big(\frac{\alpha^2}{1+r_\alpha^2}\big),
\end{align*}
where $z$ is the  Fermi coordinate of $\Gamma_\alpha$. We see that if $|z|$ is bounded then
\begin{align*}
\partial_{x_9} z\sim \frac{1}{\sqrt{1+|\nabla F_\alpha|^2}}=O\big(\frac{1}{1+r^2_\alpha}\big),
\end{align*}
by (\ref{coord 6}).
This shows our claim.
Taking  $M$  sufficiently large (but independent on $\alpha$)  we can achieve $f'(u_\alpha)>-3/2$ outside of  $\cN_M$.  We claim that we cannot have that $\psi_\alpha<0$ in $\cN_M^c$.
Indeed, a non-positive  local minimum of $\psi_\alpha$  is discarded by maximum principle. If there were  a sequence of points $x_n\in \R^9$, such that
\begin{align*}
\psi_\alpha(x_n)<0,
\end{align*}
$|x_n|\to \infty$,  and at the same time $\mbox{dist}\,(x_n,\Gamma_\alpha)>M$, for some large $M$,   the usual
compactness argument would give us a nontrivial bounded solution of
$$
\Delta \psi -c\psi = 0  \inn \R^9,\quad c(x)>1,
$$
hence a contradiction. We conclude that $\psi_\alpha>0$ in entire $\R^9$ and the proof of the theorem is concluded. \qed

}


\setcounter{equation}{0}
\section{Appendix A}\label{apendix a}
In this appendix we will provide the details of the computations needed in the proof of Lemma \ref{sup4}. We will collect first some terms appearing in the expansion formula (\ref{a terms}). We have
$\varphi(t,s)=t r^{-\sigma}$ and
\begin{align}
\begin{aligned}
\partial_t\varphi&=\frac{1}{r^\sigma} ( 1 -\frac{\sigma}{3} \cos^2 \phi),\quad \partial_s\varphi=
-\frac{\sigma t \sin^2 \phi}{ 7 r^\sigma s} \\
\partial_t^2\varphi &=\frac{C \sigma}{9 r^\sigma t} \cos^2 \phi [ \sigma \cos^2 \phi - 3 +2  \phi^{'} \sin^2 \phi  ]\\
\partial_{ts}\varphi&=-\frac{\sigma\sin^2\phi}{7r^\sigma s}\Big(1+\frac{2\phi'\cos^2\phi}{3}
-\frac{\sigma\cos^2\phi}{3}\Big)
\end{aligned}
\label{apa 1}
\end{align}
We also have
\begin{align}
\begin{aligned}
\frac{\rho^{-2} \varphi_s}{|\nabla F_0|^2}&= - \frac{ 7 s \sigma  \cos^2 \phi}{9 t r^{\sigma}}\\
\partial_s \Big(\frac{\rho^{-2} \varphi_s}{|\nabla F_0|^2}\Big)&= -\frac{ 7  \sigma \cos^2 \phi}{ 9 t r^\sigma}\Big [ 1+ \frac{ \sin^2 \phi}{7} (2\phi^{'} -\sigma)\Big]\\
\frac{\rho^{-2} \varphi_s^2}{|\nabla F_0|^2}&=  \frac{\sigma^2 \sin^2 \phi \cos^2 \phi}{9  r^{2\sigma}}\\
\partial_s \Big(\frac{\rho^{-2} \varphi_s^2}{|\nabla F_0|^2}\Big)&= - \frac{2 \sigma^2 \sin^2 \phi \cos^2 \phi}{63  r^{2 \sigma} s} ( \sigma \sin^2 \phi +  \phi^{'} \cos (2\phi) )
\\
\partial_t \Big(\frac{\rho^{-2} \varphi_s^2}{|\nabla F_0|^2}\Big)&= \frac{ 2  \sigma^2 \sin^2 \phi \cos^2 \phi}{ 27 t r^{2\sigma}} [ -\sigma  \cos^2 \phi +  \phi^{'}\cos (2\phi)  ]
\end{aligned}
\label{apa 2}
\end{align}
Using formula (\ref{a terms}) we get
\begin{align}
\begin{aligned}
H_2&=-\frac{1}{2}\partial_t \Big(\frac{\rho^{-2} \varphi_s^2}{|\nabla F_0|^2}\Big)
+2\partial_t\varphi \partial_s \Big(\frac{\rho^{-2} \varphi_s}{|\nabla F_0|^2}\Big)
- \Big(\frac{\rho^{-2} \varphi_s}{|\nabla F_0|^2}\Big)\partial_{ts}\varphi\\
H_3&=\Big(\frac{\rho^{-2} \varphi^2_s}{|\nabla F_0|^2}\Big)\partial^2_t\varphi-\frac{1}{2}\partial_t\varphi \partial_t\Big(\frac{\rho^{-2} \varphi^2_s}{|\nabla F_0|^2}\Big)+\Big[(\partial_t\varphi)^2+\Big(\frac{\rho^{-2} \varphi^2_s}{|\nabla F_0|^2}\Big)\Big]\partial_s\Big(\frac{\rho^{-2} \varphi_s}{|\nabla F_0|^2}\Big)\\
&\quad-\partial_t\varphi\partial_{ts}\varphi \Big(\frac{\rho^{-2} \varphi_s}{|\nabla F_0|^2}\Big)
-\frac{1}{2}\Big(\frac{\rho^{-2} \varphi_s}{|\nabla F_0|^2}\Big)\partial_s\Big(\frac{\rho^{-2} \varphi^2_s}{|\nabla F_0|^2}\Big)
\end{aligned}
\label{apa 3}
\end{align}
From (\ref{apa 1})--(\ref{apa 3}) we get by direct calculation
\begin{align}
\begin{aligned}
H_2&=\frac{ \sigma^2 \sin^2 \phi \cos^2 \phi}{ 27 t r^{2 \sigma}} [\sigma \cos^2 \phi  - \cos (2\phi) \phi^{'}] \\
&\quad-\frac{ 2\sigma \cos^2 \phi}{27 t r^{2\sigma}} (3-\sigma \cos^2 \phi) [7 + (2\phi^{'}-\sigma) \sin^2 \phi]
\\
&\quad +  \frac{ \sigma \cos^2 \phi \sin^2 \phi}{27 t r^{2\sigma}} ( 3-\sigma \cos^2 \phi +2 \cos^2 \phi \phi^{'} )\\
&=  \frac{\sigma \cos^2 \phi}{27 t r^{2\sigma}} [ - 6 ( 7+2 \phi^{'} \sin^2 \phi) + (3 + 2 \cos^2 \phi \phi^{'})\sin^2 \phi  +  O(\sigma) ]\\
&=  \frac{\sigma \cos^2 \phi}{27 t r^{2\sigma}} [ - 42+ \sin^2 \phi (-12\phi^{'} +3+2\phi' \cos^2 \phi) + O(\sigma)] <0,
\end{aligned}
\label{apa 4}
\end{align}
and
\begin{align}
\begin{aligned}
H_3&=\frac{\sigma^2 \sin^2 \phi \cos^2 \phi}{81 t r^{3 \sigma}} [ \sigma \cos^2 \phi  - 3 \cos 2\phi   ] \phi^{'}\\
&\quad- \frac{ \sigma \cos^2 \phi}{ 81 t r^{3\sigma}} (  9- 6 \sigma \cos^2 \phi + \sigma^2 \cos^2 \phi) (7+ (2 \phi^{'} -\sigma) \sin^2 \phi)\\
&\quad
+ \frac{\sigma^3 \sin^2 \phi \cos^4 \phi}{ 81 t r^{3\sigma}} ( \sigma \sin^2 \phi  + \cos (2\phi) \phi^{'}) \\
&\quad+ \frac{ \sigma \sin^2 \phi \cos^2 \phi}{ 81 t r^{3 \sigma}} ( 3-\sigma \cos^2 \phi)( 3-\sigma \cos^2 \phi +2 \cos^2 \phi \phi^{'})\\
&=\frac{ \sigma  \cos^2 \phi}{ 27 t r^{3\sigma}}  \left [  \sin^2 \phi (3+2 \cos^2 \phi \phi^{'}) -3 ( 7+ (2\phi^{'} -\sigma) \sin^2 \phi)\right.
\\
&\qquad\left. -\sigma \sin^2 \cos (2\phi) \phi^{'} + O(\sigma^2 \cos^2 \phi)  \right ]
\\
&= \frac{ \sigma  \cos^2 \phi}{ 27 t r^{3\sigma}}  \left [ -21 \cos^2 \phi - (6-\sigma)  \sin^2 \phi (\phi^{'}+3) \right. \\
&\qquad\left.+ (2-2\sigma) \cos^2 \phi \sin^2 \phi \phi^{'}  + O(\sigma^2 \cos^2 \phi)  \right ] \leq 0,
\end{aligned}
\label{apa 5}
\end{align}
when $\sigma>0$ is sufficiently small. From this we get (\ref{mc5a}). The proof of (\ref{mc5}) is similar.


\bigskip
{\bf Acknowledgments:} The first author has been partly supported
by research grants Fondecyt 1070389 and FONDAP, Chile.  The second
author has been supported by  Fondecyt grant 1050311 and Nucleus
Millennium grant P04-069-F. The research of the third author is
partially supported by an Earmarked Grant from RGC of Hong Kong and a Direct Grant from CUHK. We   wish to thank Professor L.F. Cheung, Professor E.N. Dancer, Professor C.-F. Gui, Professor N. Ghoussoub and Professor F. Pacard for useful discussions.


\end{document}